\documentclass[a4paper,12pt,reqno]{amsart}
\usepackage{amssymb,amsmath,array,amscd,amsthm,hhline}

\usepackage[mathscr]{euscript}
\usepackage{stmaryrd}
\usepackage{ulem}

\usepackage{bm}

\usepackage{dutchcal}

\makeatletter
\DeclareFontFamily{U}{tipa}{}
\DeclareFontShape{U}{tipa}{m}{n}{<->tipa10}{}
\newcommand{\arc@char}{{\usefont{U}{tipa}{m}{n}\symbol{62}}}%

\newcommand{\arc}[1]{\mathpalette\arc@arc{#1}}

\newcommand{\arc@arc}[2]{%
  \sbox0{$\m@th#1#2$}%
  \vbox{
    \hbox{\resizebox{\wd0}{\height}{\arc@char}}
    \nointerlineskip
    \box0
  }%
}
\makeatother

\usepackage{enumitem}

\usepackage{tikz}
\usetikzlibrary{fit, patterns, backgrounds,cd}

\usepackage{hyperref}

\usepackage{mathrsfs}

\def\red#1{{\color{red}{#1}}}
\def\blue#1{{\color{blue}{#1}}}
\def\green#1{{\color{green}{#1}}}

\def\({\left(}
\def\){\right)}

\numberwithin{equation}{subsection}

\let\oldsubsection\subsection
\renewcommand{\subsection}{
  \renewcommand{\theequation}{\thesubsection.\arabic{equation}}
  \oldsubsection}

\def\notsim{\not\sim}
\def\lm{\lambda}
\def\Lm{\Lambda}
\def\C{\mathscr C}

\def\H{\mathscr H}

\def\U{\mathscr U}

\def\D{\mathscr D}

\def\T{\mathscr T}

\def\P{\mathbf P}
\def\Z{\mathbb Z}
\def\R{\mathbb R}

\def\Seq{\mathbf{Expr}}

\def\pr{\mathop{\rm pr}\nolimits}
\def\ito{\stackrel\sim\to}

\def\p{{}^p\!}
\def\<{\langle}
\def\>{\rangle}
\def\sectsign{\mathhexbox278}

\renewcommand\emptyset{\varnothing}
\renewcommand\phi{\varphi}
\def\u{\underline}

\def\im{\mathop{\rm im}}

\def\dist{\mathop{\rm dist}}

\def\Gr{\mathop{\rm Gr}\nolimits}
\def\Cyc{\mathop{\rm Cyc}\nolimits}

\def\codim{\mathop{\rm codim}\nolimits}

\def\ord{\mathop{\rm ord}\nolimits}
\def\id{\mathrm{id}}
\renewcommand\epsilon{\varepsilon}

\def\suchthat{\mathbin{\rm |}}

\def\FDR{\mathbf{FDR}}

\def\Sub{\bm{\mathfrak{S}}}

\def\SubExpr{\mathop{\mathbf{Sub}}\nolimits}
\def\SubEx{\mathop{\mathbf{Sub}}\nolimits}
\def\Expr{\mathop{\mathbf{Expr}}\nolimits}

\def\a{A}
\def\b{B}
\def\c{\mathfrak{c}}

\def\E{\mathscr{E}}
\def\L{\mathscr{L}}
\def\op{\mathrm{op}}

\def\Tr{\mathrm{Tr}}
\def\Sq{\mathrm{Sq}}

\def\ddelta{\gamma}
\def\ggamma{\delta}

\newtheorem{theorem}{Theorem}[subsection]
\newtheorem{Theorem}{Theorem}
\newtheorem{proposition}[theorem]{Proposition}
\newtheorem{lemma}[theorem]{Lemma}
\newtheorem{corollary}[theorem]{Corollary}
\newtheorem{remark}[theorem]{Remark}

\newtheorem{definition}[theorem]{Defintition}
\newtheorem{example}[theorem]{Example}

\def\l{\ell}

\def\f{\mathbf{f}}

\def\v{\mathbf v}

\def\le{\leqslant}
\def\ge{\geqslant}

\renewcommand{\labelenumi}{{\rm\theenumi}}
\renewcommand{\theenumi}{{\rm(\arabic{enumi})}}

\def\={\equiv}

\def\BS{\mathrm{BS}}

\def\GL{\mathrm{GL}}

\def\p{\pi}

\def\r{\mathbf{r}}
\def\t{\mathbf{r}}

\title{Cycles in subexpression graphs}

\author{Vladimir Shchigolev}

\address{Financial University under the Government of the Russian Federation, Moscow, Russia}
\email{shchigolev\_vladimir@yahoo.com}


\begin{document}

%


\begin{abstract}
Let $\Sub(\u{s},w)$ be the graph whose vertices are all subexpressions with target $w$
of a fixed expression $\u{s}$ in generators of a Coxeter group
and edges are the pairs of subexpressions with Hamming distance 2.
We prove that $\Sub(\u{s},w)$ is connected and
its cycle space
is spanned by cycles of lengths 
$d+2$, where $d$ ranges over all positive divisors of
all finite orders of products of at most two entries of $\u{s}$.


\end{abstract}

\maketitle

\medskip
\medskip

\section{Introduction}

Let $(W,S)$ be a Coxeter system.
A finite sequence $\u{s}=(s_1,\ldots,s_n)$ with entries in $S$ is called an {\it expression}.
A {\it subexpression} of $\u{s}$ {\it with target} $w$ is a sequence $\u{\epsilon}=(\epsilon_1,\ldots,\epsilon_n)$
with entries $0$ or $1$ such that 
$s_1^{\epsilon_1}s_2^{\epsilon_2}\ldots s_n^{\epsilon_n}=w$.
Subexpressions emerge naturally, in the works of many authors on combinatorics of Coxeter groups, Kazhdan-Lusztig polynomials and
Soergel bimodules~\cite{Deodhar2},~\cite{EW},~\cite{Kumar}.

We will consider the set of subexpressions of $\u{s}$ with target $w$ as the set of vertices of an
undirected graph $\Sub(\u{s},w)$, called a {\it subexpression graph}, whose edges are 
pairs of vertices with Hamming distance (which is the number of positions at which the corresponding entries are different) two. Note that two is the minimal possible distance
between different subexpressions with the same target.
Below is an example\footnote{Here $W$ is the symmetric groups $S_4$, $\u{s}=(x_1,x_2,x_1,x_2,x_3,x_2,x_1,x_3)$,
where $x_i$ is the transposition $(i,i+1)$, and
$w=(1,2,3)$.
The three-dimensional visualization is due to Maple's procedure\linebreak {\tt DrawGraph} with an additional parameter {\tt layout=spring}.
}
of 
$\Sub(\u{s},w)$:

\medskip

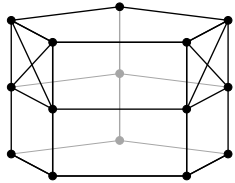
\def\aa{1.7pt}
\def\bb{0.7pt}
\def\gr{70}
\def\cc{0.17}

\begin{center}
\scalebox{0.3}{
\begin{tikzpicture}[baseline=-63pt]
\draw[gray!\gr,line width=\aa] (0., .8682408865)--(0., -5.040605632); 
\draw[line width=\aa] (4.755282582, .2683011886)--(2.938926262, -.7024216324)--(-2.938926262, -.7024216324)--(-4.755282582, .2683011886)--(0., .8682408865)--(4.755282582, .2683011886);
\draw[fill] (0., .8682408865) circle(\cc);
\draw[fill] (2.938926262, -.7024216324) circle(\cc);
\draw[fill] (-2.938926262, -.7024216324) circle(\cc);
\draw[fill] (-4.755282582, .2683011886) circle(\cc);
\draw[fill] (4.755282582, .2683011886) circle(\cc);

\draw[line width=\aa]  (4.755282582, -2.686122070)--(2.938926262, -3.656844891)--(-2.938926262, -3.656844891)--(-4.755282582, -2.686122070);
\draw[gray!\gr,line width=\bb]  (-4.755282582, -2.686122070)--(0., -2.086182372)--(4.755282582, -2.686122070);

\draw[fill] (4.755282582, -2.686122070) circle(\cc);
\draw[fill] (2.938926262, -3.656844891) circle(\cc);
\draw[fill] (-2.938926262, -3.656844891) circle(\cc);
\draw[fill] (-4.755282582, -2.686122070) circle(\cc);
\draw[gray!\gr,fill] (0., -2.086182372) circle(\cc);

\draw[line width=\aa]  (4.755282582, -5.640545330)--(2.938926262, -6.611268151)--(-2.938926262, -6.611268151)--(-4.755282582, -5.640545330);
\draw[gray!\gr,line width=\bb]  (-4.755282582, -5.640545330)--(0., -5.040605632)--(4.755282582, -5.640545330);

\draw[fill] (4.755282582, -5.640545330) circle(\cc);
\draw[fill] (2.938926262, -6.611268151) circle(\cc);
\draw[fill] (-2.938926262, -6.611268151) circle(\cc);
\draw[fill] (-4.755282582, -5.640545330) circle(\cc);
\draw[gray!\gr,fill] (0., -5.040605632) circle(\cc);

\draw[line width=\aa] (4.755282582, .2683011886)--(2.938926262, -.7024216324)--(2.938926262, -6.611268151)--(4.755282582, -5.640545330)--(4.755282582, .2683011886);
\draw[line width=\aa] (2.938926262, -.7024216324)--(-2.938926262, -.7024216324)--(-2.938926262, -6.611268151)--(2.938926262, -6.611268151)--(2.938926262, -.7024216324);
\draw[line width=\aa] (-2.938926262, -.7024216324)--(-4.755282582, .2683011886)--(-4.755282582, -5.640545330)--(-2.938926262, -6.611268151)--(-2.938926262, -.7024216324);

\draw[line width=\aa] (4.755282582, .2683011886)--(2.938926262, -3.656844891);
\draw[line width=\aa] (4.755282582, -2.686122070)--(2.938926262, -.7024216324);

\draw[line width=\aa] (-2.938926262, -.7024216324)--(-4.755282582, -2.686122070);
\draw[line width=\aa] (-2.938926262, -3.656844891)--(-4.755282582, .2683011886);

\end{tikzpicture}}
\end{center}


\medskip

\noindent
Looking at this and similar examples, we can conjecture that any graph $\Sub(\u{s},w)$ is connected
and all its cycles are sums (modulo 2) of ``short'' (see, for example, Table~\ref{table:1}) cycles.
So the main results of the paper are as follows:

\begin{Theorem}\label{theorem:connected}
Any graph $\Sub(\u{s},w)$ is connected.
\end{Theorem}

\begin{Theorem}\label{theorem:main}
The cycle space of
$\Sub(\u{s},w)$
is spanned by cycles of lengths 
$d+2$, where $d$ ranges over all positive divisors of
all finite orders of products of at most two entries of $\u{s}$.
\end{Theorem}


\noindent
Considering separately the case of type $A_1$, we get the following table:

\medskip

\begin{center}
\begin{tabular}{|l| c| c| c|c|c|c|c|}
 \hline
 root system & $A_1$ & $A_n$, $n\ge2$ & $B_2,C_2$& $B_n,C_n$, $n\ge3$ & $D_n,E_6,E_7,E_8$ &$F_4$ & $G_2$\\ [0.5ex]
 \hline
 cycle lengths& 3 & 3,4,5 & 3,4,6 &3,4,5,6 & 3,4,5& 3,4,5,6 & 3,4,5,8 \\
 \hline
\end{tabular}

\smallskip
\smallskip

{\small Table 1: Lengths of cycles in the crystallographic case}\label{table:1}
\end{center}

\smallskip

There is a certain kind of parallel between the graphs $\Sub(\u{s},w)$ and the
{\it reduced expression graphs} \cite[Section 4.2]{EW} whose vertices are reduced expressions of a given element
of $W$ and the edges are braid moves. Theorem~\ref{theorem:connected} (proved in Section~\ref{Galleries}) is thus analogous to Matsumoto's theorem~\cite{M}
and Theorem~\ref{theorem:main} is analogous to Ronan's theorem~\cite[Theorem 2.17]{Ronan} on cycles, which
were used by Elias and Williamson in~\cite{EW} to define their diagramatic categories.

The graph $\Sub(\u{s},w)$ also allows the following geometrical interpretation. Let $\Sigma=\Sigma(W,S)$
be the Coxeter complex. This is a simplicial complex endowed with a left action of $W$.
We will denote by $\Sigma_k$ the union of simplices of codimension $k$.
An {\it alcove walk}~\cite[Section~3.1]{Ram} is a continuous path $g:[a,b]\to\Sigma_0\setminus\Sigma_2$
such that $g(a),g(b)\notin\Sigma_1$ and $g^{-1}(\Sigma_1)$ is finite.
This path uniquely determines an alternating sequence of chambers and panels $\u{\gamma}$
called a {\it gallery}~\cite{CC},~\cite{G},~\cite{MST}. The {\it type} of $\u{\gamma}$ is the sequence $\u{s}$ of
the types of its panels and the {\it signature} of $\u{\gamma}$ is a subexpression $\u{\epsilon}\subset\u{s}$
such that $\epsilon_i=0$ if $\u{\gamma}$ has a fold at the $i$th panel and $\epsilon_i=1$ if it has a crossing.
Thus the set of all subexpressions of $\u{s}$ corresponds bijectively
to the set of all galleries of type $\u{s}$ beginning at a fixed chamber~\cite[Lemma 4.8]{GSch}.

A gallery $\u{\gamma}$ can be folded at (starting from) the $i$th panel~\cite[Section 3.3]{MST}
and the resulting gallery $\f_i\u{\gamma}$ has the same type as $\u{\gamma}$,
whereas its signature is different from the signature of $\u{\gamma}$ at the $i$th place.
If there is a different index $j$ such that the $i$th and $j$th panels of $\u{\gamma}$ belong to the same wall,
then (and only then) the galleries $\u{\gamma}$ and $\f_{i,j}\u{\gamma}=\f_i\f_j\u{\gamma}$ begin and end at the same chambers
(Section~\ref{Expressions_and_subexpressions}).

We can identify the vertices of $\Sub(\u{s},w)$ with all galleries $\u{\gamma}$
of type $\u{s}$ beginning at the fundamental chamber $C$ and ending at $wC$.
Theorem~\ref{theorem:connected} posits that whenever we are given two such galleries
$\u{\gamma}$ and $\u{\gamma}'$, there exists a sequence $\u{\gamma}=\u{\gamma_0},\u{\gamma_1},\ldots,\u{\gamma_n}=\u{\gamma}'$,
called a {\it double folding sequence}, such that $\u{\gamma_k}=\f_{j_k,l_k}\u{\gamma_{k-1}}$ for any $k=1,\ldots,n$.
Note that the indices $j_k$ and $l_k$ must be chosen so that the {\it double folding operators} $\f_{j_k,l_k}$
be applicable. Theorem~\ref{theorem:main} in its turn posits that any {\it cyclic} double folding sequence
(that is, such that $\u{\gamma}=\u{\gamma}'$) is a sum of cyclic double folding sequences of predicted lengths.
Let us consider the following example for the affine group of type $\tilde A_2$
with Coxeter system $\{\red{a},\green{b},\blue{c}\}$:

{

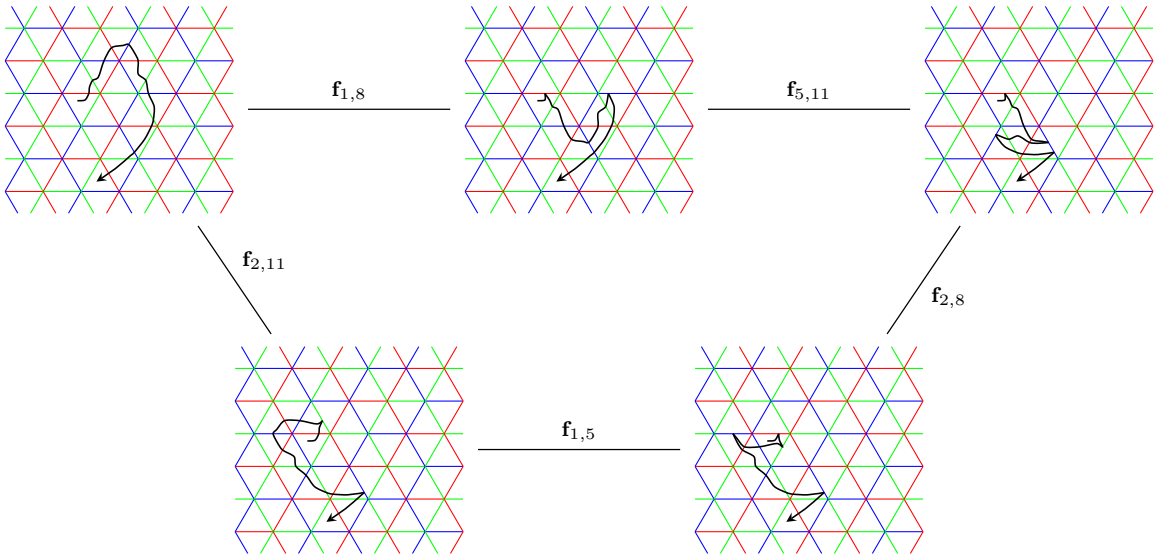
\def\L{0.5}
\def\st{0.6pt}

\def\sq{0.8660254040}
\def\p{(0,0)}
\def\pr{(\L,0)}
\def\prr{(2*\L,0)}
\def\prrr{(3*\L,0)}
\def\prrrr{(4*\L,0)}
\def\prrrrr{(5*\L,0)}
\def\prrrrrr{(6*\L,0)}

\def\plu{(0,\sq*\L)}
\def\pu{(\L/2,\sq*\L)}
\def\pru{(\L/2+\L,\sq*\L)}
\def\prru{(\L/2+2*\L,\sq*\L)}
\def\prrru{(\L/2+3*\L,\sq*\L)}
\def\prrrru{(\L/2+4*\L,\sq*\L)}
\def\prrrrru{(\L/2+5*\L,\sq*\L)}
\def\prrrrrru{(6*\L,\sq*\L)}

\def\pluu{(0,2*\sq*\L)}
\def\puu{(\L,2*\sq*\L)}
\def\pruu{(2*\L,2*\sq*\L)}
\def\prruu{(3*\L,2*\sq*\L)}
\def\prrruu{(4*\L,2*\sq*\L)}
\def\prrrruu{(5*\L,2*\sq*\L)}
\def\prrrrruu{(6*\L,2*\sq*\L)}

\def\plluuu{(0,3*\sq*\L)}
\def\pluuu{(\L/2,3*\sq*\L)}
\def\puuu{(3*\L/2,3*\sq*\L)}
\def\pruuu{(3*\L/2+\L,3*\sq*\L)}
\def\prruuu{(3*\L/2+2*\L,3*\sq*\L)}
\def\prrruuu{(3*\L/2+3*\L,3*\sq*\L)}
\def\prrrruuu{(3*\L/2+4*\L,3*\sq*\L)}
\def\prrrrruuu{(3*\L/2+4*\L+\L/2,3*\sq*\L)}

\def\plluuuu{(0,4*\sq*\L)}
\def\pluuuu{(\L,4*\sq*\L)}
\def\puuuu{(2*\L,4*\sq*\L)}
\def\pruuuu{(3*\L,4*\sq*\L)}
\def\prruuuu{(4*\L,4*\sq*\L)}
\def\prrruuuu{(5*\L,4*\sq*\L)}
\def\prrrruuuu{(6*\L,4*\sq*\L)}

\def\pllluuuuu{(0,5*\sq*\L)}
\def\plluuuuu{(\L/2,5*\sq*\L)}
\def\pluuuuu{(3*\L/2,5*\sq*\L)}
\def\puuuuu{(5*\L/2,5*\sq*\L)}
\def\pruuuuu{(3*\L/2+2*\L,5*\sq*\L)}
\def\prruuuuu{(3*\L/2+3*\L,5*\sq*\L)}
\def\prrruuuuu{(3*\L/2+4*\L,5*\sq*\L)}
\def\prrrruuuuu{(6*\L,5*\sq*\L)}

\def\pllluuuuuu{(\L/2-\L/3,5*\sq*\L+2*\sq*\L/3)}
\def\plluuuuuu{(\L+\L/2-\L/3,5*\sq*\L+2*\sq*\L/3)}
\def\plluuuuuU{(\L+\L/2-2*\L/3,5*\sq*\L+2*\sq*\L/3)}
\def\pluuuuuu{(2*\L+\L/2-\L/3,5*\sq*\L+2*\sq*\L/3)}
\def\pluuuuuU{(2*\L+\L/2-2*\L/3,5*\sq*\L+2*\sq*\L/3)}
\def\puuuuuu{(3*\L+\L/2-\L/3,5*\sq*\L+2*\sq*\L/3)}
\def\puuuuuU{(3*\L+\L/2-2*\L/3,5*\sq*\L+2*\sq*\L/3)}
\def\pruuuuuu{(4*\L+\L/2-\L/3,5*\sq*\L+2*\sq*\L/3)}
\def\pruuuuuU{(4*\L+\L/2-2*\L/3,5*\sq*\L+2*\sq*\L/3)}
\def\prruuuuuu{(5*\L+\L/2-\L/3,5*\sq*\L+2*\sq*\L/3)}
\def\prruuuuuU{(5*\L+\L/2-2*\L/3,5*\sq*\L+2*\sq*\L/3)}
\def\prrruuuuuU{(6*\L+\L/2-2*\L/3,5*\sq*\L+2*\sq*\L/3)}

\def\pd{(\L/3,-2/3*\sq*\L)}
\def\pD{(2*\L/3,-2/3*\sq*\L)}
\def\prd{(\L/3+\L,-2/3*\sq*\L)}
\def\prD{(2*\L/3+\L,-2/3*\sq*\L)}
\def\prrd{(\L/3+2*\L,-2/3*\sq*\L)}
\def\prrD{(2*\L/3+2*\L,-2/3*\sq*\L)}
\def\prrrd{(\L/3+3*\L,-2/3*\sq*\L)}
\def\prrrD{(2*\L/3+3*\L,-2/3*\sq*\L)}
\def\prrrrd{(\L/3+4*\L,-2/3*\sq*\L)}
\def\prrrrD{(2*\L/3+4*\L,-2/3*\sq*\L)}
\def\prrrrrd{(\L/3+5*\L,-2/3*\sq*\L)}
\def\prrrrrD{(2*\L/3+5*\L,-2/3*\sq*\L)}

\def\U{

\draw[red] \p--\pr;
\draw[green] \pr--\prr;
\draw[blue] \p--\pu;
\draw[green] \pu--\pr;
\draw[red] \pr--\pru;
\draw[blue] \pu--\pru;

\draw[blue] \pru--\prr;
\draw[blue] \prr--\prrr;
\draw[green] \prr--\prru;
\draw[red] \pru--\prru;
\draw[red] \prru--\prrr;

\draw[green] \pu--\plu;
\draw[red] \prrr--\prrrr;
\draw[green] \prrrr--\prrrrr;
\draw[blue] \prrrrr--\prrrrrr;
\draw[green] \prru--\prrru;
\draw[blue] \prrr--\prrru;
\draw[green] \prrru--\prrrr;
\draw[blue] \prrru--\prrrru;
\draw[red] \prrrr--\prrrru;
\draw[blue] \prrrrr--\prrrru;
\draw[red] \prrrru--\prrrrru;
\draw[green] \prrrrr--\prrrrru;
\draw[red] \prrrrrr--\prrrrru;
\draw[green] \prrrrru--\prrrrrru;

\draw[blue] \pu--\pluu;
\draw[red] \puu--\pru;
\draw[green] \pu--\puu;
\draw[red] \puu--\pluu;
\draw[green] \puu--\pruu;
\draw[blue] \pruu--\prruu;
\draw[red] \prruu--\prrruu;
\draw[green] \prrruu--\prrrruu;
\draw[blue] \prrrruu--\prrrrruu;
\draw[blue] \pruu--\pru;
\draw[green] \pruu--\prru;
\draw[red] \prruu--\prru;
\draw[blue] \prruu--\prrru;
\draw[green] \prrruu--\prrru;
\draw[red] \prrruu--\prrrru;
\draw[blue] \prrrruu--\prrrru;
\draw[green] \prrrruu--\prrrrru;
\draw[red] \prrrrruu--\prrrrru;

\draw[green] \puu--\pluuu;
\draw[blue] \puuu--\pruu;
\draw[red] \puu--\puuu;
\draw[blue] \puuu--\pluuu;
\draw[red] \puuu--\pruuu;
\draw[green] \pruuu--\prruuu;
\draw[blue] \prruuu--\prrruuu;
\draw[red] \prrruuu--\prrrruuu;
\draw[green] \prrrruuu--\prrrrruuu;
\draw[green] \pruuu--\pruu;
\draw[red] \pruuu--\prruu;
\draw[blue] \prruuu--\prruu;
\draw[green] \prruuu--\prrruu;
\draw[red] \prrruuu--\prrruu;
\draw[blue] \prrruuu--\prrrruu;
\draw[green] \prrrruuu--\prrrruu;
\draw[red] \prrrruuu--\prrrrruu;
\draw[blue] \pluu--\pluuu;
\draw[green] \plluuu--\pluuu;

\draw[red] \puuu--\pluuuu;
\draw[green] \puuuu--\pruuu;
\draw[blue] \puuu--\puuuu;
\draw[green] \puuuu--\pluuuu;
\draw[blue] \puuuu--\pruuuu;
\draw[red] \pruuuu--\prruuuu;
\draw[green] \prruuuu--\prrruuuu;
\draw[blue] \prrruuuu--\prrrruuuu;
\draw[red] \pruuuu--\pruuu;
\draw[blue] \pruuuu--\prruuu;
\draw[green] \prruuuu--\prruuu;
\draw[red] \prruuuu--\prrruuu;
\draw[blue] \prrruuuu--\prrruuu;
\draw[green] \prrruuuu--\prrrruuu;
\draw[red] \prrrruuuu--\prrrruuu;
\draw[green] \pluuu--\pluuuu;
\draw[red] \plluuuu--\pluuuu;
\draw[blue] \plluuuu--\pluuu;

\draw[blue] \puuuu--\pluuuuu;
\draw[red] \puuuuu--\pruuuu;
\draw[green] \puuuu--\puuuuu;
\draw[red] \puuuuu--\pluuuuu;
\draw[green] \puuuuu--\pruuuuu;
\draw[blue] \pruuuuu--\prruuuuu;
\draw[red] \prruuuuu--\prrruuuuu;
\draw[green] \prrruuuuu--\prrrruuuuu;
\draw[blue] \pruuuuu--\pruuuu;
\draw[green] \pruuuuu--\prruuuu;
\draw[red] \prruuuuu--\prruuuu;
\draw[blue] \prruuuuu--\prrruuuu;

\draw[green] \prrruuuuu--\prrruuuu;
\draw[red] \prrruuuuu--\prrrruuuu;
\draw[red] \pluuuu--\pluuuuu;
\draw[blue] \plluuuuu--\pluuuuu;
\draw[green] \plluuuuu--\pluuuu;
\draw[blue] \plluuuuu--\plluuuu;
\draw[green] \plluuuuu--\pllluuuuu;

\draw[blue] \plluuuuu--\pllluuuuuu;
\draw[red] \pluuuuu--\plluuuuuu;
\draw[green] \plluuuuu--\plluuuuuU;
\draw[green] \puuuuu--\pluuuuuu;
\draw[blue] \pluuuuu--\pluuuuuU;
\draw[blue] \pruuuuu--\puuuuuu;
\draw[red] \puuuuu--\puuuuuU;
\draw[red] \prruuuuu--\pruuuuuu;
\draw[green] \pruuuuu--\pruuuuuU;
\draw[green] \prrruuuuu--\prruuuuuu;
\draw[blue] \prruuuuu--\prruuuuuU;
\draw[red] \prrruuuuu--\prrruuuuuU;

\draw[blue] \p--\pd;
\draw[green] \pr--\pD;
\draw[red] \pr--\prd;
\draw[blue] \prr--\prD;
\draw[green] \prr--\prrd;
\draw[red] \prrr--\prrD;
\draw[blue] \prrr--\prrrd;
\draw[green] \prrrr--\prrrD;
\draw[red] \prrrr--\prrrrd;
\draw[blue] \prrrrr--\prrrrD;
\draw[green] \prrrrr--\prrrrrd;
\draw[red] \prrrrrr--\prrrrrD;
}

\def\xc{1.9*\L}\def\yc{2.78*\sq*\L}
\def\xd{2.1*\L} \def\yd{2.8*\sq*\L}
\def\xe{2.2*\L} \def\ye{3*\sq*\L}
\def\xf{2.22*\L}\def\yf{3.2*\sq*\L}
\def\xg{2.3*\L} \def\yg{3.4*\sq*\L}
\def\xh{2.45*\L}\def\yh{3.5*\sq*\L}
\def\xi{2.6*\L} \def\yi{4*\sq*\L}
\def\xj{2.8*\L} \def\yj{4.4*\sq*\L}
\def\xk{3.1*\L} \def\yk{4.47*\sq*\L}
\def\xl{3.25*\L} \def\yl{4.5*\sq*\L}
\def\xm{3.4*\L}\def\ym{4.25*\sq*\L}
\def\xn{3.5*\L} \def\yn{4*\sq*\L}
\def\xo{3.5*\L} \def\yo{3.6*\sq*\L}
\def\xp{3.7*\L} \def\yp{3.4*\sq*\L}
\def\xr{3.77*\L}\def\yr{3*\sq*\L}
\def\xs{3.92*\L}\def\ys{2.6*\sq*\L}
\def\xt{3.9*\L} \def\yt{2.2*\sq*\L}
\def\xu{3.85*\L}\def\yu{2*\sq*\L}
\def\xv{3.67*\L}\def\yv{1.6*\sq*\L}
\def\xw{3.4*\L} \def\yw{1.2*\sq*\L}
\def\xz{3.3*\L} \def\yz{1.1*\sq*\L}


\def\xT{2.9*\L} \def\yT{0.7*\sq*\L}
\def\xTT{2.4*\L} \def\yTT{0.3*\sq*\L}

\def\V{
\begin{tikzpicture}[baseline=28pt]
\U
\draw[->,>=stealth,black,line width=\st] plot[smooth] coordinates{
                                                     (\xc,\yc)
                                                     (\xd,\yd)
                                                     (\xe,\ye) (\xf,\yf) (\xg,\yg)
                                                     (\xh,\yh) (\xi,\yi) (\xj,\yj)
                                                     (\xk,\yk) (\xl,\yl) (\xm,\ym) (\xn,\yn)
                                                     (\xo,\yo) (\xp,\yp) (\xr,\yr) (\xs,\ys)
                                                     (\xt,\yt) (\xu,\yu)
                                                     (\xv,\yv) (\xw,\yw) (\xz,\yz)
                                                     (\xT,\yT) (\xTT,\yTT)
                                                     };
\end{tikzpicture}
}

\def\xxa{\xa}\def\yya{\ya}
\def\xxb{\xb}\def\yyb{\yb}
\def\xxc{\xc}\def\yyc{\yc}
\def\xxd{\xd}\def\yyd{\yd}
\def\xxe{\xd}\def\yye{\ye}
\def\xxf{\xf}\def\yyf{6*\sq*\L-\yf}
\def\xxg{\xg}\def\yyg{6*\sq*\L-\yg}
\def\xxh{\xh}\def\yyh{6*\sq*\L-\yh}
\def\xxi{\xi}\def\yyi{6*\sq*\L-\yi}
\def\xxj{\xj}\def\yyj{6*\sq*\L-\yj}
\def\xxk{\xk}\def\yyk{6*\sq*\L-\yk}
\def\xxl{\xl}\def\yyl{6*\sq*\L-\yl}
\def\xxm{\xm}\def\yym{6*\sq*\L-\ym}
\def\xxn{\xn}\def\yyn{6*\sq*\L-\yn}
\def\xxo{\xo}\def\yyo{6*\sq*\L-\yo}
\def\xxp{\xp}\def\yyp{6*\sq*\L-\yp}
\def\xxr{\xr}\def\yyr{6*\sq*\L-\yr}
\def\xxs{\xs}\def\yys{\ys}
\def\xxt{\xt}\def\yyt{\yt}
\def\xxu{\xu}\def\yyu{\yu}
\def\xxv{\xv}\def\yyv{\yv}
\def\xxw{\xw}\def\yyw{\yw}
\def\xxz{\xz}\def\yyz{\yz}

\def\VV{
\begin{tikzpicture}[baseline=28pt]
\U
\draw[black,line width=\st] plot[smooth] coordinates{
                                                     (\xc,\yc)
                                                     (\xxd,\yyd)
                                                     (\xxe,\yye) };

\draw[black,line width=\st] plot[smooth] coordinates{ (\xxe,\yye)
                                                     (\xxf,\yyf) (\xxg,\yyg)
                                                     (\xxh,\yyh) (\xxi,\yyi) (\xxj,\yyj)
                                                     (\xxk,\yyk) (\xxl,\yyl) (\xxm,\yym)
                                                     (\xxn,\yyn) (\xxo,\yyo) (\xxp,\yyp)
                                                     (\xxr,\yyr)
                                                     };
\draw[->,>=stealth,black,line width=\st] plot[smooth] coordinates{ (\xxr,\yyr)
                                                     (\xxs,\yys) (\xxt,\yyt)
                                                     (\xxu,\yyu) (\xxv,\yyv) (\xxw,\yyw) (\xxz,\yyz)
                                                     (\xT,\yT) (\xTT,\yTT)};
\end{tikzpicture}
}

\def\xxxa{\xa}\def\yyya{\ya}
\def\xxxb{\xb}\def\yyyb{\yb}
\def\xxxc{\xc}\def\yyyc{\yc}
\def\xxxd{\xd}\def\yyyd{\yd}
\def\xxxe{\xe}\def\yyye{\ye}
\def\xxxf{\xf}\def\yyyf{\yf}
\def\xxxg{\xg}\def\yyyg{\yg}

\def\xxxh{-\xh*3/4+\xh/4-\sq*\yh+8*3/4*\L}
\def\yyyh{3/4*\yh-\yh/4+4*\sq*\L-\sq*\xh}

\def\xxxi{-\xi*3/4+\xi/4-\sq*\yi+8*3/4*\L}
\def\yyyi{3/4*\yi-\yi/4+4*\sq*\L-\sq*\xi}

\def\xxxj{-\xj*3/4+\xj/4-\sq*\yj+8*3/4*\L}
\def\yyyj{3/4*\yj-\yj/4+4*\sq*\L-\sq*\xj}

\def\xxxk{-\xk*3/4+\xk/4-\sq*\yk+8*3/4*\L}
\def\yyyk{3/4*\yk-\yk/4+4*\sq*\L-\sq*\xk}

\def\xxxl{-\xl*3/4+\xl/4-\sq*\yl+8*3/4*\L}
\def\yyyl{3/4*\yl-\yl/4+4*\sq*\L-\sq*\xl}

\def\xxxm{-\xm*3/4+\xm/4-\sq*\ym+8*3/4*\L}
\def\yyym{3/4*\ym-\ym/4+4*\sq*\L-\sq*\xm}

\def\xxxn{-\xn*3/4+\xn/4-\sq*\yn+8*3/4*\L}
\def\yyyn{3/4*\yn-\yn/4+4*\sq*\L-\sq*\xn}

\def\xxxo{-\xo*3/4+\xo/4-\sq*\yo+8*3/4*\L}
\def\yyyo{3/4*\yo-\yo/4+4*\sq*\L-\sq*\xo}

\def\xxxp{-\xp*3/4+\xp/4-\sq*\yp+8*3/4*\L}
\def\yyyp{3/4*\yp-\yp/4+4*\sq*\L-\sq*\xp}

\def\xxxr{-\xr*3/4+\xr/4-\sq*\yr+8*3/4*\L}
\def\yyyr{3/4*\yr-\yr/4+4*\sq*\L-\sq*\xr}

\def\xxxs{-\xs*3/4+\xs/4-\sq*\ys+8*3/4*\L}
\def\yyys{3/4*\ys-\ys/4+4*\sq*\L-\sq*\xs}

\def\xxxt{-\xt*3/4+\xt/4-\sq*\yt+8*3/4*\L}
\def\yyyt{3/4*\yt-\yt/4+4*\sq*\L-\sq*\xt}

\def\xxxu{-\xu*3/4+\xu/4-\sq*\yu+8*3/4*\L}
\def\yyyu{3/4*\yu-\yu/4+4*\sq*\L-\sq*\xu}

\def\xxxv{-\xv*3/4+\xv/4-\sq*\yv+8*3/4*\L}
\def\yyyv{3/4*\yv-\yv/4+4*\sq*\L-\sq*\xv}

\def\xxxw{-\xw*3/4+\xw/4-\sq*\yw+8*3/4*\L}
\def\yyyw{3/4*\yw-\yw/4+4*\sq*\L-\sq*\xw}

\def\xxxz{\xz}\def\yyyz{\yz}

\def\VVV{
\begin{tikzpicture}[baseline=28pt]
\U
\draw[black,line width=\st] plot[smooth] coordinates{
                                                     (\xc,\yc)
                                                     (\xxxd,\yyyd)
                                                     (\xxxe,\yyye) (\xxxf,\yyyf) (\xxxg,\yyyg)};

\draw[black,line width=\st] plot[smooth] coordinates{(\xxxg,\yyyg)
                                                     (\xxxh,\yyyh) (\xxxi,\yyyi) (\xxxj,\yyyj) (\xxxk,\yyyk) (\xxxl,\yyyl)
                                                     (\xxxm,\yyym) (\xxxn,\yyyn) (\xxxo,\yyyo) (\xxxp,\yyyp)
                                                     (\xxxr,\yyyr) (\xxxs,\yyys) (\xxxt,\yyyt) (\xxxu,\yyyu)
                                                     (\xxxv,\yyyv) (\xxxw,\yyyw)};

\draw[->,>=stealth,black,line width=\st] plot[smooth] coordinates{(\xxxw,\yyyw)
                                                     (\xxxz,\yyyz)
                                                     (\xT,\yT) (\xTT,\yTT)};
\end{tikzpicture}
}

\def\xxxxd{\xxxd}\def\yyyyd{\yyyd}
\def\xxxxe{\xxxe}\def\yyyye{\yyye}
\def\xxxxf{\xxxf}\def\yyyyf{6*\sq*\L-\yyyf}
\def\xxxxg{\xxxg}\def\yyyyg{6*\sq*\L-\yyyg}

\def\xxxxh{\xxxh}\def\yyyyh{6*\sq*\L-3/4*\yh+\yh/4-4*\sq*\L+\sq*\xh}
\def\xxxxi{\xxxi}\def\yyyyi{6*\sq*\L-3/4*\yi+\yi/4-4*\sq*\L+\sq*\xi}
\def\xxxxj{\xxxj}\def\yyyyj{6*\sq*\L-3/4*\yj+\yj/4-4*\sq*\L+\sq*\xj}
\def\xxxxk{\xxxk}\def\yyyyk{6*\sq*\L-3/4*\yk+\yk/4-4*\sq*\L+\sq*\xk}
\def\xxxxl{\xxxl}\def\yyyyl{6*\sq*\L-3/4*\yl+\yl/4-4*\sq*\L+\sq*\xl}

\def\xxxxm{\xxxm}\def\yyyym{\yyym}
\def\xxxxn{\xxxn}\def\yyyyn{\yyyn}
\def\xxxxo{\xxxo}\def\yyyyo{\yyyo}
\def\xxxxp{\xxxp}\def\yyyyp{\yyyp}
\def\xxxxr{\xxxr}\def\yyyyr{\yyyr}
\def\xxxxs{\xxxs}\def\yyyys{\yyys}
\def\xxxxt{\xxxt}\def\yyyyt{\yyyt}
\def\xxxxu{\xxxu}\def\yyyyu{\yyyu}
\def\xxxxv{\xxxv}\def\yyyyv{\yyyv}
\def\xxxxw{\xxxw}\def\yyyyw{\yyyw}
\def\xxxxz{\xxxz}\def\yyyyz{\yyyz}

\def\VVVV{
\begin{tikzpicture}[baseline=28pt]
\U
\draw[black,line width=\st] plot[smooth] coordinates{
(\xc,\yc) (\xxxxd,\yyyyd) (\xxxxe,\yyyye) };

\draw[black,line width=\st] plot[smooth] coordinates{(\xxxxe,\yyyye)
(\xxxxf,\yyyyf) (\xxxxg,\yyyyg)
(\xxxxh,\yyyyh) (\xxxxi,\yyyyi) (\xxxxj,\yyyyj) (\xxxxk,\yyyyk) (\xxxxl,\yyyyl)};

\draw[black,line width=\st] plot[smooth] coordinates{(\xxxxl,\yyyyl)
(\xxxxm,\yyyym) (\xxxxn,\yyyyn)
(\xxxxo,\yyyyo) (\xxxxp,\yyyyp) (\xxxxr,\yyyyr) (\xxxxs,\yyyys) (\xxxxt,\yyyyt) (\xxxxu,\yyyyu)
(\xxxxv,\yyyyv) (\xxxxw,\yyyyw) };

\draw[->,>=stealth,black,line width=\st] plot[smooth] coordinates{(\xxxxw,\yyyyw)
(\xxxxz,\yyyyz)
(\xT,\yT) (\xTT,\yTT) };
\end{tikzpicture}
}

\def\xxxxxc{\xxc}\def\yyyyyc{\yyc}
\def\xxxxxd{\xxd}\def\yyyyyd{\yyd}
\def\xxxxxe{\xxe}\def\yyyyye{\yye}
\def\xxxxxf{\xxf}\def\yyyyyf{\yyf}
\def\xxxxxg{\xxg}\def\yyyyyg{\yyg}
\def\xxxxxh{\xxh}\def\yyyyyh{\yyh}
\def\xxxxxi{\xxi}\def\yyyyyi{\yyi}
\def\xxxxxj{\xxj}\def\yyyyyj{\yyj}
\def\xxxxxk{\xxk}\def\yyyyyk{\yyk}
\def\xxxxxl{\xxl}\def\yyyyyl{\yyl}

\def\xxxxxm{-\xxm/2-6*3/4*\L+\sq*\ym+6*\L}
\def\yyyyym{7*\sq*\L-\ym/2-\sq*\xxm}

\def\xxxxxn{-\xxn/2-6*3/4*\L+\sq*\yn+6*\L}
\def\yyyyyn{7*\sq*\L-\yn/2-\sq*\xxn}

\def\xxxxxo{-\xxo/2-6*3/4*\L+\sq*\yo+6*\L}
\def\yyyyyo{7*\sq*\L-\yo/2-\sq*\xxo}

\def\xxxxxp{-\xxp/2-6*3/4*\L+\sq*\yp+6*\L}
\def\yyyyyp{7*\sq*\L-\yp/2-\sq*\xxp}

\def\xxxxxr{-\xxr/2-6*3/4*\L+\sq*\yr+6*\L}
\def\yyyyyr{7*\sq*\L-\yr/2-\sq*\xxr}

\def\xxxxxs{-\xxs*3/4+\xxs/4-\sq*\yys+8*3/4*\L}
\def\yyyyys{3/4*\yys-\yys/4+4*\sq*\L-\sq*\xxs}

\def\xxxxxt{-\xxt*3/4+\xxt/4-\sq*\yyt+8*3/4*\L}
\def\yyyyyt{3/4*\yyt-\yyt/4+4*\sq*\L-\sq*\xxt}

\def\xxxxxu{-\xxu*3/4+\xxu/4-\sq*\yyu+8*3/4*\L}
\def\yyyyyu{3/4*\yyu-\yyu/4+4*\sq*\L-\sq*\xxu}

\def\xxxxxv{-\xxv*3/4+\xxv/4-\sq*\yyv+8*3/4*\L}
\def\yyyyyv{3/4*\yyv-\yyv/4+4*\sq*\L-\sq*\xxv}

\def\xxxxxw{-\xxw*3/4+\xxw/4-\sq*\yyw+8*3/4*\L}
\def\yyyyyw{3/4*\yyw-\yyw/4+4*\sq*\L-\sq*\xxw}

\def\xxxxxz{\xxz}
\def\yyyyyz{\yyz}

\def\VVVVV{
\begin{tikzpicture}[baseline=28pt]
\U
\draw[black,line width=\st] plot[smooth] coordinates{
(\xc,\yc) (\xxxxxd,\yyyyyd) (\xxxxxe,\yyyyye) };

\draw[black,line width=\st] plot[smooth] coordinates{(\xxxxxe,\yyyyye)
(\xxxxxf,\yyyyyf) (\xxxxxg,\yyyyyg)
(\xxxxxh,\yyyyyh) (\xxxxxi,\yyyyyi) (\xxxxxj,\yyyyyj) (\xxxxxk,\yyyyyk) (\xxxxxl,\yyyyyl)};

\draw[black,line width=\st] plot[smooth] coordinates{
(\xxxxxl,\yyyyyl)
(\xxxxxm,\yyyyym) (\xxxxxn,\yyyyyn) (\xxxxxo,\yyyyyo)
(\xxxxxp,\yyyyyp) (\xxxxxr,\yyyyyr)
(\xxxxxs,\yyyyys) (\xxxxxt,\yyyyyt)
(\xxxxxu,\yyyyyu) (\xxxxxv,\yyyyyv)
(\xxxxxw,\yyyyyw) };

\draw[->,>=stealth,black,line width=\st] plot[smooth] coordinates{(\xxxxxw,\yyyyyw)
(\xxxxxz,\yyyyyz)
(\xT,\yT) (\xTT,\yTT)       };
\end{tikzpicture}
}

\smallskip

\begin{center}
\begin{tikzcd}[column sep=-10pt]
\V\arrow[dash]{rr}{\f_{1,8}}\arrow[dash]{rd}{\f_{2,11}}&&\VV\arrow[dash]{rr}{\f_{5,11}}&&\VVVVV\arrow[dash]{ld}{\f_{2,8}}\\[20pt]
&\VVV\arrow[dash]{rr}{\f_{1,5}}&&\VVVV
\end{tikzcd}
\end{center}
}
\noindent
In this picture, the alcove walk represented by the black line (and thus the corresponding gallery)
is folded so that its endpoints do not move.
The corresponding five galleries have the same type
$\u{s}=(\red{a},\green{b},\blue{c},\red{a},\blue{c},\red{a},\green{b},\blue{c},\green{b},\red{a},\blue{c},\green{b},\red{a})$.
So there is the following cycle in the graph $\Sub(\u{s},\blue{c}\green{b}\blue{c}\red{s}\green{b})$:

\begin{center}
\tiny
\begin{tikzcd}[column sep=-30pt]
(1,1,1,1,1,1,1,1,1,1,1,1,1)\arrow[dash]{rr}\arrow[dash]{rd}&&(0,1,1,1,1,1,1,0,1,1,1,1,1)\arrow[dash]{rr}&&(0,1,1,1,0,1,1,0,1,1,0,1,1)\arrow[dash]{ld}\\
&(1,0,1,1,1,1,1,1,1,1,0,1,1)\arrow[dash]{rr}&&(0,0,1,1,0,1,1,1,1,1,0,1,1)
\end{tikzcd}
\end{center}

\smallskip

The vertices of $\Sub(\u{s},w)$ can also be thought of as the points of the fibre $f^{-1}(wB)$
fixed by the maximal torus $T$, where $f:\BS(\u{s})\to G/B$ is a Bott-Samelson resolution,
$G$ is a reductive algebraic group and $B$ is its Borel subgroup containing $T$.
%
Connecting these points by $T$-curves, we get a geometrical interpretation of the edges of $\Sub(\u{s},w)$
as in~\cite[Example A.6]{BW}, where the graph $\Sub(\u{s},w)$ is a pentagon similar to the one in our example above.
Note that this construction was used in this paper to establish a 2-torsion, that is,
the failure of the decomposition theorem over fields of characteristic 2.
So it would be interesting to investigate if cycles in other graphs $\Sub(\u{s},w)$ contribute to
the emergence of torsion for other Bott-Samelson resolutions

Actually we obtained more than just the lengths of cycles. Theorem~\ref{theorem:4} exactly
describes what these cycles are. First we have triangles and squares (Section~\ref{Triangles_and_squares}).
These are in a sense {\it inessential} cycles whose structure and lengths do not depend on the Coxeter group.
Additionally, there are more complicated cycles.
To describe them, we make use of a categorical language (Section~\ref{Categories}) similar
to that described in~\cite{cat}. It is a useful tool that allows us to replace longer subexpressions
with shorter (and simpler) ones. 
In this part of the paper, we refer to the standard facts of the theory of reflection subgroups
due to Deodhar, Dyer and Bonnaf\'e~\cite{Deodhar1},~\cite{Dyer1},~\cite{Dyer2},~\cite{BD}.
Our approach is also close to that of Vinberg~\cite{Vinberg}.

\def\rred#1{}

\rred{
Moreover, we can also replace the initial ``big'' Coxeter group with a dihedral group
and then consider subexpressions of expressions in just two symbols.
Here we develop some ideas of~\cite[Chapter 4]{BB} about pairs of two reflections or roots in the
language of sequences of vectors, which we call here {\it Fibonacci vectors} (Section~\ref{Fibonacci_vectors}).}

In Section~\ref{Dihedral_cycles}, we describe certain cycles for the dihedral groups that later are mapped to cycles
in the initial graph $\Sub(\u{s},w)$ by morphisms of categories $\Expr_\D(W,S)$ (Section~\ref{Categories}).

Theorem~\ref{theorem:main} is proved in Section~\ref{Generating_cycles}. It follows trivially form Theorem~\ref{theorem:4}
by counting the lengths of cycles (Remark~\ref{rem:3}) and forgetting about the order.
As for Theorem~\ref{theorem:4} itself, its proof proceeds by induction on the maximal vertex $\u{\gamma}$.
Here we have a kind of double induction, as before reducing $\u{\gamma}$, we move edges by a process described
in Definition~\ref{definition:cycb6:2} to a more suitable position (Lemma~\ref{lemma:moving}).

\section{Basics}

\subsection{Sets, maps, sequences, groups, spaces}
We denote by $\Z$ and $\R$ the sets of integers and real numbers respectively.
The cardinality and the power set of a set $X$ are denoted by $|X|$ and $\P(X)$ respectively.
Applying maps to elements, we will often drop brackets. For example, we will write $\phi x$ instead of $\phi(x)$.
For an function $f:X\to Y$ and a subset $Z\subset X$, we denote by $f|_Z$ the restriction of $f$ to $Z$.

The neutral element of any group $G$ is denoted by $1$. For any subset $X\subset G$, let $\<X\>$ be the subgroup
generated by $X$. The order of an element $x\in G$ will be denoted by $\ord(x)$.


For real numbers $a<b$, we denote $[a,b]=\{x\in\R\suchthat a\le x\le b\}$ and $(a,b)=\{x\in\R\suchthat a<x<b\}$.
A continuous map $f:[a,b]\to X$, where $X$ is a topological space, is called a {\it path}.
If $X=\R$, then we consider the {\it graph} of $f$ that is the following subset of the plane:
$$
\Gr(f)=\{(x,y)\in[a,b]\times\R\suchthat f(x)=y\}.
$$
Here the first coordinate $x$ is the {\it abscissa} (horizontal) and the second coordinate $y$ is the {\it ordinate} (vertical).
A function $f:X\to Y$, where $X$ and $Y$ are subsets of $\R$, is called {\it increasing} if
it preserves the natural order on $\R$.

Let $f_1:Z_1\to Y,\ldots f_n:Z_n\to Y$ be maps, where $Z_1,\ldots,Z_n\subset X$,
such that any two maps $f_i$ and $f_j$ agree on the intersection $Z_i\cap Z_j$.
Then we can {\it glue} all these maps together to a single map $f_1\cup\cdots\cup f_n:Z_1\cup\cdots\cup Z_n\to X$.
If all subsets $Z_1,\ldots,Z_n$ are simultaneously open or closed and all maps $f_1,\ldots,f_n$
are continuous, then $f$ is also continuous.

Considering a topological space $X$ and a subset $Y\subset X$, we denote by $\overline Y$ and $Y^\circ$ the {\it closure}
and the {\it interior} of $Y$ respectively.

Let $V$ be a real (over $\R$) vector space. 
We will denote by $\GL(V)$ the group of all invertible linear maps $V\to V$ and by $V^*$
its {\it dual space}, which consists of all linear
maps $f:V\to\R$ with respect to the pointwise multiplication and addition.

For a subset $X\subset V$, we say that a vector $v\in V$ is a {\it positive combination} of elements of $X$
if $v=c_1x_1+\cdots+c_nx_n$, where $n>0$, $c_1,\ldots,c_n$ are positive real numbers and $x_1,\ldots,x_n\in X$.
If $0$ is not representable in this form, then we say that $X$ is {\it positively independent}.

A subspace $L\subset V$ defined by a nonconstant linear equation
is called a {\it hyperplane}. Thus the difference $V\setminus L$ has exactly two connected
components $V^+$ and $V^-$. We say that $L$ {\it separates} any points $A\in V^+$ and $B\in V^-$ and
denote this fact by $A\not\sim_L B$. Note that this condition implies $A,B\notin L$.
The negation of $A\not\sim_L B$ will be denoted by $A\sim_L B$.
This notation will also be applied to nonempty subsets $X,Y\subset V$ disjoint with $L$:
$$
X\sim_L Y\Leftrightarrow\forall x\in X\,\forall y\in Y: x\sim_Ly,
\qquad
X\not\sim_L Y\Leftrightarrow\forall x\in X\,\forall y\in Y: x\not\sim_Ly.
$$
The relation $\sim_L$ is 
symmetric and transitive
on the set of nonempty subsets of $V\setminus L$ disjoint with $L$
and, moreover,
\begin{equation}\label{eq:sep:3}
X\notsim_L Y\;\&\;Y\notsim_L Z\Rightarrow X\sim_LZ.
\end{equation}

\subsection{Graphs} Let $\Gamma=(V,E)$ be a finite undirected graph. In this paper, we will use the following ad-hoc terminology:
a {\it subgraph} of $G$ is a subset $F\subset E$. To think in the more standard terms, the reader
can identify $F$ with the pair $(V,F)$. In particular, we call $F$ a {\it cycle} if
the corresponding pair is a cycle, that is, a connected graph whose vertices have degrees $2$ or $0$.
If $\phi:V\to U$ is a map and $F$ is a subgraph of $\Gamma$, then we define
$$
\phi(F)=\{\{\phi(x),\phi(y)\}\suchthat\{x,y\}\in F\}.
$$

The set of all subgraphs of $\Gamma$ is thus equal to the power set $\P(E)$. It is a vector space over the two-element field under
the following operations:
$$
1\cdot F=F,\qquad 0\cdot F=\emptyset,\qquad F+F'=F\mathop{\triangle}F',
$$
where $\triangle$ denotes the symmetric difference. A subgraphs of $\Gamma$ is called {\it even} if any vertex of $V$
is incident with an even number of edges of $F$. The set of all even subgraphs of $\Gamma$
is a subspace of $\P(E)$, which 
is called the {\it cycle space} of $\Gamma$, see~\cite{Wallis},~\cite{Diestel},~\cite{GY}.
It is well-known that this space is spanned by some cycles~\cite{Veblen}.

\begin{definition}\label{definition:cycb6:2}
Let $e$ and $e'$ be edges of 
a graph having a common endpoint $x$.
We say that $e$ is $x$-moved to $e'$ by a sequence of subgraphs $F_1,\ldots,F_n$ if $e'$
is a unique edge with endpoint $x$ of the sum $e+F_1+\cdots+F_n$.
\end{definition}

\noindent
Clearly, the order in which the subgraphs $F_1,\ldots,F_n$ are enumerated does not play a role.
We leave the proof of the following simple result to the reader.

\begin{proposition}\label{lemma:cycl:3}
Let $x$ be a vertex of a graph $G$.
{\renewcommand{\labelenumi}{{\it(\roman{enumi})}}
\renewcommand{\theenumi}{{\rm(\roman{enumi})}}
\begin{enumerate}
\itemsep=5pt
\item\label{lemma:cycl:3:p:1} Any edge with endpoit $x$ is $x$-moved to itself by the empty sequence of subgraphs.
\item\label{lemma:cycl:3:p:2} If $e$ is $x$-moved to $e'$ by $F_1,\ldots,F_n$, then $e'$ is $x$-moved to $e$ by $F_1,\ldots,F_n$.
\item\label{lemma:cycl:3:p:3} If $e$ is $x$-moved to $e'$ by $F_1,\ldots,F_n$ and $e'$ is $x$-moved to $e''$
by $F'_1,\ldots,F'_{n'}$,
then $e$ is $x$-moved to $e''$ by $F_1,\ldots,F_n,F'_1,\ldots,F'_{n'}$.
\end{enumerate}}
\end{proposition}

\subsection{Circle}\label{Circle} The circle $S_r^1$ of radius $r>0$ plays a central role in this paper.
For each point $X\in S_r^1$ (resp. subset $X\subset S_r^1$), we denote by $X^\op$
the point (resp. the set of points) diametrically opposite to $X$. Thus for any line $L$ through the center,
$L\cap S_r^1=\{X,X^\op\}$ for some point $X\in S^1_r$.
The reflection $\r_L:S_r^1\to S_r^1$ through $L$ is a bijection of $S_r^1$ of order 2 with
fixed points $X$ and $X^\op$.

For any distinct points $A,B\in S_r^1$, there are exactly two {\it closed arcs} with endpoints $A$ and $B$.
If $A$ and $B$ are not diametrically opposite, then these arcs have different lengths
and we denote by $\arc{AB}$ the shorter of them, which is called the {\it minor arc}.
An {\it open arc} is the interior of a closed arc.
The length of an arc $\Omega$ is denoted by $|\Omega|$.

In various constructions with the circle, it will be convenient to use the {\it separation relation}:
four points $A,B,C,D\in S_r^1$ satisfy the quaternary relation $ABCD$ if and only if
the pair $(A,C)$ separates the pair $(B,D)$. This relation satisfies certain
properties~\cite[Chapter XXV, 203]{Russell} of which we mention the following:
%
%
\begin{equation}\label{IV}
A,B,C,D\text{ are pairwise distinct}\Rightarrow ABCD\text{ or }A\,CDB\text{ or }ADBC,
\end{equation}
\begin{equation}\label{V}
\hspace{-57pt}
ABCD\text{ and }A\,CDE\Rightarrow ABDE.
\end{equation}

\noindent
There are the following obvious relations between arcs, the separation relation and the lines.

\begin{proposition}\label{proposition:cycl:1}
Let $\Omega$ be an arc, $A,C\in\Omega$, $D\in S_r^1\setminus\Omega$ and $B\in S_r^1$ be a point such that $ABCD$.
Then $B\in\Omega$.
\end{proposition}

\begin{proposition}\label{proposition:cycl:2}
Let $L$ be a line through the center of $S_r^1$ and $X\in L\cap S_r^1$. 
This line separates points $A,B\in S_r^1$ if and only if 
$AX^\op BX$.
\end{proposition}

A closed arc $\Omega$ and a closed interval $[a,b]\subset\R$ such that $|\Omega|=b-a$
are isomorphic by a distance-preserving diffeomorphism $f:[a,b]\ito\Omega$.
Actually, there are two such diffeomorphisms: the one such that $f(a)=A$ and $f(b)=B$
and the other one such that $f(a)=B$ and $f(b)=A$, where $A$ and $B$ are the endpoints of $\Omega$.
We will also say that $f$ is a {\it smooth path} having {\it natural parametrization}.
In that case, the inverse map $I=f^{-1}:\Omega\ito[a,b]$ is called a {\it natural chart},
while the values of $I$ are called {\it $I$-coordinates}.

More generally, a path $f_1\cup f_2\cup\cdots\cup f_n:[x_1,x_{n+1}]\to S_r^1$, where $f_i:[x_i,x_{i+1}]\to S_r^1$ are
smooth paths having natural parametrization, is also said to have {\it natural parametrization}.
It is smooth outside $\{x_2,\ldots,x_n\}$.

Smooth paths having natural parametrization 
are a special case of monotonic maps\linebreak $f:[a,b]\to S_r^1$ 
that are defined by the following property:
$$
\forall\, t_1,t_2,t_3,t_4\in\R: a\le t_1<t_2<t_3<t_4\le b\Rightarrow f(t_1)f(t_2)f(t_3)f(t_4).
$$
\noindent
We leave the proofs of the following simple results to the reader.

\begin{proposition}\label{proposition:cycl:3}
Let $f:[a,b]\to S_r^1$ be a monotonic map. Then $f$ is injective
and the image of $f$ is contained in exactly one of the closed arcs between $f(a)$ and $f(b)$.
\end{proposition}

\begin{proposition}\label{proposition:cycl:4}
Let $f:[a,b]\to S_r^1$ be a continuous monotonic map. Then $f$ maps $[a,b]$ bijectively
to a closed arc between $f(a)$ and $f(b)$.
\end{proposition}

\begin{lemma}\label{lemma:circ:5}
Let $\Omega$ be an arc, $A\in\Omega$ and $D\in S_r^1\setminus\Omega$.
Let $L$ be a line through the center of $S_r^1$ such that $L\cap\Omega=\{X\}$ for some point $X\in S_r^1$.
For any $Z\in S_r^1$ such that $AZXD$, the line $L$ does not separate $A$ from $Z$.
\end{lemma}
\begin{proof}
Suppose that $L$ separates $A$ and $Z$. By Proposition~\ref{proposition:cycl:2},
we get
$AX^\op ZX$. 
Hence $AX^\op XD$ by~(\ref{V}). As $A,X\in\Omega$ and $D\notin\Omega$,
we get by Propisition~\ref{proposition:cycl:1} a contradiction $X^\op\in\Omega$.
\end{proof}

\begin{lemma}\label{lemma:circ:2}
Let $f:[a,b]\to S_r^1$ be a monotonic continuous function and $D\notin\im f$.
If $a\le t_1<t_2<t_3\le b$, then $f(t_1)f(t_2)f(t_3)D$.
\end{lemma}
\begin{proof}
As $f$ is injective, the points $f(t_1)$, $f(t_2)$, $f(t_3)$, $D$ are pairwise distinct.
Therefore, by~(\ref{IV}), we are in one of the following cases:
$f(t_1)f(t_2)f(t_3)D$, $f(t_2)f(t_3)f(t_1)D$, $f(t_3)f(t_1)f(t_2)D$.
The first case is just what we need. Therefore, let us consider the second case.
By Proposition~\ref{proposition:cycl:4}, $f([t_1,t_2])$ is an arc.
As $D$ does not belong to it, $f(t_3)$ does by Proposition~\ref{proposition:cycl:1}.
However, it contradicts the injectivity of $f$.
The third case is treated similarly.
\end{proof}

Let $g:[a,b]\to S_r^1$ be a path. Suppose that $t,t'\in[a,b]$ such that $t<t'$ and
$g(t)$ and $g(t')$ both belong to a line $L$ through the center of $S_r^1$ (these points either coincide or are diametrically opposite).
Then we define a new path $\f_{t,t'}g:[a,b]\to S_r^1$ by
$$
\f_{t,t'}g(t'')=
\left\{
\begin{array}{ll}
\r_L\circ g(t'')&\text{if } t<t''<t';\\
g(t'')&\text{otherwise}.
\end{array}
\right.
$$

\noindent
For the sequel, we will need the following simple fact.

\begin{proposition}\label{proposition:circ:6}
If $g$ is a {\it path} having {\it natural parametrization},
then $\f_{t,t'}g$ is also a {\it path} having {\it natural parametrization}.
It is smooth wherever $g$ is smooth except possibly at $t$ and~$t'$.
\end{proposition}
\noindent
We will call $\f_{t,t'}g$ the {\it double fold} of $g$ at $t$ and $t'$. Suppose additionally that $X$
is a nonempty subset of the plane disjoint from $L$
whose points are not separated by $L$.
Then we say that $\f_{t,t'}g$ is the double fold of $g$ at $t$ and $t'$ {\it towards} $X$
if $L$ separates no point of $X$ and $\f_{t,t'}g(t'')$ for any $t''\in[t,t']$.

\subsection{Coxeter groups} Let $S$ be a
set and $m:S^2\to\{1,2,\ldots,\infty\}$
be a function. If $m$ is symmetric, $m(s,s)=1$ for any $s\in S$ and $m(s,s')>1$ for $s\ne s'$, then 
it is called a {\it Coxeter matrix}. The group defined by generators and relations as follows:
$$
W=\<S\suchthat (ss')^{m(s,s')}=1\text{ for all pairs }(s,s')\in S^2\text{ such that }m(s,s')<\infty\>
$$
is called a {\it Coxeter group}.
It is well-known that $\ord(ss')=m(s,s')$ for any $s,s'\in S$.
The elements of $S$ are called {\it simple reflections} and
the pair $(W,S)$ is called a {\it Coxeter system} of {\it rank} $|S|$.
We will consider Coxeter systems upto
isomorphisms, which are isomorphisms of Coxeter groups that map bijectively simple reflections to simple reflections.
Let
$$
T=\bigcup_{w\in W} wSw^{-1}.
$$
The elements of this set are called {\it reflections}.

A subgroup $W_I$ of $W$ generated by a subset $I\subset S$ is called {\it standard parabolic}.
I that case, the pair $(W_I,I)$ is also a Coxeter system.
All subgroups $W'=wW_Iw^{-1}$ of $W$ for $w\in W$ are called {\it parabolic}.
The {\it rank} of $W'$ is defined to be $|I|$. For a subset $X\subset W$, the smallest 
parabolic subgroup of $W$ containing $X$ is called (if it exists)
the {\it parabolic closure} of~$X$. Passing to a standard parabolic subgroup of finite rank
containing all simple reflections of the expression under consideration
allows us
to consider only Coxeter groups of finite ranks in the rest of the paper.
The definitions and basic facts about Coxeter groups can be found in~\cite{Bourbaki},\cite{Humphreys},\cite{BB}.
\subsection{Based root systems}
We are going to outline the main constructions of~\cite{BD}. 
Let $V$ be a 
real vector space. Suppose that $\alpha\in V$ and $\alpha^\vee\in V^*$
are such that $\alpha^\vee(\alpha)=2$. The {\it reflection} of $V$ with {\it root} $\alpha$ and {\it coroot} $\alpha^\vee$
is a linear map $\r_{\alpha,\alpha^\vee}:V\to V$ defined by
$$
v\mapsto v-\alpha^\vee(v)\alpha.
$$
We are especially interested in the case where there is a symmetric
bilinear form $(\_|\_):V\times V\to\R$ such that $\alpha$ is a {\it unit} vector, that is, $(\alpha|\alpha)=1$,
and $\alpha^\vee$ is determined 
by 
$$
\alpha^\vee(v)=2(v|\alpha).
$$
In this case, we abbreviate $\r_\alpha=\r_{\alpha,\alpha^\vee}$.
Note that $(\r_\alpha v|\r_\alpha u)=(v|u)$.
In particular, a vector $\r_\alpha v$ is unit iff $v$ is so.
It is well-known that $\r_{-\alpha}=\r_\alpha$, $\r_\alpha^2=\id$ and
\begin{equation}\label{eq:r_r=rrr}
\r_{\r_\alpha v}=\r_\alpha\r_v\r_\alpha.
\end{equation}

Let $\Pi\subset V$ be a subset of unit vectors satisfying the following properties:
\begin{itemize}
\item the set $\Pi$ is positively independent;
\item for distinct $\alpha,\beta\in\Pi$, either $$(\alpha|\beta)=-\cos\(\frac{\pi}{m(\alpha,\beta)}\)$$ for some
      natural number $m(\alpha,\beta)\ge2$ or $(\alpha|\beta)\le-1$, in which case we define $m(\alpha,\beta)=\infty$.
\end{itemize}
Let also $m(\alpha,\alpha)=1$ for any $\alpha\in\Pi$ and
$$
S=\{\r_\alpha\suchthat \alpha\in\Pi\},\quad W=\<S\>\subset\GL(V),\quad \Phi=\{w\alpha\suchthat w\in W,\alpha\in\Pi\}.
$$
The pair $(\Phi,\Pi)$ is called a {\it based root system} in $V$. Specifically, elements of $\Phi$ are called {\it roots},
elements of $\Pi$ are called {\it simple roots} and $V$ is called the {\it total space}.
%
In this case, $(W,S)$ is a Coxeter system
with Coxeter matrix $m$, called the Coxeter system {\it associated} to $(\Phi,\Pi)$.
In that case, $W$ also acts on the dual space $V^*$ by the rule $(wf)(v)=f(w^{-1}v)$,
where $w\in W$, $f\in V^*$ and $v\in V$.

The set $\Phi$ is the disjoint union of $\Phi^+$ and $\Phi^-$, where
$\Phi^+$ denotes the set of all roots, called {\it positive}, that are positive combinations of elements of $\Pi$
and $\Phi^-=-\Phi^+$. 
The facts $\alpha\in\Phi^+$ and $\alpha\in\Phi^-$ are abbreviated to $\alpha>0$ and $\alpha<0$ respectively.
We will also say that two roots {\it have different signs} if they
do not belong simultaneously to $\Phi^+$ or to $\Phi^-$.

Note that there is a bijection $\Phi^+\to T$ given by $\alpha\mapsto\r_\alpha$.
We will denote the inverse map by $t\mapsto\v_t$. It follows from~(\ref{eq:r_r=rrr}) that
\begin{equation}\label{eq:b}
\r_{w\v_s}=wsw^{-1},\quad
\r_\alpha=\r_\beta\Leftrightarrow\alpha=\pm\beta,\quad w\r_\alpha w^{-1}=\r_{w\alpha}
\end{equation}
for $s\in S$, $w\in W$ and $\alpha,\beta\in\Phi$.

\begin{remark}\rm\label{rem:arbitrary:WS}
An arbitrary Coxeter system $(W,S)$ is isomorphic to the Coxeter system associated to
a based root system, for which the simple roots form a basis of the total space.
\end{remark}

\subsection{Reflection subgroups}\label{Reflection_subgroups} Let $(W,S)$ be a Coxeter system. A subgroup $W'\subset W$ is called {\it a reflection subgroup}
if it is generated by reflections it contains, that is, $W'=\<W'\cap T\>$.
To describe a generating set of $W'$, let us consider the function $N:W\to\P(T)$ defined by
$$
N(w)=\{t\in T\suchthat \l(tw)<\l(w)\},
$$
where $\l:W\to\{0,1,2,\ldots\}$ is the length function for $W$ with respect to $S$.
Let
$$
\chi(W')=\{t\in T\suchthat N(t)\cap W'=\{t\}\}.
$$

\begin{proposition}[\mbox{\cite{Deodhar1},\cite{Dyer1}}]\label{proposition:DeDy}
Let $W'$ be a reflection subgroup of a Coxeter group $W$.
Then $(W',\chi(W'))$ is a Coxeter system.
\end{proposition}

\begin{proposition}[\mbox{\cite[Lemma 3.5]{BD}}]\label{proposition:0}
Let $(\Phi,\Pi)$ be a based root system in $V$, $(W,S)$ be the associated Coxeter system
and $W'$ be a reflection subgroup of $W$. Let
$$
\Phi_{W'}=\{\alpha\in\Phi\suchthat\r_\alpha\in W'\},\quad \Pi_{W'}=\{\alpha\in\Phi^+\suchthat\r_\alpha\in\chi(W')\},\quad
\Phi_{W'}^+=\Phi_{W'}\cap\Phi^+.
$$
Then $(\Phi_{W'},\Pi_{W'})$ is a based root system in $V$ with associated Coxeter system $(W',\chi(W'))$
and positive roots $\Phi_{W'}^+$.
\end{proposition}
\noindent
The set $\Pi_{W'}$ from this proposition is called the {\it canonical simple system} for $W'$.
The cardinality of this set is equal to $|\chi(W')|$. 

\begin{proposition}[\mbox{\cite[Lemma 2.9]{Dyer2}}]\label{proposition:Dy}
Let $W'$ be a finite reflection subgroup of a Coxeter group $W$. Then the parabolic closure
of $W'$
exists and is a finite reflection subgroup of the same rank as~$W'$.
\end{proposition}

In this paper, we will only be interested in the case, where a root subgroup $W'$ is generated by two distinct reflections.
Such subgroups are called {\it dihedral reflection subgroups}.
In that case, $|\chi(W')|=2$ by~\cite[3.9]{Dyer1}. Therefore, the canonical simple system consists of two roots.
It is easy to check that they are linearly independent.

\begin{corollary}\label{corollary:6}
The set of all finite orders of products of two reflections in a Coxeter group $W$
coincides with the set of all positive divisors of finite entries of the Coxeter matrix of~$W$.
\end{corollary}
\begin{proof}
Let $t_1$ and $t_2$ be reflections. If $t_1=t_2$, then $\ord(t_1t_2)=\ord(1)=1$.
Suppose that $t_1\ne t_2$ and $n=\ord(t_1t_2)<\infty$. 
By Proposition~\ref{proposition:Dy},
the parabolic closure $\D'$ of $\D=\<t_1,t_2\>$ exists and has rank 2, that is, $\D'=wW_{\{s_1,s_2\}}w^{-1}$, where $s_1\ne s_2$
are simple reflections.
Hence $|\D'|=|W_{\{s_1,s_2\}}|=2m(s_1,s_2)$, where $m$ denotes the Coxeter matrix of $W$.
By Lagrange's theorem $|\D|=2n$ is a divisor of $2m(s_1,s_2)$, whence $n$ is a divisor of $m(s_1,s_2)$.
The converse statement follows from a standard description of subgroups of dihedral groups.
\end{proof}

\def\V{\mathscr V}

\subsection{Walls, chambers and panels}\label{Walls, chambers and panels}
Let $(W,S)$ be a Coxeter system associated to a based root system $(\Phi,\Pi)$
in $V$ such that $\Pi$ is a basis of $V$.
As usual, the set of reflections is denoted by $T$ .
For $t\in T$ and $\alpha\in\Phi$,
we consider the following hyperplanes in $V^*$:
$$
H_t=\{f\in V^*\suchthat tf=f\},\quad H_\alpha=\{f\in V^*\suchthat f(\alpha)=0\}
$$
called a {\it $t$-wall} and an {\it $\alpha$-wall}. 
Clearly, $H_t=H_\alpha$ if $t=\r_\alpha$ and $H_t=H_r$ implies $t=r$.
Moreover, $wH_t=H_{wtw^{-1}}$ for any $t\in T$ and $w\in W$.
Note that $H_\alpha$ divides the difference $V^*\setminus H_\alpha$ into the halfspaces
$$
V^{*+}_\alpha=\{f\in V^*\suchthat f(\alpha)>0\},\quad V^{*-}_\alpha=\{f\in V^*\suchthat f(\alpha)<0\}.
$$
We will consider the following convex cone
\begin{equation}\label{eq:C}
C=\{f\in V^*\suchthat\forall s\in S: f(\v_s)>0\}.
\end{equation}
Recall that $\v_s$ is the simple root corresponding to $s$.
It follows from this definition that $f(\alpha)>0$ for any $f\in C$ and $\alpha\in \Phi^+$.

Any set $w\,C$, where $w\in W$, is called a {\it chamber}
and $C$ itself is called the {\it fundamental} chamber.
Its closure 
is equal to
$$
\overline{C}=\{f\in V^*\suchthat\forall s\in S: f(\v_s)\ge0\}
$$
and is the disjoint union of the following sets:
\begin{equation}\label{eq:CX}
C_X=\{f\in V^*\suchthat \forall s\in X: f(\v_s)=0;\; \forall s\in S\setminus X: f(\v_s)>0\}
\end{equation}
over all $X\subset S$. Obviously, $C_\emptyset=C$ and $C_S=\{0\}$.
These sets have the following property whose prove is almost\footnote{The only difference is~\cite[Ch.\,V, \sectsign 4, no.\,5]{Bourbaki},
where case a) may produce a hyperbola instead of a line.} identical
to~\cite[Ch.\,V, \sectsign 4, no.\,6, Proposition~5]{Bourbaki}.
\begin{proposition}\label{proposition:4}
\!Suppose that $w\,C_X\cap w'\,C_{X'}$ $\ne\emptyset$ for some $w,w'\in W$ and $X,X'\subset S$.
Then $X=X'$, $w\<X\>=w'\<X'\>$, and $w\,C_X=w'\,C_{X'}$.
\end{proposition}

\noindent
We will use the following abbreviation: $C_s=C_{\{s\}}$ for any $s\in S$.
These sets are called {\it panels}.
Clearly, $C_s\subset H_s$.

For any union
$R=\bigcup_{i\in I} w_iC_{X_i}$, where $w_i\in W$ and $X_i\subset S$, we define
$\codim R$ to be the minimum of $|X_i|$ for $i\in I$.
For example, $\codim w\,C=0$ and $\codim w\,C_s=1$.

\begin{lemma}\label{lemma:16} Let
$\alpha\in\Phi$ and $w\in W$.
If $w^{-1}\alpha>0$, then $wC\subset V^{*+}_\alpha$.
If $w^{-1}\alpha<0$, then $wC\subset V^{*-}_\alpha$.
Therefore, walls and chambers are disjoint.
\end{lemma}
\begin{proof}
Suppose that $w^{-1}\alpha>0$. For any $c\in C$, we get $wc(\alpha)=c(w^{-1}\alpha)>0$.
The second claim is proved similarly.
\end{proof}

\begin{corollary}\label{corollary:11} Let 
$\alpha\in\Phi$ and
$x,y\in W$.
The wall $H_\alpha$ separates the chambers $x\,C$ and $y\,C$ if and only if
the roots $x^{-1}\alpha$ and $y^{-1}\alpha$ have different signs.
\end{corollary}
\begin{proof}
The result follows from Lemma~\ref{lemma:16}.
\end{proof}

\begin{definition}\label{connection}
Chambers $w\,C$ and $w'\,C$ are connected through a wall
$H_t$ if and only if
$\codim w\,\overline{C}\cap w'\,\overline{C}\cap H_t=1$.
\end{definition}

\begin{remark}\label{rem:1}\rm The codimension in Definition~\ref{connection}
is well-defined
and 
it suffices to claim that it be at most $1$. Therefore, if $w\,C$ and $w'\,C$ are connected through $H_t$,
then $w\,C$ and $w\,C$ (resp. $w'\,C$ and $w'\,C$) are also connected through $H_t$.
\end{remark}

\begin{remark}\label{rem:4}\rm
If $\dim V=2$, then $w\,C$ and $w'\,C$ are connected through $H_t$ if and only
if $w\,\overline{C}\cap w'\,\overline{C}\cap H_t\ne0$.
This result follows from the previous remark.
\end{remark}

\noindent
The following results follow directly from the definitions and Proposition~\ref{proposition:4}.

\begin{proposition}\label{proposition:2}
Chambers $w\,C$ and $w'\,C$ are connected through a wall $H_t$ if and only if there exists $s\in S$
such that $t=wsw^{-1}$ and either $w'=w$ or $w'=ws$.
\end{proposition}

\begin{proposition}\label{corollary:12}
If chambers $w\,C$ and $w'\,C$ are connected through a wall $H_t$
and a wall $H_r$ separates them, then $H_r=H_t$.
\end{proposition}

Finally, we define the Coxeter complex $\Sigma(W,S)$ to be a simplicial complex whose simplices of codimension
$k$ are $(\overline C_X\setminus\{0\})/\R^{>0}$, where $|X|=k$ and $\R^{>0}$ denotes the group of positive integers.
The union of simplices of codimension zero (the total space) is $(\mathfrak X\setminus\{0\})/\R^{>0}$,
where $\mathfrak X$ is the {\it Tits cone} of $(W,S)$, which is by definition the union of the closures of all chambers.
We will call the images of chambers, walls and panels under the projection $\mathfrak X\setminus\{0\}\to(\mathfrak X\setminus\{0\})/\R^{>0}$
are also {\it chambers}, {\it walls} and {\it panels}.

\section{Subexpressions and galleries}

\subsection{Expressions and subexpressions}\label{Expressions_and_subexpressions} As we already mentioned in the introduction, an {\it expression}
is a finite sequence $\u{s}=(s_1,\ldots,s_m)$, where each $s_i\in S$ for a Coxeter system $(W,S)$.
Its {\it subexpression} is a sequence $\u{\gamma}=(\gamma_1,\ldots,\gamma_m)$ such that $\gamma_i\in\{0,1\}$
for each $i=1,\ldots,m$. We denote this fact by $\u{\gamma}\subset\u{s}$.
The set of all subexpressions of $\u{s}$ is denoted by $\SubExpr(\u{s})$. Note that $\SubExpr(\u{s})\cong\{0,1\}^{|\u{s}|}$.
Yet we will assume that the sets $\SubExpr(\u{s})$ are disjoint for different $\u{s}$.
In other words, each subexpression $\u{\gamma}$ ``knows'' its expression (or type) $\u{s}$.
If we, however, want to say that two sequences (of any origin) $\u{\gamma}$ and $\u{\delta}$ in $0$ and $1$
are equal, we will write $\u{\gamma}\=\u{\delta}$.

Let $\u{\gamma}$ and $\u{\delta}$ be 
subexpressions of $\u{s}$ and $X=\{i=1,\ldots,m\suchthat\gamma_i\ne\delta_i\}$.
We define the {\it Hamming distance} by $\dist(\u{\gamma},\u{\delta})=|X|$.
If it is grater then zero, then  we say that $\u{\gamma}$ and $\u{\delta}$ {\it differ maximally} at
the maximal element of $X$.

For each $i=1,\ldots,m$, we consider the {\it folding operator} $\f_i:\SubExpr(\u{s})\to\SubExpr(\u{s})$
given by
$$
\f_i(\gamma_1,\ldots,\gamma_m)=(\gamma_1,\ldots,\gamma_{i-1},1-\gamma_i,\gamma_{i+1},\ldots,\gamma_m).
$$
The operators $\f_i$ have order two and commute with each other. We will use the notation
\begin{equation}\label{eq:less}
\u{\gamma}^{<i}=s_1^{\gamma_1}s_2^{\gamma_2}\cdots s_{i-1}^{\gamma_{i-1}},
\qquad\u{\gamma}^{\max}=\u{s}^{\u{\gamma}}.
\end{equation}

We assume additionally that $(W,S)$ is the Coxeter system associated to a based root system $(\Phi,\Pi)$ such that $\Pi$
is a basis of its total space $V$. In that case, there is a bijection $T\ito\Phi^+$ given by the map $t\mapsto\v_t$.
For any subexpression $\u{\gamma}\subset\u{s}$, we define the vector
\begin{equation}\label{eq:to}
\u{\gamma}^{\to i}=\u{\gamma}^{<i}(-\v_{s_i}).
\end{equation}
The reader can check the following simple formulas:
\begin{equation}\label{eq:cyb1:cat:5}
(\f_i\u{\gamma})^{<k}=
\left\{
\begin{array}{ll}
\r_{\u{\gamma}^{\to i}}\u{\gamma}^{<k}&\text{if }k>i;\\[3pt]
\u{\gamma}^{<k}&\text{if }k\le i,
\end{array}
\right.
\quad
(\f_i\u{\gamma})^{\to k}=
\left\{
\begin{array}{ll}
\r_{\u{\gamma}^{\to i}}\u{\gamma}^{\to k}&\text{if }k>i;\\[3pt]
\u{\gamma}^{\to k}&\text{if }k\le i.
\end{array}
\right.
\end{equation}

For a pair of indices $(i,j)$ such that $1\le i<j\le m$, we define the {\it double folding operator} $\f_{i,j}$ as follows:
it is applicable to a subexpression $\u{\gamma}\subset\u{s}$ if and only if
$\f_i\f_j\u{\gamma}$ has the same target as $\u{\gamma}$, in which case we define $\f_{i,j}\u{\gamma}=\f_i\f_j\u{\gamma}$.
We call $\f_{i,j}\u{\gamma}$ the {\it double fold} of $\u{\gamma}$ at $i$ and $j$.
Note the following equivalent (and sometimes easier to check) reformulations of the applicability condition.

\begin{proposition}\label{proposition:applicable}
Let $\u{\gamma}\subset\u{s}$. The double folding operator $\f_{i,j}$ is applicable to $\u{\gamma}$ if and only if
one of the following equivalent conditions holds:
{\renewcommand{\labelenumi}{{\it(\roman{enumi})}}
\renewcommand{\theenumi}{{\rm(\roman{enumi})}}
\begin{enumerate}
\itemsep=4pt
\item\label{proposition:applicable:i} $s_is_{i+1}^{\gamma_{i+1}}\cdots s_{j-1}^{\gamma_{j-1}}=s_{i+1}^{\gamma_{i+1}}\cdots s_{j-1}^{\gamma_{j-1}}s_j$;
\item\label{proposition:applicable:ii} $\u{\gamma}^{\to i}=\pm\u{\gamma}^{\to j}$.
\end{enumerate}}
\end{proposition}
\begin{proof}
Part~\ref{proposition:applicable:i} follows from the equality $\u{s}^{\u{\gamma}}=\u{s}^{\f_i\f_j\u{\gamma}}$
by cancelling out the factors at positions $1,\ldots,i-1,j+1,\ldots,|\u{s}|$.
As $s_j=\r_{\v_{s_j}}$, part~\ref{proposition:applicable:i} is equivalent to
$$
s_i=\r_{s_{i+1}^{\gamma_{i+1}}\cdots s_{j-1}^{\gamma_{j-1}}\v_{s_j}}.
$$
Conjugating with $s_1^{\gamma_1}\cdots s_i^{\gamma_i}$, we get an equivalent
form $\r_{\u{\gamma}^{\to i}}=\r_{\u{\gamma}^{\to j}}$,
whence part~\ref{proposition:applicable:ii}.
\end{proof}

There are the following formulas for double folding operators similar to~(\ref{eq:cyb1:cat:5}):
\begin{equation}\label{eq:33}
(\f_{i,j}\u{\gamma})^{<k}=
\left\{\arraycolsep=2pt
\begin{array}{ll}
\r_{\u{\gamma}^{\to i}}\u{\gamma}^{<k}&\text{if }\;i<k\le j;\\[3pt]
\u{\gamma}^{<k}&\text{otherwise},
\end{array}
\quad
\right.
(\f_{i,j}\u{\gamma})^{\to k}=
\left\{\arraycolsep=2pt
\begin{array}{ll}
\r_{\u{\gamma}^{\to i}}\u{\gamma}^{\to k}&\text{if }\;i<k\le j;\\[3pt]
\u{\gamma}^{\to k}&\text{otherwise}.
\end{array}
\right.
\end{equation}

Thus any edge of $\Sub(\u{s},w)$ has the form $(\u{\gamma},\f_{i,j}\u{\gamma})$. The {\it color} $\alpha$ of this edge
is the positive of two roots $\pm\u{\gamma}^{\to i}$
We will also say that $(i,j)$ is an {\it $\alpha$-pair} for $\u{\gamma}$.

\subsection{The order}\label{Orders} We will use the following binary relation on $\SubExpr(\u{s})$:
$\u{\delta}<\u{\gamma}$ if and only if
$$
\u{\delta}\ne\u{\gamma},\quad \u{\delta}^{\max}=\u{\gamma}^{\max},\quad \l(\u{\delta}^{<i})<\l(\u{\gamma}^{<i}),
$$
where $\u{\delta}$ and $\u{\gamma}$ differ maximally at $i$ and $\l$ denotes the length function for $(W,S)$.
The reader can easily prove, applying~\cite[Lemma 3.1]{BD}, that $\u{\delta}<\u{\gamma}$ is equivalent to
$\u{\gamma}^{\to i}>0$ and to $\u{\delta}^{\to i}<0$.

\begin{lemma}
The relation $<$ a partial strict order on $\SubEx(\u{s})$.
It is a total strict order on each $\SubEx(\u{s},w)$.
\end{lemma}
\begin{proof}
By the discussion above, we need only to prove the transitivity condition.
Let $\u{\sigma}<\u{\delta}<\u{\gamma}$, where $\u{\sigma}$ and $\u{\delta}$
differ maximally at $j$ and $\u{\delta}$ and $\u{\gamma}$ differ maximally at $i$.
If $j<i$, then $\u{\gamma}$ and $\u{\sigma}$ differ maximally at $i$, whence $\u{\gamma}>\u{\sigma}$
by $\u{\gamma}^{\to i}>0$. If $i<j$, then $\u{\gamma}$ and $\u{\sigma}$ differ maximally at $j$,
whence $\u{\gamma}>\u{\sigma}$ by $\u{\sigma}^{\to j}<0$. Finally, the case $i=j$ is impossible,
as it would imply that $\u{\delta}^{\to i}>0$ and $\u{\delta}^{\to i}<0$ hold simultaneously.
\end{proof}

\subsection{Galleries}\label{Galleries}
In this section, we keep the notation of the previous sections.

\begin{definition}\label{gallery}
A (labelled) gallery is a sequence
$$
\u{\Gamma}=(C_1,L_1,C_2,L_2,\ldots,C_m,L_m,C_{m+1}),
$$
where $C_1,\ldots,C_{m+1}$ are chambers and $L_1,\ldots,L_m$ are walls
such that $C_i$ and $C_{i+1}$ are connected through $L_i$ for any $i=1,\ldots,m$.
We say that $\u{\Gamma}$ begins at $C_1$ and ends at $C_{m+1}$ and
that $C_i$ is its $i$th chamber and $L_i$ is its $i$th wall. We denote $|\u{\Gamma}|=m$.
\end{definition}
\noindent
The difference between this definition and the definition of galleries in~\cite{CC},~\cite{G},~\cite{MST}
is that we use walls instead of panels. We will omit the adjective ``labelled''.

\begin{lemma}\label{lemma:20}
Let $\u{\Gamma}=(C_1,L_1,C_2,L_2,\ldots,C_m,L_m,C_{m+1})$ be a gallery,
$L$ be a wall and $i,j$ be indices such that $1\le i<j\le m+1$ and
$C_i\not\sim_LC_j$.
Then there exists an index $k$ such that $i\le k<j$, $L_k=L$ and
$C_k\notsim_LC_{k+1}$
but
$C_i\sim_LC_k$
and
$C_{k+1}\sim_LC_j$
\end{lemma}
\begin{proof}
Let $k$ be the maximal index such that $i\le k<j$ such that
$C_i\sim_LC_k$.
Therefore, $C_i\notsim_L C_{k+1}$ and $C_k\notsim_L C_{k+1}$.
Hence $L=L_k$ by Lemma~\ref{corollary:12}. 
The relation $C_{k+1}\sim_L C_j$ follows from
$C_i\notsim_LC_{k+1}$
and
$C_i\notsim_LC_j$
by~(\ref{eq:sep:3}).
\end{proof}

The following result is a version of~\cite[Lemma 4.8]{GSch}. Its proof follows from
Proposition~\ref{proposition:2}.

\begin{proposition}\label{proposition:3}
The set of all subexpressions of all expressions in $S$ is bijectively mapped to the set of all galleries beginning at $C$ by
$$
\u{\gamma}\mapsto\u{\Gamma}=(\u{\gamma}^{<1}C,H_{\u{\gamma}^{\to1}},\u{\gamma}^{<2}C,H_{\u{\gamma}^{\to2}},\ldots,\u{\gamma}^{\max}C).
$$
Thus for any $w\in W$, the set of all subexpressions with target $w$ of all expressions in $S$
is bijectively mapped to the set of all galleries beginning at $C$
and ending at $w\,C$.
\end{proposition}

In what follows, we will always mean this connection between small and capital underlined Greek letters,
the former meaning subexpressions and the latter the corresponding galleries starting at the fundamental chamber.

In the above propositions, the sign of the root $\u{\gamma}^{\to i}$ can be ignored.
However, it delivers an important information about how the gallery $\u{\Gamma}$ approaches
its $i$th wall.

\begin{lemma}\label{lemma:to}
Let $\u{\gamma}\subset\u{s}$ and $\u{\Gamma}$ be the corresponding gallery.
The $i$th wall of $\u{\Gamma}$ separates the fundamental chamber from
the $i$th chamber of $\u{\Gamma}$ if and only if $\u{\gamma}^{\to i}>0$.
\end{lemma}
\begin{proof}
Let $\alpha=\u{\gamma}^{\to i}$. Then the $i$th wall is $H_\alpha$ and the $i$th chamber is $\u{\gamma}^{<i}C$.
As
$$
1^{-1}\alpha=\u{\gamma}^{\to i},\quad (\u{\gamma}^{<i})^{-1}\alpha=-\v_{s_i}<0,
$$
the result follows from Corollary~\ref{corollary:11}.
\end{proof}

\begin{corollary}\label{corolary:0}
Let $\u{\gamma}\subset\u{s}$ be such that $\u{\gamma}^{\to j}>0$.
Then there exists an index $k<j$ such that $\u{\gamma}^{\to k}=-\u{\gamma}^{\to j}$.
\end{corollary}
\begin{proof}
Let $\u{\Gamma}=(C_1,L_1,C_2,L_2,\ldots,C_m,L_m,C_{m+1})$ be the gallery corresponding to $\u{\gamma}$.
By Lemma~\ref{lemma:to}, the wall $L_j$ separates $C_1$ and $C_j$.
Let $k$ be an index as in Lemma~\ref{lemma:20} for $i=1$ and $L=L_j$.
In view of $L_k=L=L_j$, we get $\u{\gamma}^{\to k}=\pm\u{\gamma}^{\to j}$.
As $L_k$ does not separate $C_k$ from the fundamental chamber $C_1$, we get $\u{\gamma}^{\to k}<0$
by Lemma~\ref{lemma:to}. 
\end{proof}


\begin{corollary}\label{corolary:00}
Let $\u{\gamma}$ be a subexpressions of $\SubEx(\u{s},w)$ distinct from the minimal subexpression of this set.
Then there are indices $i$ and $j$ such that $\f_{i,j}\u{\gamma}<\u{\gamma}$.
\end{corollary}
\begin{proof}
Let $\u{\gamma}$ and the minimal subexpression $\u{\delta}$ of $\SubEx(\u{s},w)$
differ maximally at $j$. As $\u{\gamma}>\u{\delta}$, we get $\u{\gamma}^{\to j}>0$.
By Corollary~\ref{corolary:0}, there is some $i<j$
such that $\u{\gamma}^{\to i}=-\u{\gamma}^{\to j}$. By Proposition~\ref{proposition:applicable},
the operator $\f_{i,j}$ is applicable to $\u{\gamma}$. Hence $\f_{i,j}\u{\gamma}<\u{\gamma}$,
as $\u{\gamma}$ and $\f_{i,j}\u{\gamma}$ differ maximally at~$j$.
\end{proof}

Now we can proof the first result declared in the introduction.

\begin{proof}[Proof of Theorem~\ref{theorem:connected}]
It suffices to assume that the set of vertices of $\Sub(\u{s},w)$ is not empty. 
Let $\u{\delta}$ be the minimal of them.
By Corollary~\ref{corolary:00}, there is a path from any vertex of $\Sub(\u{s},w)$ to $\u{\delta}$.
Therefore, any two vertices are connected by a path (through $\u{\delta}$).
\end{proof}

\begin{definition}\label{def:gallery_double_fold}
Let $\u{\Gamma}$ be a gallery denoted as in Definition~\ref{gallery} and
$i,j$ be indices such that $1\le i<j\le m$ and
$L_i=L_j$ be $t$-walls. Then we define
$$
\f_{i,j}\u{\Gamma}=(C_1,L_1,\ldots,C_i,L_i,tC_{i+1},tL_{i+1}\ldots,tC_j,L_j,C_{j+1},\ldots,L_m,C_{m+1}).
$$
We call $\f_{i,j}\u{\Gamma}$ the double fold of $\u{\Gamma}$ at $i$ and $j$ and say that
it is towards the fundamental chamber if $L_i$ does not separate the chambers $tC_{i+1},\ldots,tC_j$
from the fundamental chamber.
\end{definition}
\noindent
By Proposition~\ref{proposition:2},
the sequence $\f_{i,j}\u{\Gamma}$ is a well-defined gallery.

\begin{proposition}\label{proposition:cycb6:6}
Let $\u{\Gamma}$ be a gallery and $i,j$ be indices. The folding operator $\f_{i,j}$
is applicable to $\u{\Gamma}$ if and only if it is applicable to $\u{\gamma}$.
If this is the case, let $\u{\Delta}=\f_{i,j}\u{\Gamma}$. Then $\u{\delta}=\f_{i,j}\u{\gamma}$.
If, moreover, $\u{\Delta}$ is a double fold of $\u{\Gamma}$ towards the fundamental chamber,
then $\u{\gamma}>\u{\delta}$.
\end{proposition}
\begin{proof}
The applicability claim follows directly from Proposition~\ref{proposition:applicable}.
The equality $\u{\delta}=\f_{i,j}\u{\gamma}$ is an easy consequence of~(\ref{eq:33}).
Finally, suppose that $\u{\Delta}$ is a double fold of $\u{\Gamma}$ towards the fundamental chamber.
Then the $j$th wall of $\u{\Gamma}$ separates the $j$th chamber of $\u{\Gamma}$ from the $j$th chamber of $\u{\Delta}$
and does not separate the latter chamber from the fundamental chamber.
By Lemma~\ref{lemma:to}, we get $\u{\gamma}^{\to j}>0$. Therefore, $\u{\gamma}>\u{\delta}$.
\end{proof}

\begin{remark}\rm\label{rm:leftarrow}
The reader may wonder why we have Lemma~\ref{lemma:to} and not a lemma telling us
when the $i$th wall of $\u{\Gamma}$ separates the fundamental chamber
from the $i+1$th chamber.
Such a result also holds and states that these chambers are separated if and only if
$\u{\gamma}^{\leftarrow i}=\u{\gamma}^{<{i+1}}(-\v_{s_i})>0$.
So the signs of the roots $\u{\gamma}^{\to i}$ and $\u{\gamma}^{\leftarrow i}$ show how $\u{\Gamma}$
approaches and goes away from the $i$th wall from the point of view of the fundamental chamber.
\end{remark}

\subsection{Special pairs} We will use the following type of pairs.

\begin{definition}\label{definition:cycb6:3} Let $\u{\gamma}\subset\u{s}$. A pair of indices
$(i,j)$ is special for $\u{\gamma}$ if
$1\le i<j\le|\u{s}|$ and
{\renewcommand{\labelenumi}{{\it(\roman{enumi})}}
\renewcommand{\theenumi}{{\rm(\roman{enumi})}}
\begin{enumerate}
\itemsep=5pt
\item\label{definition:cycb6:3:p:i} $\u{\gamma}^{\to j}=-\u{\gamma}^{\to i}>0$;
\item\label{definition:cycb6:3:p:ii} there is no $k$ such that $k<i$ and $\u{\gamma}^{\to k}=\u{\gamma}^{\to j}$;
\item\label{definition:cycb6:3:p:iii} there is no $k$ such that $i<k<j$ such that $\u{\gamma}^{\to k}=\u{\gamma}^{\to i}$.
\end{enumerate}}
\end{definition}

\begin{lemma}\label{lemma:cycl:2}
Let $\u{\gamma}\subset\u{s}$ and
$
\u{\Gamma}=(C_1,L_1,\ldots,C_m,L_m,C_{m+1})
$
be the corresponding gallery. Let $(i,j)$ be a special pair for $\u{\gamma}$.
Then the wall $L_i$ does not separate the chambers $C_1,\ldots,C_i$
and separates $C_1$ from any chamber $C_{i+1},\ldots,C_j$.
\end{lemma}
\begin{proof}
We denote $\alpha=\u{\gamma}^{\to j}$ and $L=L_i=L_j$. We have $\alpha>0$ and $L=H_\alpha$.
Note that $C_1$ is the fundamental chamber.

Suppose first that $C_1\notsim_L C_q$ for some $q\le i$.
As $\u{\gamma}^{\to i}=-\alpha<0$, we get $C_1\sim_L C_i$
by Lemma~\ref{lemma:to}.
Therefore, $C_q\notsim_L C_i$, whence $q<i$.
By Lemma~\ref{lemma:20}, there exists an index $k$
such that $q\le k<i$ and $L_k=L$ and $C_k\notsim_L C_i$.
From the first equality it follows that $\u{\gamma}^{\to k}=\pm\alpha$.
To determine the sing, note that $C_1\notsim_L C_k$.
Therefore, $\u{\gamma}^{\to k}>0$ by Lemma~\ref{lemma:to}.
Hence $\u{\gamma}^{\to k}=\alpha=\u{\gamma}^{\to j}$,
which 
violates part~\ref{definition:cycb6:3:p:ii} of Definition~\ref{definition:cycb6:3}.

Suppose now that $C_1\sim_L C_q$ for some $q$ such that $i<q\le j$.
As $\u{\gamma}^{\to j}=\alpha>0$, we get $C_1\notsim_L C_j$ by Lemma~\ref{lemma:to}.
Therefore, $C_q\notsim_L C_j$, whence $q<j$.
By Lemma~\ref{lemma:20}, there exists an index $k$ such that $q\le k<j$, $L_k=L$ and $C_q\sim_LC_k$.
It follows from the first equality that $\u{\gamma}^{\to k}=\pm\alpha$.
To determine the sing, note that $C_1\sim_L C_k$.
Therefore, $\u{\gamma}^{\to k}<0$ by Lemma~\ref{lemma:to}.
Hence $\u{\gamma}^{\to k}=-\alpha=\u{\gamma}^{\to i}$,
which 
violates part~\ref{definition:cycb6:3:p:iii} of Definition~\ref{definition:cycb6:3}.
\end{proof}

\subsection{Categories}\label{Categories} Let $\D$ be a dihedral reflection subgroup of $W$
and $\Pi_\D=\{\alpha,\beta\}$ be the canonical simple system for $\D$.
It is easy to check that $\alpha$ and $\beta$ are linearly independent.
We consider the subspace $\mathscr V=\R\alpha\oplus\R\beta\subset V$.
This space is $\D$-invariant. Therefore, there is the restriction homomorphism $\rho:\D\to\GL(\mathscr V)$,
which happens to be injective. Its image 
is generated by the reflections
$A=\mathcal r_\alpha$ and $B=\mathcal r_\beta$, where $\mathcal r_v$ denotes
the reflection of $\mathscr V$ with respect to a unit vector $v\in\mathscr V$.
Here the bilinear form on $\mathscr V$ is the restriction of that on $V$.
Therefore, $\mathcal r_v$ is the restriction of $\r_v$ to $\mathscr V$.
We will identify $\D\cong\im\rho$.
Let
\begin{itemize}
\itemsep=3pt
\item $\u{s}$ be an expression in $A$ and $B$; 
\item $\u{t}$ be an expression in $A$ and $B$ or in $S$; 
\item $p:\{1,\ldots,|\u{s}|\}\to\{1,\ldots,|\u{t}|\}$ be an increasing embedding.
\end{itemize}

A pair $(\u{\ggamma},\u{\ddelta})$ of subexpressions of $\u{s}$ and $\u{t}$ respectively
is a called a {\it $p$-pair} if
for any $i=1,\ldots,|\u{s}|$, there exists $\epsilon_i=\pm1$ such that
\begin{equation}\label{eq:cyb1:cat:3}
\u{\ddelta}^{\to p(i)}=\epsilon_i\u{\ggamma}^{\to i}.
\end{equation}
We will call the sequence $\u{\epsilon}=(\epsilon_1,\ldots,\epsilon_{|\u{s}|})$ the {\it cosign}
of $(\u{\ggamma},\u{\ddelta})$. If $\u{\epsilon}=(1,\ldots,1)$, then we say that $(\u{\ggamma},\u{\ddelta})$
is a $p$-pair of {\it positive} cosign. Note that the notion of the {\it sign} also makes sense (see Remark~\ref{rm:leftarrow})
and can be developed as well as the notion of the cosign. We do not, however, need it for the proofs of the main results of this paper.

\begin{lemma}\label{lemma:cat:1}
If $(\u{\ggamma},\u{\ddelta})$ is a $p$-pair, then $(\f_k\u{\ggamma},\f_{p(k)}\u{\ddelta})$ is
a $p$-pair of the same cosign for any $k=1,\ldots,|\u{\ggamma}|$.
\end{lemma}
\noindent
{\it Proof.}
Let $\u{\ggamma}\subset\u{s}$ and $\u{\ddelta}\subset\u{t}$ and $\u{\epsilon}$
be the cosign of $(\u{\ggamma},\u{\ddelta})$.
Let $\nu=1$ if $k<i$ (equivalently $p(k)<p(i)$) and $\nu=0$ otherwise.
Applying~(\ref{eq:cyb1:cat:5}) and~(\ref{eq:cyb1:cat:3}), we get
$$
\hspace{90pt}
(\f_{p(k)}\u{\ddelta})^{\to p(i)}=\r_{\u{\ddelta}^{\to p(k)}}^\nu\u{\ddelta}^{\to p(i)}
=\epsilon_i\r_{\u{\ggamma}^{\to k}}^\nu\u{\ggamma}^{\to i}
=\epsilon_i(\f_k\u{\ggamma})^{\to i}.
\hspace{90pt}\qed
$$

Let $\Expr_\D(W,S)$ be the category whose objects are all expressions in $\a$ and $\b$
and all expressions in $S$. Its morphisms are described as follows. Let $\u{s}$ be an expression in $\a$ and $\b$
and $\u{t}$ be an expression in $\a$ and $\b$ or in $S$.
A morphism from $\u{s}$ to $\u{t}$ is a pair $(p,\phi)$, where $p:\{1,\ldots,|\u{s}|\}\to\{1,\ldots,|\u{t}|\}$
is an increasing embedding and $\phi$ is a map $\SubEx(\u{s})\to\SubEx(\u{t})$
such that $(\u{\ggamma},\phi(\u{\ggamma}))$ is a $p$-pair and $\phi(\f_i\u{\ggamma})=\f_{p(i)}\phi(\u{\ggamma})$ for any $\u{\ggamma}\subset\u{s}$ and $i=1,\ldots,|\u{s}|$.

By Lemma~\ref{lemma:cat:1}, it suffices that $(\u{\ggamma},\phi(\u{\ggamma}))$ be a $p$-pair
just for one subexpression $\u{\ggamma}$. By the same lemma, the cosign of this pair does not depend on $\u{\ggamma}$.
We will call this cosign the {\it cosign of the morphism} $(p,\phi)$.

The composition of morphisms $(p,\phi):\u{s}\to\u{t}$ and $(q,\psi):\u{t}\to\u{r}$ is given by the rule
$$
(q,\psi)\circ(p,\phi)=(q\circ p,\psi\circ\phi).
$$
It is a well-defined morphism. Indeed, denoting by $\u{\epsilon}$ and $\u{\tau}$ the cosigns of our morphisms, we get
$$
\u{\ggamma}^{\to i}=\epsilon_i\phi(\u{\ggamma})^{\to p(i)}=\epsilon_i\tau_{p(i)}\,\psi\circ\phi(\u{\ggamma})^{\to q\circ p(i)}
,\;\;\;
\psi\circ\phi(\f_i(\u{\ggamma}))=\psi(\f_{p(i)}\phi(\u{\ggamma}))=\f_{q\circ p(i)}(\psi\circ\phi(\u{\ggamma}))
$$
for any $\u{\ggamma}\subset\u{s}$ and $i=1,\ldots,|\u{s}|$. The cosign of this composition is
$(\epsilon_1\tau_{p(1)},\epsilon_2\tau_{p(2)},\ldots,\epsilon_{|\u{s}|}\tau_{p(|\u{s}|)})$.
Hence we get the following result, which we will use to prove Theorem~\ref{theorem:main}.

\begin{proposition}\label{proposition:8}
The composition of morphisms in $\Expr_\D(W,S)$ of positive cosign is a morphism of positive cosign.
\end{proposition}

\begin{lemma}\label{lemma:17}
Let $(p,\phi):\u{s}\to\u{t}$ be a morphism in $\Expr_\D(W,S)$ and $1\le i<j\le|s|$.\linebreak
The operator $\f_{i,j}$ is applicable to $\u{\ggamma}\subset\u{s}$
if and only if $\f_{p(i),p(j)}$ is applicable to $\phi(\u{\ggamma})$. In that case,
$\f_{p(i),p(j)}\phi(\u{\ggamma})=\phi(\f_{i,j}\u{\ggamma})$.
\end{lemma}
\begin{proof}
We need to prove only the applicability claim, which in view of Proposition~\ref{proposition:applicable}
can be formulated as the following equivalence:
$$
\u{\ggamma}^{\to i}=\pm\u{\ggamma}^{\to j}\Longleftrightarrow\phi(\u{\ggamma})^{\to p(i)}=\pm\phi(\u{\ggamma})^{\to p(j)}.
$$
But this fact simply follows from the fact that $(\u{\ggamma},\phi(\u{\ggamma}))$ is a $p$-pair.
\end{proof}

\begin{lemma}\label{lemma:21}
Let $(p,\phi)$ be a morphism in $\Expr_\D(W,S)$ of positive cosign. Then $\phi$ preserves the order $<$.
\end{lemma}
\begin{proof} Let $(p,\phi)$ be a morphism from $\u{s}$ to $\u{t}$. Let $\u{\delta}<\u{\gamma}$
for some $\u{\delta},\u{\gamma}\subset\u{s}$. In this case $\u{\delta},\u{\gamma}\in\SubEx(\u{s},w)$
for the corresponding $w\in W$.
By Theorem~\ref{theorem:connected}, we get
$
\displaystyle
\u{\gamma}=\f_{i_1,j_1}\cdots\f_{i_h,j_h}\u{\delta}
$
for some indices $i_1,\ldots,i_h$ and $j_1,\ldots,j_h$. Applying $\phi$, we get
$\displaystyle
\phi(\u{\gamma})=\f_{p(i_1),p(j_1)}\cdots\f_{p(i_h),p(j_h)}\phi(\u{\delta}).
$ by Lemma~\ref{lemma:17}.
Hence $\phi(\u{\delta})^{\max}=\phi(\u{\gamma})^{\max}$.

%
%
Let $\u{\delta}$ and $\u{\gamma}$ differ maximally at $j$. Then
$\displaystyle
\u{\gamma}=\f_{k_1}\cdots\f_{k_m}\f_j\u{\delta}
$
for some $k_1,\ldots,k_m<j$.
Applying $\phi$, we get
$\displaystyle
\phi(\u{\gamma})=\f_{p(k_1)}\cdots\f_{p(k_m)}\f_{p(j)}\phi(\u{\delta}).
$
Hence $\phi(\u{\gamma})$ and $\phi(\u{\delta})$ differ maximally at $p(j)$.
As the cosign of $(p,\phi)$ is positive,
$\phi(\u{\gamma})^{\to p(j)}=\u{\gamma}^{\to j}>0$, whence $\phi(\u{\delta})<\phi(\u{\gamma})$. 
\end{proof}

Let $\Expr_\D$ be the full subcategory of $\Expr_\D(W,S)$ whose objects are expressions in $A$ and $B$.
Note that $W$ and $S$ do not entre in the notation of this subcategory, as it depends only on $\D$.

\subsection{Projection}\label{Projection}
Let us stick to the notation of the previous section and let $\T$ be the set of reflections of $\D$.
Note that $\T=\D\cap T$.
For $t\in\T$, we denote the $t$-wall in $\mathscr V^*$ by $\H_t$. Similarly, let $\mathscr C$ denote
the fundamental chamber in $\mathscr V^*$.

Let $\iota:\mathscr V\to V$ be the natural embedding.
As this map is $\D$-equivariant, the dual map
$\iota^*:V^*\twoheadrightarrow\mathscr V^*$ is also $\D$-equivariant. It is continuous being linear.
We are going to see what happens to walls and chambers if we apply $\iota^*$ to them.

\begin{lemma}\label{lemma:7}
Let $t\in\T$. Then $\iota^*(H_t)=\H_t$.
\end{lemma}
\begin{proof}
Obviously, $\iota^*(H_t)\subset\H_t$.
Let $f\in\H_t$. As $\iota^*$ is epimorphic, we get $f=\iota^*(g)$ for some $g\in V^*$.
Let $h=(g+tg)/2$. Then $h\in H_t$ and $\iota^*(h)=(\iota^*(g)+t\iota^*(g))/2=(f+tf)/2=f$.
\end{proof}


\begin{lemma}\label{lemma:8}
Let $s\in S$ and $w\in W$. Then $0\notin\iota^*(wC_s)$.
\end{lemma}
\begin{proof}
Suppose that $\iota^*(wc)=0$ for some $c\in C_s$. Then we get $wc(\alpha)=0$ and $wc(\beta)=0$.
Therefore, $c(w^{-1}\alpha)=0$ and $c(w^{-1}\beta)=0$.
These equalities can only hold if $w^{-1}\alpha=\pm\v_s$ and $w^{-1}\beta=\pm\v_s$.
Hence $\alpha$ and $\beta$ are proportional, which is a contradiction.
\end{proof}


\begin{lemma}\label{lemma:5}
Let $\iota^*(wC)\subset d\mathscr C$, where $w\in W$ and $d\in\D$. Let $s\in S$
and $t=wsw^{-1}$. 
{\renewcommand{\labelenumi}{{\it(\roman{enumi})}}
\renewcommand{\theenumi}{{\rm(\roman{enumi})}}
\begin{enumerate}
\setlength{\itemindent}{-10pt}
\item\label{lemma:5:p:1} If $t\notin\D$, then $\iota^*(wsC)\subset d\mathscr C$.\\[-8pt]
\item\label{lemma:5:p:2} If $t\in\D$, then $\iota^*(wsC)\subset td\mathscr C$ and the chambres
                         $d\mathscr C$ and $td\mathscr C$ are connected through~$\H_t$.
\end{enumerate}}
\end{lemma}
\begin{proof}
First let us prove the equavalence
\begin{equation}\label{eq:23}
\iota^*(wC)\subset d\mathscr C\Leftrightarrow w^{-1}d\alpha>0, w^{-1}d\beta>0.
\end{equation}
Indeed, by $\D$-equivariance of $\iota^*$, we get
\begin{multline*}
\iota^*(wC)\subset d\mathscr C\Leftrightarrow \iota^*(d^{-1}wC)\subset\mathscr C\Leftrightarrow\forall c\in C:\; d^{-1}wc(\alpha)>0, d^{-1}wc(\beta)>0\\
\Leftrightarrow\forall c\in C:\; c(w^{-1}d\alpha)>0, c(w^{-1}d\beta)>0\Leftrightarrow w^{-1}d\alpha>0,w^{-1}d\beta>0.
\end{multline*}

\ref{lemma:5:p:1} We claim that $w^{-1}d\alpha$ and $w^{-1}d\beta$ are positive roots distinct from $\v_s$.
The positivity part follows from~(\ref{eq:23}). Suppose that $w^{-1}d\alpha=\v_s$. Thus
$$
s=\t_{\v_s}=\t_{w^{-1}d\alpha}=w^{-1}d\,\t_\alpha(w^{-1}d)^{-1}=w^{-1}d\,\t_\alpha d^{-1}w.
$$
Hence, we get a contradiction $t=wsw^{-1}=d\,\t_\alpha d^{-1}\in\D$. Similarly, $w^{-1}d\beta\ne\v_s$.
By~\cite[Lemma 3.4(a)]{BD}, $sw^{-1}d\alpha>0$ and $sw^{-1}d\beta>0$.
Therefore, $\iota^*(wsC)\subset d\mathscr C$ by~(\ref{eq:23}).

\ref{lemma:5:p:2}
As
$$
(ws)^{-1}td\alpha=w^{-1}d\alpha>0,\quad (ws)^{-1}td\beta=w^{-1}d\beta>0,
$$
we get by~(\ref{eq:23}) that $\iota^*(wsC)\subset td\mathscr C$.

Let us prove the second claim. Let $c\in C_s$. Then $c\in\overline{C}\cap s\overline{C}\cap H_s$.
Therefore, $wc\in w\overline{C}\cap ws\overline{C}\cap H_t$.
By part~\ref{lemma:5:p:1}, 
Lemmas~\ref{lemma:7}, \ref{lemma:8} and continuity of $\iota^*$, we get
$$
0\ne\iota^*(wc)\in\iota^*(w\overline{C})\cap\iota^*(ws\overline{C})\cap\iota^*(H_t)\subset d\,\overline{\mathscr C}\cap td\,\overline{\mathscr C}\cap\H_t.
$$
Now the result follows from Remark~\ref{rem:4}.
\end{proof}

\begin{corollary}\label{corollary:4}
For any $w\in W$, there exists a unique $d\in\D$ such that $\iota^*(wC)\subset d\mathscr C$.
\end{corollary}
\begin{proof}
The result follows by induction on $\l(w)$ from $\iota^*(C)\subset\mathscr C$ and Lemma~\ref{lemma:5}.
\end{proof}
\noindent
In the notation of this corollary, we denote $\pr_\D(wC)=d\mathscr C$ and $\pr_\D(w)=d$.
We have the following characteristic property:
$$
\iota^*(wC)\subset\pr_\D(wC)=\pr_\D(w)\mathscr C.
$$
Thus we get $\pr_\D(C)\,{=}\,\mathscr C$.\!
The projection $\pr_\D(w)$ can be computed inductively:
$$
\pr_\D(1)=1,\qquad
\pr_\D(ws)=\left\{\!
\begin{array}{ll}
\pr_\D(w)&\text{ if }wsw^{-1}\notin\D;\\
wsw^{-1}\pr_\D(w)&\text{ if }wsw^{-1}\in\D.
\end{array}
\right.
$$

\begin{corollary}\label{corollary:5}
Let $wC$ and $w'C$ be connected through $H_t$, where $w,w'\in W$ and $t\in T$.
If $t\notin\D$, then $\pr_\D(wC)=\pr_\D(w'C)$. Otherwise $\pr_\D(wC)$
and
$\pr_\D(w'C)$
are connected through $\H_t$ and are distinct if and only if $wC\ne w'C$.
\end{corollary}
\begin{proof} By Proposition~\ref{proposition:2}, there exists $s\in S$
such that $t=wsw^{-1}$ and either $w'=w$ or $w'=ws$.
If $t\notin\D$, then $\pr_\D(wC)=\pr_\D(w'C)$ 
trivially or by Lemma~\ref{lemma:5}\ref{lemma:5:p:1}.

Suppose now that $t\in\D$. If $w'=ws$, then
$\pr_\D(wC)$ 
and
$\pr_\D(w'C)$ 
are distinct and connected through $\H_t$
by Lemma~\ref{lemma:5}\ref{lemma:5:p:2}. The case $w=w'$ follows from Remark~\ref{rem:1}.
\end{proof}

The following definition is a version of projection of galleries to root subsystems
described in~\cite{subroot} for the specific situation of a dihedral reflection subgroup
considered in this paper.

\begin{definition}\label{definition:1}
Let $\u{\Gamma}=(C_1,L_1,\ldots,C_m,L_m,C_{m+1})$ be a gallery in $V^*$.
Let $\u{\Gamma}(\D)=\{p(1)<\cdots<p(k)\}$ be the set of all indices $i=1,\ldots,m$ such that $L_i=H_t$ for some $t\in\T$.
We denote
$$
\pr_\D(\u{\Gamma})=(\pr_\D(C_{p(1)}),\iota^*(L_{p(1)}),\ldots,\pr_\D(C_{p(k)}),\iota^*(L_{p(k)}),\pr_\D(C_{m+1})).
$$
\end{definition}

By Corollary~\ref{corollary:5}, we get that $\pr_\D(\u{\Gamma})$ is a well-defined gallery in $\mathscr V^*$.
This fact is similar to~\cite[Proposition~14]{subroot}.
Note that if $\u{\Gamma}$ starts at the fundamental chamber,
then
$$
\u{\Gamma}(\D)=\{i=1,\ldots,m\suchthat\u{\gamma}^{\to i}\in \D\{\alpha,\beta\}\}.
$$

\begin{theorem}\label{lemma:15}
Let $\u{\Gamma}$ be a gallery in $V^*$ starting at $C$.
Then the gallery $\u{\Delta}=\pr_\D(\u{\Gamma})$ starts at $\C$.
Let $p:\{1,\ldots,|\u{\Delta}|\}\ito\u{\Gamma}(\D)$ be the increasing bijection.
Then $(\u{\delta},\u{\gamma})$ is a $p$-pair of positive cosign.
Moreover, there exists a unique map $\phi:\SubEx(\u{t})\to\SubEx(\u{s})$, where $\u{\delta}\subset\u{t}$ and $\u{\gamma}\subset\u{s}$,
such that $\phi(\u{\delta})=\u{\gamma}$ and the pair $(\phi,p)$ is a morphism $\u{t}\to\u{s}$
of positive cosign in $\Expr_\D(W,S)$.
\end{theorem}
\begin{proof}
The first claim follows from $\pr_\D(C)=\mathscr C$. Comparing the $i$th walls of $\pr_\D(\u{\Gamma})$ and $\u{\Delta}$,
we get by Lemma~\ref{lemma:7} that
$$
\iota^*\Big(H_{\u{\gamma}^{\to p(i)}}\Big)=\H_{\u{\gamma}^{\to p(i)}}=\H_{\u{\delta}^{\to i}}.
$$
Hence $\u{\gamma}^{\to p(i)}=\epsilon\u{\delta}^{\to i}$ for $\epsilon=\pm1$.
Comparing the $i$th chambers of $\pr_\D(\u{\Gamma})$ and $\u{\Delta}$, we get
$$
\iota^*(\u{\gamma}^{<p(i)}C)\subset\pr_\D(\u{\gamma}^{<p(i)}C)=\u{\delta}^{<i}\C.
$$
Let $c$ be an arbitrary element of $C$. As $\iota^*$ is $\D$-equivariant,
we get $\iota^*((\u{\delta}^{<i})^{-1}\gamma^{<p(i)}c)\in\C$.
Assuming that $\u{\delta}\subset\u{t}$ and $\u{\gamma}\subset\u{s}$ for the corresponding expressions, we conclude that
\begin{multline*}
0<\iota^*((\u{\delta}^{<i})^{-1}\u{\gamma}^{<p(i)}c)(\v_{t_i})
=((\u{\delta}^{<i})^{-1}\u{\gamma}^{<p(i)}c)(\v_{t_i})\\
=-c((\gamma^{<p(i)})^{-1}\u{\delta}^{\to i})
=-\epsilon c((\gamma^{<p(i)})^{-1}\u{\gamma}^{\to p(i)})=\epsilon c(\v_{s_{p(i)}}).
\end{multline*}
As $c(\v_{s_{p(i)}})>0$, we get $\epsilon=1$.

Finally, to get the required map $\phi$, we continue the map $\u{\delta}\mapsto\u{\gamma}$
to a map $\phi:\SubExpr(\u{t})\to\SubExpr(\u{s})$, using the rule
$
\displaystyle\phi(\f_{k_1}\cdots \f_{k_m}\u{\delta})=\f_{p(k_1)}\cdots\f_{p(k_m)}\u{\gamma}
$
by Lemma~\ref{lemma:cat:1}
\end{proof}
\noindent
In the above lemma, $\u{t}$ is an expression in $A$ and $B$ and $\u{s}$ is an expression in $S$.

\section{Dihedral cycles}\label{Dihedral_cycles}

Throughout this section, $(\D,\{A,B\})$ is the Coxeter system associated to
a based root system $(\Phi_\D,\{\alpha,\beta\})$ in a two-dimensional real vector space $\mathscr V\cong\R\alpha\oplus\R\beta$.
Let $n=\ord(AB)$. Thus $A=\mathcal r_\alpha$ and $B=\mathcal r_\beta$ are reflections of $\mathscr V$
and $n\ge2$. We will also use the letter $\mathfrak c$ as a wildcard to denote either $A$ and $B$.
Once the $\mathfrak c$ is chosen, $\bar{\mathfrak c}$ denotes $A$ or $B$ so that $\bar{\mathfrak c}\ne\mathfrak c$.
We will use the category $\Expr_\D$ introduced at the end of Section~\ref{Categories} and
use curly letters to denote chambers, walls and panels in $\mathscr V^*$ as stipulated in Section~\ref{Projection}.

\subsection{Galleries as alcove walks}\label{Galleries_as_paths}
We are going to describe the Coxeter complex $\Sigma(\D,\{A,B\})$.
Following the introduction, we denote by $\Sigma_k(\D,\{A,B\})$ the union of all simplices of codimension $k$.
In the present case, $\Sigma_k(\D,\{A,B\})=\emptyset$ for $k\ge2$.
%
A chamber is said to be {\it attached} to a panel if its closure contains this panel.

First suppose that $n<\infty$. Then there exists a $\D$-equivariant isomorphism $\E\cong\E^*$
given by $e\mapsto e^\vee$, where $e^\vee(e')=(e|e')$ for any $e'\in\E$.
Therefore, we can describe the fundamental chamber and the walls as follows:
$$
\C=\{e\in\E\suchthat (e|\alpha)>0,(e|\beta)>0\},\quad
\H_\tau=\{e\in\E\suchthat (e|\tau)=0\}.
$$
Following the definitions at the end of Section~\ref{Walls, chambers and panels}, we get
$$
\Sigma_0(\D,\{A,B\})\cong\{e\in\E\suchthat\sqrt{(e|e)}=n/\pi\},\quad
\Sigma_1(\D,\{A,B\})\cong\Sigma_0(\D,\{A,B\})\cap\bigcup_{t\in\T}\H_t.
$$
The former set is just the circle $S_{n/\pi}^1$ 
and the latter set, which we denote by $P_n$,
consists of the vertices of a regular $2n$-gone.
The panels are one-point subsets of $P_n$ and the walls are unions of diametrically opposite panels.
The chambers are open arcs $\arc{pq}^\circ$ for adjacent (having distance $1$ along the circle) panels.
We will call these chambers {\it arc chambers}.

If $(\alpha|\beta)=-1$, then
$$
\Sigma_0(\D,\{A,B\})\cong\{x\alpha^*+y\beta^*\suchthat x+y=1\}
$$
and $\Sigma_1(\D,\{A,B\})$ is identified with its subset with integer coordinates $x$ and $y$.
The panels and the walls are one-point subsets of the latter set and the chambers are
open line segments between adjacent panels.

Finally, suppose that $(\alpha|\beta)<-1$.
Let $\xi=-(\alpha|\beta)$.
We define a two-sided sequence $\{\gamma^*_i\}_{i\in\Z}$ of vectors of $\mathscr V^*$ by the rules:
$$
\gamma^*_0=\alpha^*,\quad \gamma^*_1=\beta^*,\quad \forall i\in\Z:\gamma^*_{i-1}+\gamma^*_{i+1}=2\xi\gamma^*_i.
$$
Then we get
$$
\begin{array}{c}
\Sigma_0(\D,\{A,B\})\cong\{x\alpha^*+y\beta^*\in\E^*\suchthat x^2+y^2+2\xi xy=1,x+y>0\},\\[6pt]
\Sigma_1(\D,\{A,B\})\cong\{\gamma^*_i\suchthat i\in\Z\}.
\end{array}
$$
As in the previous case, the panels and the walls are one-point subsets of the latter set and the chambers are
open hyperbola segments between adjacent panels.
In all considered above cases, we fix the fundamental chamber $\mathscr F=\Sigma_0(\D,\{A,B\})\cap\mathscr C$.

\begin{definition}\label{definition:cycb6:1}
An alcove walk is a continuous map $g:[a,b]\to\Sigma_0(\D,\{A,B\})$ such that
the set $g^{-1}(\Sigma_1(\D,\{A,B\}))$ is finite and
contains neither $a$ nor $b$.
Let $g^{-1}(\Sigma_1(\D,\{A,B\}))=\{x_1<\cdots<x_m\}$. The gallery generated by $g$ is the following gallery in $\mathscr V^*$:
$$
\u{\Gamma}=(\mathscr C_1,\mathscr L_1,\mathscr C_2,\mathscr L_2,\ldots,\mathscr C_m,\mathscr L_m,\mathscr C_{m+1})
$$
such that $g(x_i)\in\mathscr L_i$ for any $i=1,\ldots,m$ and $g((x_{i-1},x_i))\subset\mathscr C_i$ for any $i=1,\ldots,m+1$,
where $x_0=a$ and $x_{m+1}=b$.
\end{definition}

Clearly, to any alcove walk, there corresponds a unique gallery in $\mathscr V^*$.
Vice versa, any gallery in $\mathscr V^*$ is generated by an alcove walk.
The following result easily follows from the definitions, see Section~\ref{Circle} and
Definition~\ref{def:gallery_double_fold}.

\begin{proposition}\label{proposition:cycb6:4}
Let $n=\ord(AB)$ be finite, $\u{\Gamma}$ be the gallery corresponding to an alcove walk $g$ and
$g^{-1}(P_n)=\{x_1<\cdots<x_m\}$. Then for any indices
$1\le i<j\le m$ such that
the points $g(x_i)$ and $g(x_j)$ belong to the same line through the center of $S_{n/\pi}^1$,
the gallery $\f_{i,j}\u{\Gamma}$ corresponds to the alcove walk $\f_{x_i,x_j}g$.
If the last path is a double fold of $g$ towards the fundamental chamber,
then $\f_{i,j}\u{\Gamma}$ is a double fold of $\u{\Gamma}$ also towards the fundamental chamber.
\end{proposition}

\subsection{Minimal paths}
For a point $x\in\R$, we say that a property $\Theta(t)$ holds as $t\to x^-$ (as $t$ approaches $x$ from the left)
if there exists a number $\epsilon>0$ such that $\Theta(t)$ holds (and thus is defined) for any $t\in(x-\epsilon,x)$.
To formulate the following result, we need the separation relation on a circle, see Section~\ref{Circle}.

\begin{lemma}\label{corollary:cycl:2}
Let $X,Y,Z$ be distinct points of a circle $S_r^1$, $P\subset S_r^1$ and $a,b,c,d\in\R$ be such that $a<b$ and $c<d$.
Let $f:[a,b]\to S_r^1\setminus\{Z\}$ and $g:[c,d]\to S_r^1\setminus\{Z\}$ be continuous functions
such that $g$ is monotonic and
\begin{enumerate}
\itemsep=5pt
\item\label{corollary:cycl:2:p:1} $f^{-1}(P)$ is finite;
\item\label{corollary:cycl:2:p:2} $X=f(a)=g(c)$ and $Y=f(b)=g(d)$;
\item\label{corollary:cycl:2:p:3} $Y\in P$;
\item\label{corollary:cycl:2:p:4} $XYf(t)Z$ for no $t\in[a,b]$.
\end{enumerate}
There exists an increasing map $q:g^{-1}(P)\to f^{-1}(P)$
such that
{\renewcommand{\labelenumi}{{\it(\roman{enumi})}}
\renewcommand{\theenumi}{{\rm(\roman{enumi})}}
\begin{enumerate}
\itemsep=5pt
\item\label{corollary:cycl:2:p:i} $f\circ q(t)=g(t)$ for any $t\in g^{-1}(P)$;
\item\label{corollary:cycl:2:p:ii} $Xf(t')g(t)Z$ as $t'\to q(t)^-$ for any $t\in  g^{-1}(P)\setminus\{c\}$;
\item\label{corollary:cycl:2:p:iii} $q(d)=b$;
\item\label{corollary:cycl:2:p:iv}  $X\in P\Rightarrow q(c)=a$.
\end{enumerate}}
\end{lemma}
\begin{proof}
Let us identify $S_r^1\setminus\{Z\}\cong\R$ by the stereographic projection from $Z$
in such a way that $g$ becomes increasing. Then $X<Y$,
the image of $g$ is the interval $[X,Y]$, \ref{corollary:cycl:2:p:4} means that $f(t)\le Y$ for any $t$
and $Xf(t')g(t)Z$ is equivalent to $f(t')<g(t)$.

First, we define a function $q':g^{-1}(P)\to f^{-1}(P)$ by the rule $q'(t)=\min f^{-1}(g(t))$.
By the intermediate value theorem, $q'$ is increasing and satisfies~\ref{corollary:cycl:2:p:i},
\ref{corollary:cycl:2:p:ii},~\ref{corollary:cycl:2:p:iv}.
To ensure~\ref{corollary:cycl:2:p:iii}, we define $q$ to be equal to $q'$
everywhere except $d$, where it is set to be equal to $b$.
\end{proof}

\begin{example}\label{example:5}
\rm Consider the following example (in the stereographic projection).

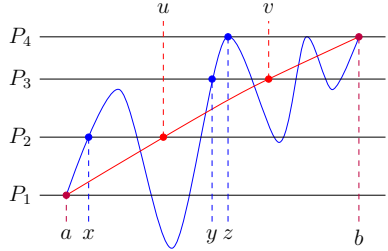
\def\uu{1pt}
\def\vv{0.06}

\begin{center}
\scalebox{0.7}{
\begin{tikzpicture}
\draw[blue] plot[smooth] coordinates {(0,0) (1,2) (2,-1) (3,2.98) (4,1) (4.5,2.99) (5,2) (5.5,3)};
\draw[red] plot[smooth] coordinates {(0,0) (2,1.2) (3.5,2.05) (5.5,3)};
\draw (-0.5,3)node[anchor=east,black]{$P_4$}--(6,3);
\draw (-0.5,0)node[anchor=east,black]{$P_1$}--(6,0);
\draw (-0.5,2.2)node[anchor=east,black]{$P_3$}--(6,2.2);
\draw (-0.5,1.1)node[anchor=east,black]{$P_2$}--(6,1.1);
\draw[purple,dashed] (0,-0.5)node[anchor=north,black]{$a$}--(0,0);
\draw[fill,purple] (0,0) circle(\vv);
\draw[blue,dashed] (0.42,-0.5)node[anchor=north,black]{$x$}--(0.42,1.1);
\draw[fill,blue] (0.42,1.1) circle(\vv);
\draw[blue,dashed] (2.74,-0.5)node[anchor=north,black]{$y$}--(2.74,2.2);
\draw[fill,blue] (2.74,2.2) circle(\vv);
\draw[purple,dashed] (5.5,-0.5)node[anchor=north,black]{$b$}--(5.5,3);
\draw[fill,purple] (5.5,3) circle(\vv);
\draw[red,dashed] (1.824,1.1)--(1.824,3.3)node[anchor=south,black]{$u$};
\draw[fill,red] (1.824,1.1) circle(\vv);
\draw[red,dashed] (3.804,2.2)--(3.804,3.3)node[anchor=south,black]{$v$};
\draw[fill,red] (3.804,2.2) circle(\vv);
\draw[blue,dashed] (3.04,-0.5)node[anchor=north,black]{$z$}--(3.04,3);
\draw[fill,blue] (3.04,3) circle(\vv);
\end{tikzpicture}}
\end{center}
\end{example}

\vspace{-10pt}

\noindent
Here $a=c$ and $b=d$,
$P\cap[X,Y]=\{P_1,P_2,P_3,P_4\}$,
the blue and red solid lines represent the graph of $f$ and $g$ respectively,
$g^{-1}(P)=\{a,u,v,b\}$ and $q'$ and $q$ are given by $a\mapsto a$, $u\mapsto x$, $v\mapsto y$, $b\mapsto z$
and $a\mapsto a$, $u\mapsto x$, $v\mapsto y$, $b\mapsto b$ respectively.

\subsection{Diagrams}\label{Diagrams} In this paper, we use plane diagrams that are finite horizontal compositions
$D=L_p(a)\cup M_{r_1}(b_1)\cup\cdots\cup M_{r_k}(b_k)\cup R_q(c)$ of the following diagrams
(where the numbers mean ordinates, the widths of all strips equal $1$ and the red lines are tilted $45$ degrees):

\medskip

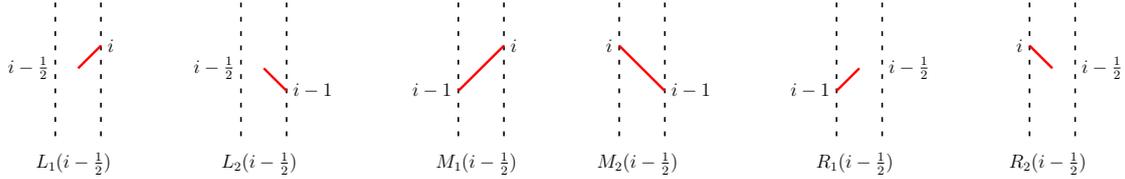
\def\sl{6pt}

\def\tt{1.5pt}
\def\st{1pt}

\begin{center}

\scalebox{0.6}{
\begin{tikzpicture}
\draw[loosely dashed,line width=\st] (0,-2)--(0,1);
\draw[loosely dashed,line width=\st] (1,-2)--(1,1);
\draw[line width=\tt,red] (0.5,-0.5)--(1,0);
\draw (1,0)node[anchor=west]{$i$};
\draw (0,-0.5)node[anchor=east]{$i-\frac12$};
\draw (0.4,-2.6) node{$L_1(i-\frac12)$};
\end{tikzpicture}}
\quad\;
\scalebox{0.6}{
\begin{tikzpicture}
\draw[loosely dashed,line width=\st] (0,-2)--(0,1);
\draw[loosely dashed,line width=\st] (1,-2)--(1,1);
\draw[line width=\tt,red] (0.5,-0.5)--(1,-1);
\draw (1,-1)node[anchor=west]{$i-1$};
\draw (0,-0.5)node[anchor=east]{$i-\frac12$};
\draw (0.4,-2.6) node{$L_2(i-\frac12)$};
\end{tikzpicture}}
\quad\;
\scalebox{0.6}{
\begin{tikzpicture}
\draw[loosely dashed,line width=\st] (0,-2)--(0,1);
\draw[loosely dashed,line width=\st] (1,-2)--(1,1);
\draw[line width=\tt,red] (0,-1)--(1,0);
\draw (1,0)node[anchor=west]{$i$};
\draw (0,-1)node[anchor=east]{$i-1$};
\draw (0.4,-2.6) node{$M_1(i-\frac12)$};
\end{tikzpicture}}
\quad\;
\scalebox{0.6}{
\begin{tikzpicture}
\draw[loosely dashed,line width=\st] (0,-2)--(0,1);
\draw[loosely dashed,line width=\st] (1,-2)--(1,1);
\draw[line width=\tt,red] (1,-1)--(0,0);
\draw (1,-1)node[anchor=west]{$i-1$};
\draw (0,0)node[anchor=east]{$i$};
\draw (0.4,-2.6) node{$M_2(i-\frac12)$};
\end{tikzpicture}}
\quad\;
\scalebox{0.6}{
\begin{tikzpicture}
\draw[loosely dashed,line width=\st] (0,-2)--(0,1);
\draw[loosely dashed,line width=\st] (1,-2)--(1,1);
\draw[line width=\tt,red] (0.5,-0.5)--(0,-1);
\draw (1,-0.5)node[anchor=west]{$i-\frac12$};
\draw (0,-1)node[anchor=east]{$i-1$};
\draw (0.4,-2.6) node{$R_1(i-\frac12)$};
\end{tikzpicture}}
\quad\;
\scalebox{0.6}{
\begin{tikzpicture}
\draw[loosely dashed,line width=\st] (0,-2)--(0,1);
\draw[loosely dashed,line width=\st] (1,-2)--(1,1);
\draw[line width=\tt,red] (0.5,-0.5)--(0,0);
\draw (1,-0.5)node[anchor=west]{$i-\frac12$};
\draw (0,0)node[anchor=east]{$i$};
\draw (0.4,-2.6) node{$R_2(i-\frac12)$};
\end{tikzpicture}}
\end{center}

\vspace{-2pt}

\noindent
where $i\in\Z$, so that the red segments add up to continuous curves (in the metric topology).
Note that the case $k=0$ is possible. We say that $D$ {\it starts at ordinate} $a$.

We will consider all diagrams up to horizontal shifts thus ignoring the abscissas
and extending them infinitely to the right and to the left.
The {\it width} of a diagram is the length of the horizontal projection of its curve
(that is, its length divided by $\sqrt2$).

The horizontal line in such a diagram consisting of points with ordinate $i\in\Z$
is called the {\it $i$-wall} or simply a {\it wall}.
The spaces between adjacent walls are called {\it chambers}.
The camber between the $0$-wall and the $1$-wall is called the {\it fundamental chamber}.
It is usually shaded in our pictures.
Let us consider the following moves:

\def\tt{1.5pt}
\def\st{1pt}


\vspace{-7pt}

\begin{equation}\label{upper_move}
\scalebox{0.6}{
\begin{tikzpicture}[baseline=-33pt]
\draw[line width=\st] (-0.2,0)node[anchor=east]{$i+1$}--(7.6,0);
\draw[line width=\st] (-0.2,-1)node[anchor=east]{$1\le i$}--(7.6,-1);
\draw[line width=\st] (-0.2,-2)node[anchor=east]{$i-1$}--(7.6,-2);
\draw[line width=\tt,red] (0,-1)--(1,0) -- (2,-1) -- (3,0)--(3.3,-0.3);
\draw[line width=\tt,red] (5.1,-0.7)--(5.4,-1)--(6.4,0)--(7.4,-1);
\draw[fill,red] (3.8,-0.5) circle(0.02);
\draw[fill,red] (4.2,-0.5) circle(0.02);
\draw[fill,red] (4.6,-0.5) circle(0.02);
\end{tikzpicture}}
\quad
\to
\quad
\scalebox{0.6}{
\begin{tikzpicture}[baseline=-33pt]
\draw[line width=\st] (-0.2,0)--(7.6,0);
\draw[line width=\st] (-0.2,-1)--(7.6,-1);
\draw[line width=\st] (-0.2,-2)--(7.6,-2);
\draw[line width=\tt,red] (0,-1)--(1,-2) -- (2,-1) -- (3,-2)-- (3.3,-1.7);
\draw[fill,red] (3.8,-1.5) circle(0.02);
\draw[fill,red] (4.2,-1.5) circle(0.02);
\draw[fill,red] (4.6,-1.5) circle(0.02);
\draw[line width=\tt,red] (5.1,-1.3)--(5.4,-1)--(6.4,-2)--(7.4,-1) ;
\end{tikzpicture}}
\end{equation}

\begin{equation}\label{lower_move}
\scalebox{0.6}{
\begin{tikzpicture}[baseline=-33pt]
\draw[line width=\st] (-0.2,0) node[anchor=east]{$i+1$}--(7.6,0);
\draw[line width=\st] (-0.2,-1)node[anchor=east]{$0\ge i$}--(7.6,-1);
\draw[line width=\st] (-0.2,-2)node[anchor=east]{$i-1$}--(7.6,-2);
\draw[line width=\tt,red] (0,-1)--(1,-2) -- (2,-1) -- (3,-2)-- (3.3,-1.7);
\draw[fill,red] (3.8,-1.5) circle(0.02);
\draw[fill,red] (4.2,-1.5) circle(0.02);
\draw[fill,red] (4.6,-1.5) circle(0.02);
\draw[line width=\tt,red] (5.1,-1.3)--(5.4,-1)--(6.4,-2)--(7.4,-1) ;
\end{tikzpicture}}
\quad
\to
\quad
\scalebox{0.6}{
\begin{tikzpicture}[baseline=-33pt]
\draw[line width=\st] (-0.2,0) --(7.6,0);
\draw[line width=\st] (-0.2,-1)--(7.6,-1);
\draw[line width=\st] (-0.2,-2)--(7.6,-2);
\draw[line width=\tt,red] (0,-1)--(1,0) -- (2,-1) -- (3,0)--(3.3,-0.3);
\draw[line width=\tt,red] (5.1,-0.7)--(5.4,-1)--(6.4,0)--(7.4,-1);
\draw[fill,red] (3.8,-0.5) circle(0.02);
\draw[fill,red] (4.2,-0.5) circle(0.02);
\draw[fill,red] (4.6,-0.5) circle(0.02);
\end{tikzpicture}}
\end{equation}

\noindent
They can naturally be thought of as folds towards the fundamental chamber.

For a diagram of width $m$, we define its {\it signature} $\u{\epsilon}=(\epsilon_1,\ldots,\epsilon_m)$ by


\smallskip

\begin{center}
\scalebox{0.6}{
\begin{tikzpicture}[baseline=-6pt]
\draw[line width=\st] (0,0) --(2,0);
\draw[line width=\tt,red] (0.5,-0.5)--(1.5,0.5);
\end{tikzpicture}
}
$\Rightarrow\epsilon_i=1$
\qquad\qquad
\scalebox{0.6}{
\begin{tikzpicture}[baseline=-15pt]
\draw[line width=\st] (0,0) --(2,0);
\draw[line width=\tt,red] (0.5,-0.5)--(1,0);
\draw[line width=\tt,red] (1,0)--(1.5,-0.5);
\end{tikzpicture}
}
$\Rightarrow\epsilon_i=0$
\qquad\qquad
\scalebox{0.6}{
\begin{tikzpicture}[baseline=-23pt]
\draw[line width=\st] (0,-1) --(2,-1);
\draw[line width=\tt,red] (0.5,-0.5)--(1,-1);
\draw[line width=\tt,red] (1,-1)--(1.5,-0.5);
\end{tikzpicture}
}
$\Rightarrow\epsilon_i=0$
\end{center}
\smallskip
for the $i$-th intersection of the red line and the walls counted from the left.

We can rewrite moves~(\ref{upper_move}) and ~(\ref{lower_move}) in the language of signatures as follows:
\begin{equation}\label{eq:cycb6:3}
(\epsilon_1,\ldots,\epsilon_k,0,\ldots,0,\epsilon_l,\ldots,\epsilon_m)\to
(\epsilon_1,\ldots,1-\epsilon_k,0,\ldots,0,1-\epsilon_l,\ldots,\epsilon_m),
\end{equation}
where the left and the right intersections of the red curve with the $i$-wall are the $k$th and the $l$th
intersections 
respectively. Note that
\begin{equation}\label{eq:parity}
k+l\equiv0\pmod 2.
\end{equation}
This representation of double folds is however deficient in that we do not know if we fold towards or
away from the fundamental chamber.

%

The next proposition relates diagrams and paths on a circle. Its proof is a simple geometrical observation,
which we leave to the reader.

\begin{proposition}\label{proposition:cycb6:5}
Let $2\le n<\infty$, $\Omega$ be a closed arc in $S^1_{n/\pi}$,
and $g:[a,b]\to\Omega$ be a path with natural parametrization smooth outside $P_n$ such that
$g(a)$ and $g(b)$ belong to the centers of some arc chambers.
Let $I:\Omega\ito[c,d]$ be a natural chart such that $I(\Omega\cap P_n)\subset\Z$. Then the graph $\Gr(I\circ g)$ 
is a diagram\footnote{after extending it horizontally and forgetting the abscissas}
of width $m=b-a$. 

Suppose additionally that  $g(a)\in\mathscr F$. Let $\u{\epsilon}$ be the signature of $\Gr(I\circ g)$
and $\u{\Gamma}$ be the gallery generated by $g$.
Then $\u{\epsilon}\=\u{\gamma}\subset\u{s}$, where $\u{s}=(\c,\bar\c,\ldots)_m$ and $g(a+1/2)\in\H_\c$.
\end{proposition}

\begin{corollary}\label{corollary:cycl:3}
Let $2\le n<\infty$, $D$ be a diagram starting at ordinate $1/2$ situated between the $1-n$-wall and the $n$-wall and
$D'$ be obtained from $D$ by a move~(\ref{upper_move}) or~(\ref{lower_move}). 
Let $\u{\epsilon}$ and $\u{\epsilon}'$ be the signatures of $D$ and $D'$ respectively and
$m$ be the width of $D$.

Let us choose a generator $\c\in\{\a,\b\}$, set $\u{s}=(\c,\bar\c,\ldots)_m$ and assume
$\u{\epsilon},\u{\epsilon}'\subset\u{s}$.
%
Then 
$\dist(\u{\epsilon},\u{\epsilon}')=2$
and $\u{\epsilon}>\u{\epsilon}'$ as elements of $\SubExpr(\u{s},w)$, where $w=\u{s}^{\u{\epsilon}}$.
\end{corollary}
\begin{proof}
In view of~(\ref{eq:cycb6:3}) and~(\ref{eq:parity}), we need only to prove that $\u{\epsilon}>\u{\epsilon}'$.
Let $i$ be the index as in~(\ref{upper_move}) or~(\ref{lower_move}) and $j$ be the ordinate of the leftmost
intersection of the curve of $D$ with a wall. Obviously, $j\in\{0,1\}$ and
$1-n<i<n$.

Let $\mathscr C_\c\cap S^1_{n/\pi}=\{p\}$ and $\mathscr C_{\bar\c}\cap S^1_{n/\pi}=\{q\}$.
Let $\Omega=S^1_{n/\pi}\setminus\mathscr F^\op$ and $I:\Omega\ito[1-n,n]$ be the natural chart
such that $I(p)=j$ and $I(q)=1-j$. Note that $I(\Omega\cap P_n)\subset\Z$ and the fundamental chamber $\mathscr F$
is the open minor arc between $p$ and $q$.

By the hypothesis, there exists a path $g:[a,b]\to \Omega$ having natural parametrization smooth outside $P_n$
such that $D=\Gr(I\circ g)$. Note that $g(a)$ is then in the center of the fundamental chamber
and $g(a+1/2)=p\in\mathscr H_\c$.

Let $g^{-1}(P_n)=\{x_1<\cdots<x_m\}$.
Choosing $k$ and $l$ as in~(\ref{eq:cycb6:3}), we get $I(x_k)=I(x_l)=i$ and thus $D'=\Gr(I\circ g')$, where $g'=\f_{x_k,x_l}g$.
As, the image of $g$ is contained in $\Omega$, the path $g'$ is a double fold of $g$ towards the fundamental chamber. 
Denoting by $\u{\Gamma}$ and $\u{\Gamma}'$ the galleries generated by $g$ and $g'$ respectively,
we get by Proposition~\ref{proposition:cycb6:4} that $\u{\Gamma}'$ is also a double fold of $\u{\Gamma}$
towards the fundamental chamber. Hence
$
\u{\epsilon}\=\u{\gamma}>\u{\gamma}'\=\u{\epsilon}'
$
by Propositions~\ref{proposition:cycb6:6} and~\ref{proposition:cycb6:5}.
\end{proof}

\subsection{Cycles of the first type}\label{cycles_first} We assume here that $n\ge2$ is finite. Let us choose $\c\in\{\a,\b\}$
and $x,y\in\Z$ such that $1\le x\le n-1$ and $x+n+1\le y\le2n$. Let $\mathscr C_\c\cap S^1_{n/\pi}=\{p\}$ and $\mathscr C_{\bar\c}\cap S^1_{n/\pi}=\{q\}$.
We consider the closed arc
$$
\Omega=S^1_{n/\pi}\setminus\mathscr F^\op
$$
and the natural chart $I:\Omega\to[1-n,n]$ such that $I(p)=1$ and $I(q)=0$.
Then the point $O=I^{-1}(1/2)$ is the center of the fundamental arc chamber.
Let $g:[1/2,y+1/2]\to S^1_{n/\pi}$
be a smooth path having natural parametrization such that $g(1/2)=O$ and $g(1)=p$.

Let $\u{\Delta}$, $\u{\Sigma}$, $\u{\widetilde\Sigma}$ be the galleries generated by
the paths
$g$, $h=\f_{x+1,x+n+1}g$, $\tilde h=\f_{x,x+n}g$
respectively. These paths look as follows (the numbers representing the $I$-coordinates):

\medskip

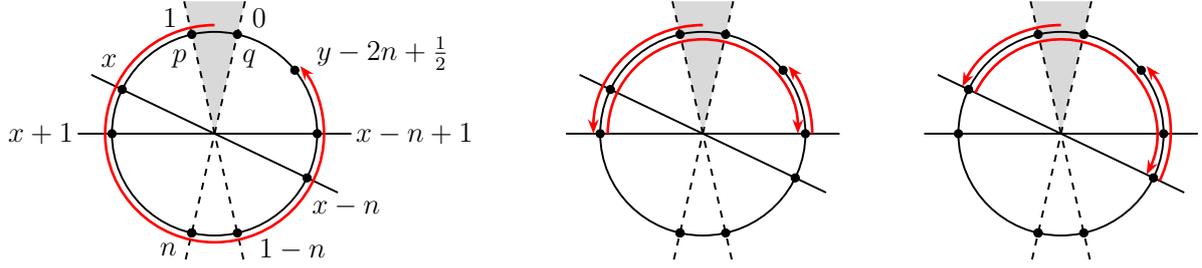
\def\tt{1.1pt}
\def\st{0.8pt}

\begin{center}
\scalebox{0.9}{
\begin{tikzpicture}[baseline=-63pt]
\fill[gray!30] (0,0)--(-.4450418670, 1.949855825)--(.4450418670, 1.949855825);
\draw[line width=\st] (0,0) circle(1.5);
\draw[line width=\st,dashed] (-.4450418670, 1.949855825) -- (.4450418670,-1.949855825);
\draw[fill] (.3337814002, 1.462391868) circle(0.06)node[above=7pt,right=2pt]{$0$}node[below=10pt,right=-2pt]{$q$};
\draw[line width=\st,dashed] (.4450418670, 1.949855825) -- (-.4450418670,-1.949855825);
\draw[fill] (-.3337814002, 1.462391868) circle(0.06)node[above=7pt,left=2pt]{$1$}node[below=10pt,left=-2pt]{$p$};
\draw[line width=\st] (-1.801937736, .8677674786)--(1.801937736, -.8677674786);
\draw[fill] (-1.351453302, .6508256088) circle(0.06)node[above=12pt,left=-2pt]{$x$};
\draw[fill] (1.351453302, -.6508256088) circle(0.06)node[below=12pt,right=-2pt]{$x-n$};

\draw[line width=\st] (-2, 0)--(2, 0);
\draw[fill] (-1.5, 0) circle(0.06)node[left=12pt]{$x+1$};
\draw[fill] (1.5, 0) circle(0.06)node[right=12pt]{$x-n+1$};

\draw[fill] (1.172747224, .9352347028) circle(0.06)node[above=7pt,right=5pt]{$y-2n+\frac12$};
\draw[fill] (.3337814002, -1.462391868) circle(0.06);
\draw[fill] (-.3337814002, -1.462391868) circle(0.06)node[below=7pt,left=2pt]{$n$};
\draw[fill] (-.3337814002, -1.462391868) circle(0.06)node[below=7pt,right=24pt]{$1-n$};

\draw[-{Stealth[length=\sl]},line width=\tt,red] (0,1.6) arc (90:398.5714286:1.6);

\end{tikzpicture}}
\qquad
\scalebox{0.9}{
\begin{tikzpicture}[baseline=-63pt]
\fill[gray!30] (0,0)--(-.4450418670, 1.949855825)--(.4450418670, 1.949855825);
\draw[line width=\st] (0,0) circle(1.5);
\draw[line width=\st,dashed] (-.4450418670, 1.949855825) -- (.4450418670,-1.949855825);
\draw[fill] (.3337814002, 1.462391868) circle(0.06);
\draw[line width=\st,dashed] (.4450418670, 1.949855825) -- (-.4450418670,-1.949855825);
\draw[fill] (-.3337814002, 1.462391868) circle(0.06);
\draw[line width=\st] (-1.801937736, .8677674786)--(1.801937736, -.8677674786);
\draw[fill] (-1.351453302, .6508256088) circle(0.06);
\draw[fill] (1.351453302, -.6508256088) circle(0.06);

\draw[line width=\st] (-2, 0)--(2, 0);
\draw[fill] (-1.5, 0) circle(0.06);
\draw[fill] (1.5, 0) circle(0.06);

\draw[fill] (1.172747224, .9352347028) circle(0.06);
\draw[fill] (.3337814002, -1.462391868) circle(0.06);
\draw[fill] (-.3337814002, -1.462391868) circle(0.06);
\draw[fill] (-.3337814002, -1.462391868) circle(0.06);

\draw[-{Stealth[length=\sl]},line width=\tt,red] (0,1.6) arc (90:180:1.6);

\draw[-{Stealth[length=\sl]},line width=\tt,red] (-1.39, 0) arc (180:0:1.39);

\draw[-{Stealth[length=\sl]},line width=\tt,red] (1.6, 0) arc (0:38.5714286:1.6);

\end{tikzpicture}}
\qquad
\scalebox{0.9}{
\begin{tikzpicture}[baseline=-63pt]
\fill[gray!30] (0,0)--(-.4450418670, 1.949855825)--(.4450418670, 1.949855825);
\draw[line width=\st] (0,0) circle(1.5);
\draw[line width=\st,dashed] (-.4450418670, 1.949855825) -- (.4450418670,-1.949855825);
\draw[fill] (.3337814002, 1.462391868) circle(0.06);
\draw[line width=\st,dashed] (.4450418670, 1.949855825) -- (-.4450418670,-1.949855825);
\draw[fill] (-.3337814002, 1.462391868) circle(0.06);
\draw[line width=\st] (-1.801937736, .8677674786)--(1.801937736, -.8677674786);
\draw[fill] (-1.351453302, .6508256088) circle(0.06);
\draw[fill] (1.351453302, -.6508256088) circle(0.06);

\draw[line width=\st] (-2, 0)--(2, 0);
\draw[fill] (-1.5, 0) circle(0.06);
\draw[fill] (1.5, 0) circle(0.06);

\draw[fill] (1.172747224, .9352347028) circle(0.06);
\draw[fill] (.3337814002, -1.462391868) circle(0.06);
\draw[fill] (-.3337814002, -1.462391868) circle(0.06);
\draw[fill] (-.3337814002, -1.462391868) circle(0.06);

\draw[-{Stealth[length=\sl]},line width=\tt,red] (0,1.6) arc (90:154.2857143:1.6);

\draw[-{Stealth[length=\sl]},line width=\tt,red] (-1.252346726, .6030983975) arc (154.2857143:-25.71428571:1.39);

\draw[-{Stealth[length=\sl]},line width=\tt,red] (1.441550189, -.6942139821) arc (-25.71428571:38.5714286:1.6);

\end{tikzpicture}}
\end{center}
\noindent
Here the actual paths are the radial projections of the red lines to the circle.
Note the images of the last two paths are contained in $\Omega$.

By Proposition~\ref{proposition:cycb6:4}, we get $\u{\Sigma}=\f_{x+1,x+n+1}\u{\Gamma}$ and $\u{\widetilde\Sigma}=\f_{x,x+n}\u{\Gamma}$.
Clearly, $\u{\delta}=\u{1}^y\subset\u{t}=(\c,\bar\c,\ldots)_y$.
By Proposition~\ref{proposition:cycb6:6}, we get
$$
\u{\sigma}=(1,\ldots,\underbrace{0}_{x+1},\ldots,\underbrace{0}_{x+n+1},\ldots,1),\quad
\u{\tilde\sigma}=(1,\ldots,\underbrace{0}_{x},\ldots,\underbrace{0}_{x+n},\ldots,1).
$$
As both paths $h$ and $\tilde h$ are double folds of $g$ towards the fundamental chamber,
the galleries $\u{\Sigma}$ and $\u{\widetilde\Sigma}$ are also double folds of $\u{\Delta}$ towards
the fundamental chamber by Proposition~\ref{proposition:cycb6:4}. Hence
\begin{equation}\label{eq:cycb6:5}
\u{\sigma}<\u{\delta},\quad \u{\tilde\sigma}<\u{\delta}
\end{equation}
by Proposition~\ref{proposition:cycb6:6}. 
By Proposition~\ref{proposition:cycb6:5}, the diagrams $D=\Gr(I\circ h)$
and $\widetilde D=\Gr(I\circ \tilde h)$ have width $y$ and
signatures $\u{\sigma}$ and $\u{\tilde\sigma}$ respectively.
The diagram $D$ can be reduced as follows\footnote{Here and in what follows, ``reduced'' means obtained
by moves~(\ref{upper_move}) and~(\ref{lower_move}). The dashed red line
represents an imaginary rectangle that helps keep track of transformations of our paths}:

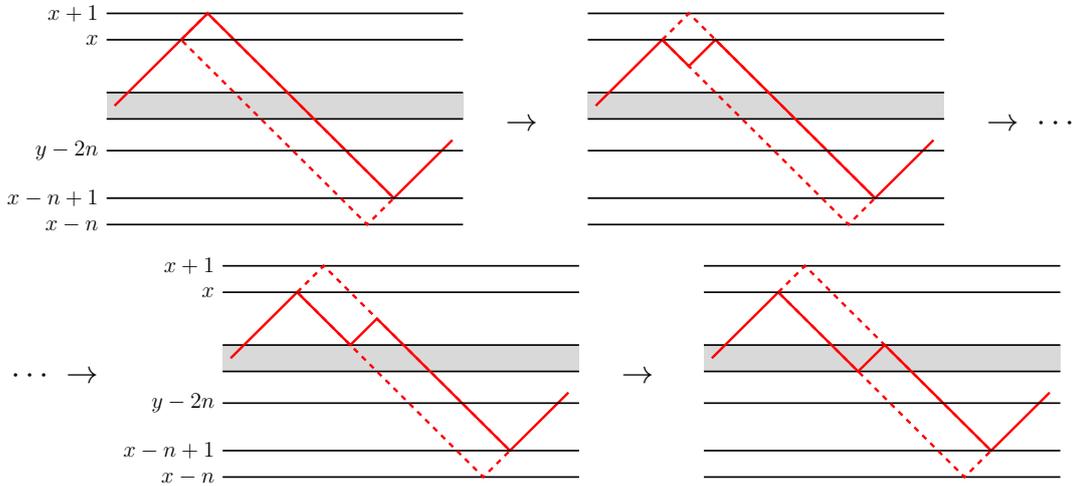
\def\tt{1.3pt}
\def\st{0.9pt}

\medskip

\begin{center}
\scalebox{0.7}{
\begin{tikzpicture}[baseline=-63pt]
\fill[gray!30] (-0.9,-1.5)--(5.8,-1.5)--(5.8,-2)--(-0.9,-2)--(-0.9,-1.5);
\draw[line width=\st] (-0.9,-1.5)--(5.8,-1.5);
\draw[line width=\st] (-0.9,-2)--(5.8,-2);

\draw[line width=\tt,red] (-0.75,-1.75)--(1,0)--(4.5,-3.5)--(5.6,-2.4);

\draw[line width=\st] (-0.9,-0)node[anchor=east]{$x+1$}--(5.8,-0);
\draw[line width=\st] (-0.9,-0.5)node[anchor=east]{$x$}--(5.8,-0.5);
\draw[line width=\st] (-0.9,-4)node[anchor=east]{$x-n$}--(5.8,-4);
\draw[line width=\st] (-0.9,-3.5)node[anchor=east]{$x-n+1$}--(5.8,-3.5);
\draw[line width=\st] (-0.9,-2.6)node[anchor=east]{$y-2n$}--(5.8,-2.6);

\draw[line width=\tt,red,dashed] (0.5,-0.5)--(1,0)--(4.5,-3.5)--(4,-4)--(0.5,-0.5);

\end{tikzpicture}}
$\quad\to\quad$
\scalebox{0.7}{
\begin{tikzpicture}[baseline=-63pt]
\fill[gray!30] (-0.9,-1.5)--(5.8,-1.5)--(5.8,-2)--(-0.9,-2)--(-0.9,-1.5);
\draw[line width=\st] (-0.9,-1.5)--(5.8,-1.5);
\draw[line width=\st] (-0.9,-2)--(5.8,-2);

\draw[line width=\tt,red] (-0.75,-1.75)--(0.5,-0.5)--(1,-1)--(1.5,-0.5)--(4.5,-3.5)--(5.6,-2.4);

\draw[line width=\st] (-0.9,-0)--(5.8,-0);
\draw[line width=\st] (-0.9,-0.5)--(5.8,-0.5);
\draw[line width=\st] (-0.9,-4)--(5.8,-4);
\draw[line width=\st] (-0.9,-3.5)--(5.8,-3.5);
\draw[line width=\st] (-0.9,-2.6)--(5.8,-2.6);

\draw[line width=\tt,red,dashed] (0.5,-0.5)--(1,0)--(4.5,-3.5)--(4,-4)--(0.5,-0.5);
\end{tikzpicture}}
$\quad\to\;\cdots$
\end{center}


\begin{center}
$\cdots\;\to$
\scalebox{0.7}{
\begin{tikzpicture}[baseline=-63pt]
\fill[gray!30] (-0.9,-1.5)--(5.8,-1.5)--(5.8,-2)--(-0.9,-2)--(-0.9,-1.5);
\draw[line width=\st] (-0.9,-1.5)--(5.8,-1.5);
\draw[line width=\st] (-0.9,-2)--(5.8,-2);

\draw[line width=\tt,red] (-0.75,-1.75)--(0.5,-0.5)--(1.5,-1.5)--(2,-1)--(4.5,-3.5)--(5.6,-2.4);

\draw[line width=\st] (-0.9,-0)node[anchor=east]{$x+1$}--(5.8,-0);
\draw[line width=\st] (-0.9,-0.5)node[anchor=east]{$x$}--(5.8,-0.5);
\draw[line width=\st] (-0.9,-4)node[anchor=east]{$x-n$}--(5.8,-4);
\draw[line width=\st] (-0.9,-3.5)node[anchor=east]{$x-n+1$}--(5.8,-3.5);
\draw[line width=\st] (-0.9,-2.6)node[anchor=east]{$y-2n$}--(5.8,-2.6);

\draw[line width=\tt,red,dashed] (0.5,-0.5)--(1,0)--(4.5,-3.5)--(4,-4)--(0.5,-0.5);

\end{tikzpicture}}
$\quad\to\quad$
\scalebox{0.7}{
\begin{tikzpicture}[baseline=-63pt]
\fill[gray!30] (-0.9,-1.5)--(5.8,-1.5)--(5.8,-2)--(-0.9,-2)--(-0.9,-1.5);
\draw[line width=\st] (-0.9,-1.5)--(5.8,-1.5);
\draw[line width=\st] (-0.9,-2)--(5.8,-2);

\draw[line width=\tt,red] (-0.75,-1.75)--(0.5,-0.5)--(2,-2)--(2.5,-1.5)--(4.5,-3.5)--(5.6,-2.4);

\draw[line width=\st] (-0.9,-0)--(5.8,-0);
\draw[line width=\st] (-0.9,-0.5)--(5.8,-0.5);
\draw[line width=\st] (-0.9,-4)--(5.8,-4);
\draw[line width=\st] (-0.9,-3.5)--(5.8,-3.5);
\draw[line width=\st] (-0.9,-2.6)--(5.8,-2.6);

\draw[line width=\tt,red,dashed] (0.5,-0.5)--(1,0)--(4.5,-3.5)--(4,-4)--(0.5,-0.5);

\end{tikzpicture}}
\end{center}

\medskip
\noindent
This process takes $x$ steps and is described by the signatures
$\u{\sigma}^{(x+1)},\ldots,\u{\sigma}^{(1)}$ given by
$$
\u{\sigma}^{(i)}=
\left\{
\begin{array}{ll}
\u{\sigma}&\text{if }i=x+1;\\[6pt]
(\ldots\,\underbrace{0}_{x}\;\;\;\ldots\!\!\!\underbrace{0,0}_{2x+1-i,2x+2-i}\!\!\!\ldots\;\;\;\underbrace{0}_{x+n+1}\ldots)&\text{if }1\le i\le x.
\end{array}
\right.
$$
Here and in what follows, we indicate only the positions of zeros.

Similarly, the diagram $\widetilde D$ can be reduced to the same diagram as follows:

\medskip

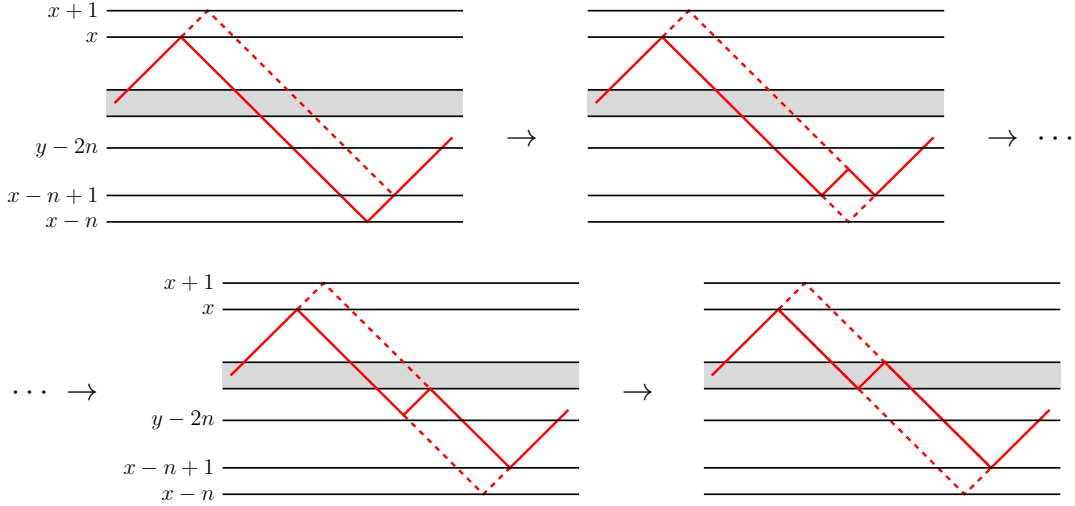
\begin{center}
\scalebox{0.7}{
\begin{tikzpicture}[baseline=-73pt]
\fill[gray!30] (-0.9,-1.5)--(5.8,-1.5)--(5.8,-2)--(-0.9,-2)--(-0.9,-1.5);
\draw[line width=\st] (-0.9,-1.5)--(5.8,-1.5);
\draw[line width=\st] (-0.9,-2)--(5.8,-2);

\draw[line width=\tt,red] (-0.75,-1.75)--(0.5,-0.5)--(4,-4)--(5.6,-2.4);
\draw[line width=\tt,red,dashed] (0.5,-0.5)--(1,0)--(4.5,-3.5);

\draw[line width=\st] (-0.9,-0)node[anchor=east]{$x+1$}--(5.8,-0);
\draw[line width=\st] (-0.9,-0.5)node[anchor=east]{$x$}--(5.8,-0.5);
\draw[line width=\st] (-0.9,-4)node[anchor=east]{$x-n$}--(5.8,-4);
\draw[line width=\st] (-0.9,-3.5)node[anchor=east]{$x-n+1$}--(5.8,-3.5);
\draw[line width=\st] (-0.9,-2.6)node[anchor=east]{$y-2n$}--(5.8,-2.6);

\end{tikzpicture}}
$\quad\to\quad$
\scalebox{0.7}{
\begin{tikzpicture}[baseline=-73pt]
\fill[gray!30] (-0.9,-1.5)--(5.8,-1.5)--(5.8,-2)--(-0.9,-2)--(-0.9,-1.5);
\draw[line width=\st] (-0.9,-1.5)--(5.8,-1.5);
\draw[line width=\st] (-0.9,-2)--(5.8,-2);

\draw[line width=\tt,red] (-0.75,-1.75)--(0.5,-0.5)--(3.5,-3.5)--(4,-3)--(4.5,-3.5)--(5.6,-2.4);
\draw[line width=\tt,red,dashed] (0.5,-0.5)--(1,0)--(4,-3);
\draw[line width=\tt,red,dashed] (3.5,-3.5)--(4,-4)--(4.5,-3.5);

\draw[line width=\st] (-0.9,-0)--(5.8,-0);
\draw[line width=\st] (-0.9,-0.5)--(5.8,-0.5);
\draw[line width=\st] (-0.9,-4)--(5.8,-4);
\draw[line width=\st] (-0.9,-3.5)--(5.8,-3.5);
\draw[line width=\st] (-0.9,-2.6)--(5.8,-2.6);
\end{tikzpicture}}
$\quad\to\;\cdots$
\end{center}


\smallskip

\smallskip

\begin{center}
$\cdots\;\to$
\scalebox{0.7}{
\begin{tikzpicture}[baseline=-63pt]
\fill[gray!30] (-0.9,-1.5)--(5.8,-1.5)--(5.8,-2)--(-0.9,-2)--(-0.9,-1.5);
\draw[line width=\st] (-0.9,-1.5)--(5.8,-1.5);
\draw[line width=\st] (-0.9,-2)--(5.8,-2);

\draw[line width=\tt,red] (-0.75,-1.75)--(0.5,-0.5)--(2.5,-2.5)--(3,-2)--(4.5,-3.5)--(5.6,-2.4);

\draw[line width=\tt,red,dashed] (0.5,-0.5)--(1,0)--(3,-2);
\draw[line width=\tt,red,dashed] (2.5,-2.5)--(4,-4)--(4.5,-3.5);

\draw[line width=\st] (-0.9,-0)node[anchor=east]{$x+1$}--(5.8,-0);
\draw[line width=\st] (-0.9,-0.5)node[anchor=east]{$x$}--(5.8,-0.5);
\draw[line width=\st] (-0.9,-4)node[anchor=east]{$x-n$}--(5.8,-4);
\draw[line width=\st] (-0.9,-3.5)node[anchor=east]{$x-n+1$}--(5.8,-3.5);
\draw[line width=\st] (-0.9,-2.6)node[anchor=east]{$y-2n$}--(5.8,-2.6);

\end{tikzpicture}}
$\quad\to\quad$
\scalebox{0.7}{
\begin{tikzpicture}[baseline=-63pt]
\fill[gray!30] (-0.9,-1.5)--(5.8,-1.5)--(5.8,-2)--(-0.9,-2)--(-0.9,-1.5);
\draw[line width=\st] (-0.9,-1.5)--(5.8,-1.5);
\draw[line width=\st] (-0.9,-2)--(5.8,-2);

\draw[line width=\tt,red] (-0.75,-1.75)--(0.5,-0.5)--(2,-2)--(2.5,-1.5)--(4.5,-3.5)--(5.6,-2.4);

\draw[line width=\tt,red,dashed] (0.5,-0.5)--(1,0)--(4.5,-3.5)--(4,-4)--(0.5,-0.5);

\draw[line width=\st] (-0.9,-0)--(5.8,-0);
\draw[line width=\st] (-0.9,-0.5)--(5.8,-0.5);
\draw[line width=\st] (-0.9,-4)--(5.8,-4);
\draw[line width=\st] (-0.9,-3.5)--(5.8,-3.5);
\draw[line width=\st] (-0.9,-2.6)--(5.8,-2.6);

\end{tikzpicture}}
\end{center}

\medskip
\noindent
This process takes $n-x$ steps and is described by the signatures
$\u{\sigma}^{(x-n)},\ldots,\u{\sigma}^{(0)}$ given by
$$
\u{\sigma}^{(i)}=
\left\{
\begin{array}{ll}
\u{\tilde\sigma}&\text{if }i=x-n;\\[6pt]
(\ldots\,\underbrace{0}_{x}\;\;\;\ldots\!\!\!\underbrace{0,0}_{2x-i,2x+1-i}\!\!\!\ldots\,\,\,\underbrace{0}_{x+n+1}\ldots)&\text{if }x-n<i\le0.
\end{array}
\right.
$$

Let $w=(\c\bar\c\cdots)_y$. We get the following cycle in $\Sub(\u{t},w)$:
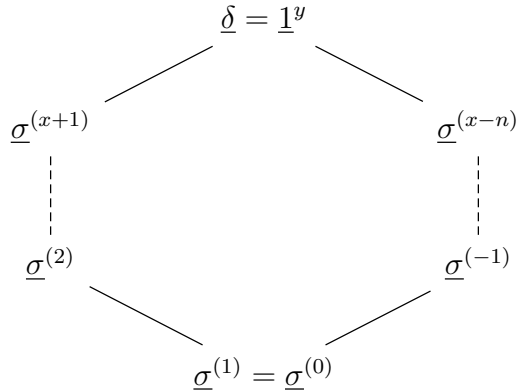
$$
\begin{tikzcd}[row sep=30pt]
&\u{\delta}=\u{1}^y\arrow[no head]{rd}\arrow[no head]{ld}&\\[-10pt]
\u{\sigma}^{(x+1)}\arrow[no head,dashed]{d}&&\u{\sigma}^{(x-n)}\arrow[no head,dashed]{d}\\
\u{\sigma}^{(2)}\arrow[no head]{rd}&&\u{\sigma}^{(-1)}\arrow[no head]{ld}\\[-10pt]
&\u{\sigma}^{(1)}=\u{\sigma}^{(0)}&
\end{tikzcd}
$$
of length $n+2$ with maximal vertex $\u{\delta}$.
Indeed, the virtices in the
left path followed from top to down
form a chain and
decrease by~(\ref{eq:cycb6:5}) and Corollary~\ref{corollary:cycl:3}.
Similarly, the vertices in the
right path 
form a chain and decrease.
Therefore, the vertices in each path do not coincide.
The vertices from different paths except the initial ones also do not coincide,
as can be directly seen from the above diagrams (or their sigmatures). We denote this cycle by $\Cyc^1_\c(x,y)$.

\subsection{Cycles of the second type}\label{cycles_second} Let $n$, $\c$, $p$, $q$, $\Omega$, $I$, $O$
be as in the previous section and $x,y\in\Z$ be such that $1\le x<y\le n$.
Let $g_1:[1/2,x+n]\to S^1_{n/\pi}$
and $g_2:[x+n,2x-y+2n+1/2]\to S^1_{n/\pi}$ be smooth paths
having natural parametrizations such that
$$
g_1(\tfrac12)=O, \quad I\circ g_1(x+n)=I\circ g_2(x+n)=x-n,\quad I\circ g_2(2x-y+2n+\tfrac12)=y-\tfrac12.
$$
Concatenating, we get a path $g=g_1\cup g_2$ having natural parametrization
that is smooth everywhere except at $x+n$.
Let $\u{\Delta}$, $\u{\Sigma}$, $\u{\widetilde\Sigma}$, $\u{\widehat\Sigma}$ be the galleries generated by
the paths
$$
g,\quad h=\f_{y,2x-y+2n}g,\quad \tilde h=\f_{x,x+n}g,\quad \hat h=\f_{2x-y+n,2x+2n-y}\tilde h
$$
respectively. These paths look as follows (the numbers representing the $I$-coordinates):

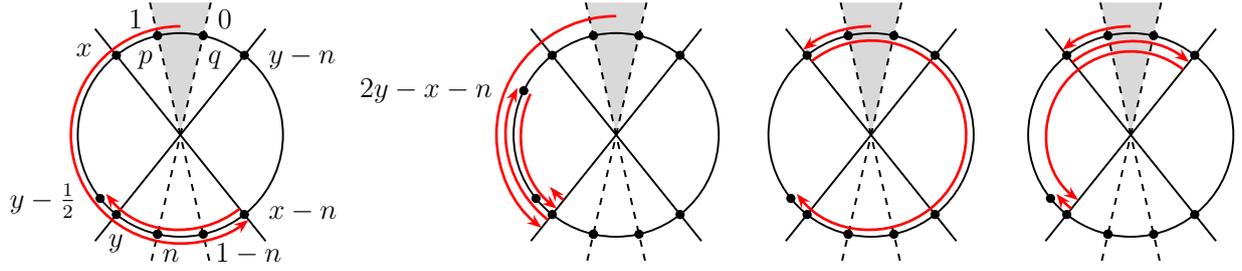
\def\tt{1.1pt}
\def\st{0.8pt}

\medskip

\begin{center}
\!\!\!\!\!\!
\scalebox{0.9}{
\begin{tikzpicture}[baseline=-63pt]
\fill[gray!30] (0,0)--(-.4450418670, 1.949855825)--(.4450418670, 1.949855825);
\draw[line width=\st] (0,0) circle(1.5);
\draw[line width=\st,dashed] (-.4450418670, 1.949855825) -- (.4450418670,-1.949855825);
\draw[fill] (.3337814002, 1.462391868) circle(0.06)node[above=7pt,right=2pt]{$0$}node[below=10pt,right=-2pt]{$q$};
\draw[line width=\st,dashed] (.4450418670, 1.949855825) -- (-.4450418670,-1.949855825);
\draw[fill] (-.3337814002, 1.462391868) circle(0.06)node[above=7pt,left=2pt]{$1$}node[below=10pt,left=-2pt]{$p$};
\draw[line width=\st] (-1.246979604, 1.563662964)--(1.246979604, -1.563662964);
\draw[fill] (-.9352347033, 1.172747223) circle(0.06)node[above=2pt,left=6pt]{$x$};
\draw[fill] (.9352347033, -1.172747223) circle(0.06)node[above=1pt,right=6pt]{$x-n$};


\draw[fill] (-.9352347012, -1.172747225) circle(0.06)node[below=12pt,left=-7pt]{$y$};
\draw[fill] (-1.172747224, -.9352347024) circle(0.06)node[below=1pt,left=6pt]{$y-\frac12$};

\draw[fill] (.9352347012, 1.172747225) circle(0.06)node[below=1pt,right=6pt]{$y-n$};
\draw[line width=\st] (-1.246979602, -1.563662967)--(1.246979602, 1.563662967);

\draw[fill] (.3337814002, -1.462391868) circle(0.06);
\draw[fill] (-.3337814002, -1.462391868) circle(0.06)node[below=9pt,right=-2pt]{$n$};
\draw[fill] (-.3337814002, -1.462391868) circle(0.06)node[below=8pt,right=20pt]{$1-n$};

\draw[-{Stealth[length=\sl]},line width=\tt,red] (0,1.6) arc (90:308.5714286:1.6);
\draw[-{Stealth[length=\sl]},line width=\tt,red] (.8666508244, -1.086745761) arc (308.5714286:218.5714286:1.39);
\end{tikzpicture}}
\!\!\!\!\!
\scalebox{0.9}{
\begin{tikzpicture}[baseline=-63pt]
\fill[gray!30] (0,0)--(-.4450418670, 1.949855825)--(.4450418670, 1.949855825);
\draw[line width=\st] (0,0) circle(1.5);
\draw[line width=\st,dashed] (-.4450418670, 1.949855825) -- (.4450418670,-1.949855825);
\draw[fill] (.3337814002, 1.462391868) circle(0.06);
\draw[line width=\st,dashed] (.4450418670, 1.949855825) -- (-.4450418670,-1.949855825);
\draw[fill] (-.3337814002, 1.462391868) circle(0.06);
\draw[line width=\st] (-1.246979604, 1.563662964)--(1.246979604, -1.563662964);
\draw[fill] (-.9352347033, 1.172747223) circle(0.06);
\draw[fill] (.9352347033, -1.172747223) circle(0.06);

\draw[fill] (-.9352347012, -1.172747225) circle(0.06);
\draw[fill] (-1.172747224, -.9352347024) circle(0.06);

\draw[fill] (.9352347012, 1.172747225) circle(0.06);
\draw[line width=\st] (-1.246979602, -1.563662967)--(1.246979602, 1.563662967);

\draw[fill] (.3337814002, -1.462391868) circle(0.06);
\draw[fill] (-.3337814002, -1.462391868) circle(0.06);
\draw[fill] (-.3337814002, -1.462391868) circle(0.06);
\draw[fill] (-1.351453302, .6508256088) circle(0.06)node[left=9pt]{$2y-x-n$};

\draw[-{Stealth[length=\sl]},line width=\tt,red] (0,1.75) arc (90:231.4285714:1.75);
\draw[-{Stealth[length=\sl]},line width=\tt,red] (-.9975836813, -1.250930373) arc (231.4285714:154.2857143:1.6);
\draw[-{Stealth[length=\sl]},line width=\tt,red] (-1.252346726, .6030983975) arc (154.2857143:231.4285714:1.39);
\draw[-{Stealth[length=\sl]},line width=\tt,red] (-.7668924550, -.9616527245) arc (231.4285714:218.5714286:1.23);
\end{tikzpicture}}
\;\;
\scalebox{0.9}{
\begin{tikzpicture}[baseline=-63pt]
\fill[gray!30] (0,0)--(-.4450418670, 1.949855825)--(.4450418670, 1.949855825);
\draw[line width=\st] (0,0) circle(1.5);
\draw[line width=\st,dashed] (-.4450418670, 1.949855825) -- (.4450418670,-1.949855825);
\draw[fill] (.3337814002, 1.462391868) circle(0.06);
\draw[line width=\st,dashed] (.4450418670, 1.949855825) -- (-.4450418670,-1.949855825);
\draw[fill] (-.3337814002, 1.462391868) circle(0.06);
\draw[line width=\st] (-1.246979604, 1.563662964)--(1.246979604, -1.563662964);
\draw[fill] (-.9352347033, 1.172747223) circle(0.06);
\draw[fill] (.9352347033, -1.172747223) circle(0.06);

\draw[fill] (-.9352347012, -1.172747225) circle(0.06);
\draw[fill] (-1.172747224, -.9352347024) circle(0.06);

\draw[fill] (.9352347012, 1.172747225) circle(0.06);
\draw[line width=\st] (-1.246979602, -1.563662967)--(1.246979602, 1.563662967);

\draw[fill] (.3337814002, -1.462391868) circle(0.06);
\draw[fill] (-.3337814002, -1.462391868) circle(0.06);
\draw[fill] (-.3337814002, -1.462391868) circle(0.06);

\draw[-{Stealth[length=\sl]},line width=\tt,red] (0,1.6) arc (90:128.5714286:1.6);
\draw[-{Stealth[length=\sl]},line width=\tt,red] (-.8666508251, 1.086745760) arc (128.5714286:-141.4285714:1.39);
\end{tikzpicture}}
\;\;
\scalebox{0.9}{
\begin{tikzpicture}[baseline=-63pt]
\fill[gray!30] (0,0)--(-.4450418670, 1.949855825)--(.4450418670, 1.949855825);
\draw[line width=\st] (0,0) circle(1.5);
\draw[line width=\st,dashed] (-.4450418670, 1.949855825) -- (.4450418670,-1.949855825);
\draw[fill] (.3337814002, 1.462391868) circle(0.06);
\draw[line width=\st,dashed] (.4450418670, 1.949855825) -- (-.4450418670,-1.949855825);
\draw[fill] (-.3337814002, 1.462391868) circle(0.06);
\draw[line width=\st] (-1.246979604, 1.563662964)--(1.246979604, -1.563662964);
\draw[fill] (-.9352347033, 1.172747223) circle(0.06);
\draw[fill] (.9352347033, -1.172747223) circle(0.06);

\draw[fill] (-.9352347012, -1.172747225) circle(0.06);
\draw[fill] (-1.172747224, -.9352347024) circle(0.06);

\draw[fill] (.9352347012, 1.172747225) circle(0.06);
\draw[line width=\st] (-1.246979602, -1.563662967)--(1.246979602, 1.563662967);

\draw[fill] (.3337814002, -1.462391868) circle(0.06);
\draw[fill] (-.3337814002, -1.462391868) circle(0.06);
\draw[fill] (-.3337814002, -1.462391868) circle(0.06);

\draw[-{Stealth[length=\sl]},line width=\tt,red] (0,1.6) arc (90:128.5714286:1.6);
\draw[-{Stealth[length=\sl]},line width=\tt,red] (-.8541810290, 1.071109131) arc (128.5714286:51.42857143:1.37);
\draw[-{Stealth[length=\sl]},line width=\tt,red] (.7668924562, .9616527235) arc (51.42857143:231.4285714:1.23);
\draw[-{Stealth[length=\sl]},line width=\tt,red] (-.8666508231, -1.086745762) arc (231.4285714:218.5714286:1.39);
\end{tikzpicture}}
\end{center}

\noindent
Here again the actual paths are the radial projections. Note that the point 
with coordinate $2x-y+n$
in the second picture
is drawn for the case $2y-x-n>x$. The images of the second and the forth paths are contained in $\Omega$.
By Proposition~\ref{proposition:cycb6:4}, we get
$$
\u{\Sigma}=\f_{y,2x-y+2n}\u{\Delta},\quad \u{\widetilde\Sigma}=\f_{x,x+n}\u{\Delta},\quad \u{\widehat\Sigma}=\f_{2x-y+n,2x-y+2n}\u{\widetilde\Sigma}.
$$
Obviously, $\u{\delta}\subset\u{t}=(\c,\bar\c,\ldots)_{2x-y+2n}$. By construction and
Proposition~\ref{proposition:cycb6:6}, we get
$$
\u{\delta}=(1,\ldots,\underbrace{0}_{x+n},\ldots,1),\qquad \u{\sigma}=(1,\ldots,\underbrace{0}_{y},\ldots,\underbrace{0}_{x+n} ,\ldots,1,0),
$$
$$
\u{\tilde\sigma}=(1,\ldots,\underbrace{0}_{x},\ldots,1),\quad
\u{\hat\sigma}=(1,\ldots,\underbrace{0}_{x},\ldots,\underbrace{0}_{2x-y+n},\ldots,1,0).
$$
As $h$ and $\tilde h$ are double folds of $g$ towards the fundamental chamber and
$\hat h$ is a double fold of $\tilde h$ also towards the fundamental chamber,
the following inequalities hold:
\begin{equation}\label{eq:cycb6:6}
\u{\sigma}<\u{\delta},\quad \u{\hat\sigma}<\u{\tilde\sigma}<\u{\delta}
\end{equation}
by Propositions~\ref{proposition:cycb6:4} and~\ref{proposition:cycb6:6}.
By Proposition~\ref{proposition:cycb6:5}, the graphs
$D=\Gr(I\circ h)$ and $\widehat D=\Gr(I\circ\hat h)$ are diagrams of width $2x-y+2n$ and
with signatures $\u{\sigma}$ and $\u{\hat\sigma}$ respectively.

{\it Case 1:} $2y-x-n\ge x$. The diagram $D$ can be reduced as follows:

\medskip

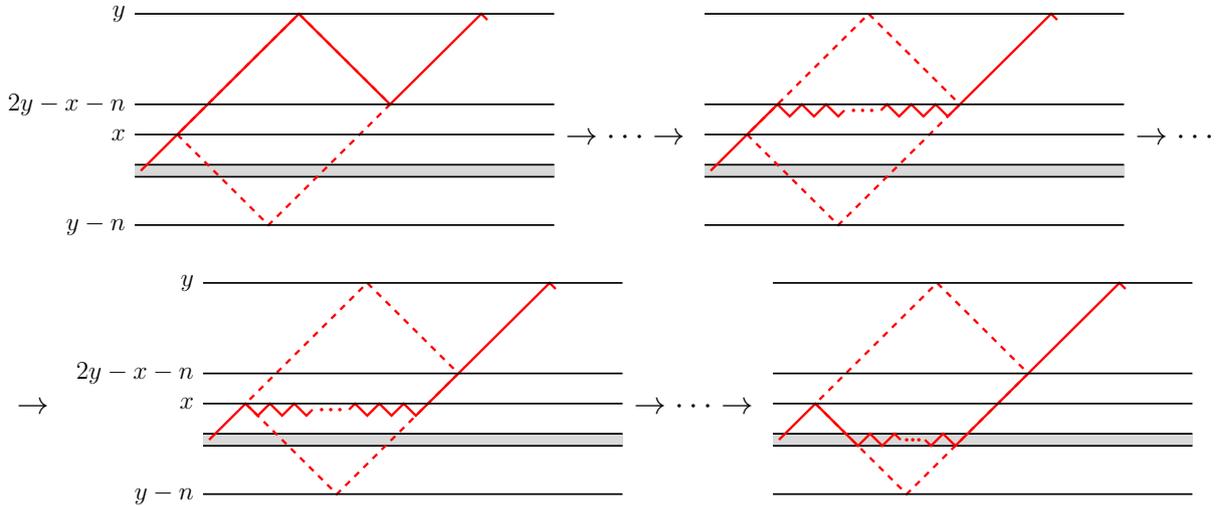
\begin{center}
\scalebox{0.8}{
\begin{tikzpicture}[baseline=-33pt]
\fill[gray!30] (-0.7,-1.5)--(6.2,-1.5)--(6.2,-1.7)--(-0.7,-1.7)--(-0.7,-1.5);

\draw[line width=\st] (-0.7,-1.5)--(6.2,-1.5);
\draw[line width=\st] (-0.7,-1.7)--(6.2,-1.7);
\draw[line width=\tt,red] (-0.6,-1.6)--(2,1)--(3.5,-0.5)--(5,1)--(5,1)--(5.1,0.9);

\draw[line width=\tt,red,dashed] (0,-1)--(1.5,-2.5)--(3.5,-0.5)--(2,1)--(0,-1);

\draw[line width=\st] (-0.7,1)node[anchor=east]{$y$}--(6.2,1);
\draw[line width=\st] (-0.7,-1)node[anchor=east]{$x$}--(6.2,-1);
\draw[line width=\st] (-0.7,-0.5)node[anchor=east]{$2y-x-n$}--(6.2,-0.5);
\draw[line width=\st] (-0.7,-2.5)node[anchor=east]{$y-n$}--(6.2,-2.5);
\end{tikzpicture}}
$\to\cdots\to$
\scalebox{0.8}{
\begin{tikzpicture}[baseline=-33pt]
\fill[gray!30] (-0.7,-1.5)--(6.2,-1.5)--(6.2,-1.7)--(-0.7,-1.7)--(-0.7,-1.5);

\draw[line width=\st] (-0.7,-1.5)--(6.2,-1.5);
\draw[line width=\st] (-0.7,-1.7)--(6.2,-1.7);
\draw[line width=\tt,red] (-0.6,-1.6)--(0.5,-0.5)--(0.7,-0.7)--(0.9,-0.5)--(1.1,-0.7)--(1.3,-0.5)--(1.5,-0.7)--(1.6,-0.6);

\draw[fill,red] (1.75,-0.6) circle(0.02);
\draw[fill,red] (1.9,-0.6) circle(0.02);
\draw[fill,red] (2.05,-0.6) circle(0.02);

\draw[line width=\tt,red] (2.2,-0.6)--(2.3,-0.5)--(2.5,-0.7)--(2.7,-0.5)--(2.9,-0.7)--(3.1,-0.5)--(3.3,-0.7)--(5,1)--(5.1,0.9);

\draw[line width=\tt,red,dashed] (0,-1)--(1.5,-2.5)--(3.5,-0.5)--(2,1)--(0,-1);

\draw[line width=\st] (-0.7,1)--(6.2,1);
\draw[line width=\st] (-0.7,-1)--(6.2,-1);
\draw[line width=\st] (-0.7,-0.5)--(6.2,-0.5);
\draw[line width=\st] (-0.7,-2.5)--(6.2,-2.5);
\end{tikzpicture}}
$\to\cdots$
\end{center}

\smallskip

\begin{center}
$\to$
\scalebox{0.8}{
\begin{tikzpicture}[baseline=-33pt]
\fill[gray!30] (-0.7,-1.5)--(6.2,-1.5)--(6.2,-1.7)--(-0.7,-1.7)--(-0.7,-1.5);

\draw[line width=\st] (-0.7,-1.5)--(6.2,-1.5);
\draw[line width=\st] (-0.7,-1.7)--(6.2,-1.7);
\draw[line width=\tt,red] (-0.6,-1.6)--(0,-1)--(0.2,-1.2)--(0.4,-1)--(0.6,-1.2)--(0.8,-1)--(1,-1.2)--(1.1,-1.1);

\draw[fill,red] (1.25,-1.1) circle(0.02);
\draw[fill,red] (1.4,-1.1) circle(0.02);
\draw[fill,red] (1.55,-1.1) circle(0.02);

\draw[line width=\tt,red] (1.7,-1.1)--(1.8,-1)--(2,-1.2)--(2.2,-1)--(2.4,-1.2)--(2.6,-1)--(2.8,-1.2)--(5,1)--(5.1,0.9);

\draw[line width=\tt,red,dashed] (0,-1)--(1.5,-2.5)--(3.5,-0.5)--(2,1)--(0,-1);

\draw[line width=\st] (-0.7,1)node[anchor=east]{$y$}--(6.2,1);
\draw[line width=\st] (-0.7,-1)node[anchor=east]{$x$}--(6.2,-1);
\draw[line width=\st] (-0.7,-0.5)node[anchor=east]{$2y-x-n$}--(6.2,-0.5);
\draw[line width=\st] (-0.7,-2.5)node[anchor=east]{$y-n$}--(6.2,-2.5);
\end{tikzpicture}}
$\to\cdots\to$
\scalebox{0.8}{
\begin{tikzpicture}[baseline=-33pt]
\fill[gray!30] (-0.7,-1.5)--(6.2,-1.5)--(6.2,-1.7)--(-0.7,-1.7)--(-0.7,-1.5);

\draw[line width=\st] (-0.7,-1.5)--(6.2,-1.5);
\draw[line width=\st] (-0.7,-1.7)--(6.2,-1.7);
\draw[line width=\tt,red] (-0.6,-1.6)--(0,-1)--(0.7,-1.7)--(0.9,-1.5)--(1.1,-1.7)--(1.3,-1.5)--(1.4,-1.6); 

\draw[fill,red] (1.5,-1.6) circle(0.02);
\draw[fill,red] (1.6,-1.6) circle(0.02);
\draw[fill,red] (1.7,-1.6) circle(0.02);

\draw[line width=\tt,red] (1.8,-1.6)--(1.9,-1.7)--(2.1,-1.5)--(2.3,-1.7)--(5,1)--(5.1,0.9);

\draw[line width=\tt,red,dashed] (0,-1)--(1.5,-2.5)--(3.5,-0.5)--(2,1)--(0,-1);

\draw[line width=\st] (-0.7,1)--(6.2,1);
\draw[line width=\st] (-0.7,-1)--(6.2,-1);
\draw[line width=\st] (-0.7,-0.5)--(6.2,-0.5);
\draw[line width=\st] (-0.7,-2.5)--(6.2,-2.5);
\end{tikzpicture}}
\end{center}

\medskip
\noindent
This process takes $y-1$ steps and is described by the signatures $\u{\sigma}^{(y)},\ldots,\u{\sigma}^{(1)}$
given by
$$
\u{\sigma}^{(i)}=
\left\{
\begin{array}{ll}
\u{\sigma}&\text{if }i=y;\\[6pt]
(\ldots\underbrace{0,\ldots,0}_{i,\ldots,2y-i}\ldots\underbrace{0}_{x+n}\ldots0)&\text{if }2y-x-n<i<y;\\[8pt]
(\ldots\!\!\!\!\!\!\!\!\!\!\!\underbrace{0,\ldots,0}_{\qquad\quad i,\ldots,2x-2y+2n-1+i}\!\!\!\!\!\!\!\!\!\!\!\!\!\!\!\!\!\;\;\;\;\;\;\;\;\ldots\;\;\;\;\;\;\;\;\;0)&\text{if }x\le i\le 2y-x-n;\\[8pt]
(\ldots\quad\underbrace{0}_x\;\;\;\;\ldots\!\!\!\!\underbrace{0,\ldots,0}_{2x+1-i,\ldots,2x-2y+2n-1+i}\!\!\!\!\!\!\!\!\!\!\!\!\;\;\;\;\;\;\;\ldots\;\;\;\;\;0)&\text{if }1\le i<x.
\end{array}
\right.
$$

The diagram $\widehat D$ can be reduced to the same diagram as follows:

\medskip

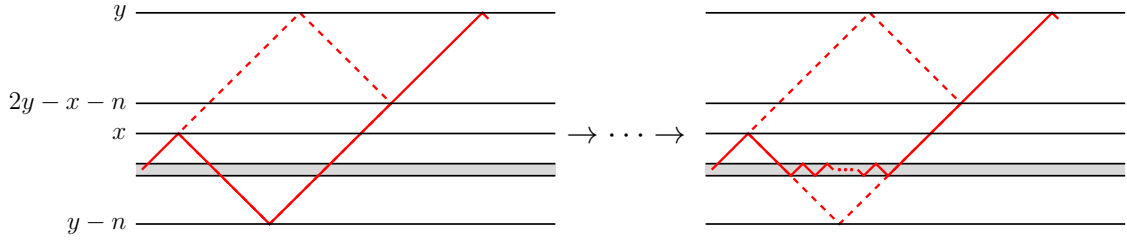
\begin{center}
\scalebox{0.8}{
\begin{tikzpicture}[baseline=-33pt]
\fill[gray!30] (-0.7,-1.5)--(6.2,-1.5)--(6.2,-1.7)--(-0.7,-1.7)--(-0.7,-1.5);

\draw[line width=\st] (-0.7,-1.5)--(6.2,-1.5);
\draw[line width=\st] (-0.7,-1.7)--(6.2,-1.7);
\draw[line width=\tt,red] (-0.6,-1.6)--(0,-1)--(1.5,-2.5)--(3.5,-0.5)--(5,1)--(5.1,0.9);

\draw[line width=\tt,red,dashed] (0,-1)--(1.5,-2.5)--(3.5,-0.5)--(2,1)--(0,-1);

\draw[line width=\st] (-0.7,1)node[anchor=east]{$y$}--(6.2,1);
\draw[line width=\st] (-0.7,-1)node[anchor=east]{$x$}--(6.2,-1);
\draw[line width=\st] (-0.7,-0.5)node[anchor=east]{$2y-x-n$}--(6.2,-0.5);
\draw[line width=\st] (-0.7,-2.5)node[anchor=east]{$y-n$}--(6.2,-2.5);
\end{tikzpicture}}
$\to\cdots\to$
\scalebox{0.8}{
\begin{tikzpicture}[baseline=-33pt]
\fill[gray!30] (-0.7,-1.5)--(6.2,-1.5)--(6.2,-1.7)--(-0.7,-1.7)--(-0.7,-1.5);

\draw[line width=\st] (-0.7,-1.5)--(6.2,-1.5);
\draw[line width=\st] (-0.7,-1.7)--(6.2,-1.7);
\draw[line width=\tt,red] (-0.6,-1.6)--(0,-1)--(0.7,-1.7)--(0.9,-1.5)--(1.1,-1.7)--(1.3,-1.5)--(1.4,-1.6); 

\draw[fill,red] (1.5,-1.6) circle(0.02);
\draw[fill,red] (1.6,-1.6) circle(0.02);
\draw[fill,red] (1.7,-1.6) circle(0.02);

\draw[line width=\tt,red] (1.8,-1.6)--(1.9,-1.7)--(2.1,-1.5)--(2.3,-1.7)--(5,1)--(5.1,0.9);

\draw[line width=\tt,red,dashed] (0,-1)--(1.5,-2.5)--(3.5,-0.5)--(2,1)--(0,-1);

\draw[line width=\st] (-0.7,1)--(6.2,1);
\draw[line width=\st] (-0.7,-1)--(6.2,-1);
\draw[line width=\st] (-0.7,-0.5)--(6.2,-0.5);
\draw[line width=\st] (-0.7,-2.5)--(6.2,-2.5);
\end{tikzpicture}}
\end{center}

\medskip

\noindent
This process takes $n-y$ steps and is described by the signatures $\u{\sigma}^{(y-n)},\ldots,\u{\sigma}^{(0)}$
given by

$$
\u{\sigma}^{(i)}=
\left\{
\begin{array}{ll}
\u{\hat\sigma}&\text{if }i=y-n;\\[6pt]
(\ldots\underbrace{0}_{x}\;\;\;\ldots\!\!\!\underbrace{0,\ldots,0}_{2x-i,\ldots,2x-2y+2n+i}\!\!\!\!\!\!\!\!\!\;\;\;\;\;\ldots\;\;\;\;0)&\text{if }y-n<i\le0.
\end{array}
\right.
$$

{\it Case 2:} $x>2y-x-n>0$. The diagram $D$
can be reduced as follows:

\medskip

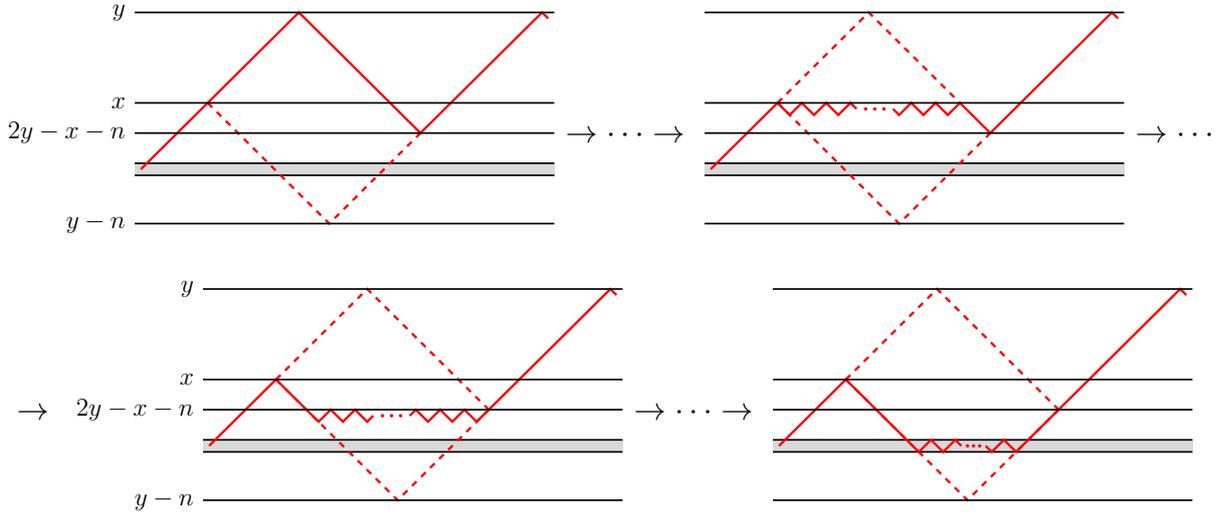
\begin{center}
\scalebox{0.8}{
\begin{tikzpicture}[baseline=-33pt]
\fill[gray!30] (-0.7,-1.5)--(6.2,-1.5)--(6.2,-1.7)--(-0.7,-1.7)--(-0.7,-1.5);

\draw[line width=\st] (-0.7,-1.5)--(6.2,-1.5);
\draw[line width=\st] (-0.7,-1.7)--(6.2,-1.7);
\draw[line width=\tt,red] (-0.6,-1.6)--(2,1)--(4,-1)--(6,1)--(6.1,0.9);

\draw[line width=\tt,red,dashed] (0.5,-0.5)--(2.5,-2.5)--(4,-1);

\draw[line width=\st] (-0.7,1)node[anchor=east]{$y$}--(6.2,1);
\draw[line width=\st] (-0.7,-1)node[anchor=east]{$2y-x-n$}--(6.2,-1);
\draw[line width=\st] (-0.7,-0.5)node[anchor=east]{$x$}--(6.2,-0.5);
\draw[line width=\st] (-0.7,-2.5)node[anchor=east]{$y-n$}--(6.2,-2.5);
\end{tikzpicture}}
$\to\cdots\to$
\scalebox{0.8}{
\begin{tikzpicture}[baseline=-33pt]
\fill[gray!30] (-0.7,-1.5)--(6.2,-1.5)--(6.2,-1.7)--(-0.7,-1.7)--(-0.7,-1.5);

\draw[line width=\st] (-0.7,-1.5)--(6.2,-1.5);
\draw[line width=\st] (-0.7,-1.7)--(6.2,-1.7);
\draw[line width=\tt,red] (-0.6,-1.6)--(0.5,-0.5)--(0.7,-0.7)--(0.9,-0.5)--(1.1,-0.7)--(1.3,-0.5)--(1.5,-0.7)--(1.7,-0.5)--(1.8,-0.6);
\draw[line width=\tt,red,dashed] (0.5,-0.5)--(2,1)--(3.5,-0.5);
\draw[fill,red] (1.95,-0.6) circle(0.02);
\draw[fill,red] (2.1,-0.6) circle(0.02);
\draw[fill,red] (2.25,-0.6) circle(0.02);
\draw[line width=\tt,red] (2.4,-0.6)--(2.5,-0.7)--(2.7,-0.5)--(2.9,-0.7)--(3.1,-0.5)--(3.3,-0.7)--(3.5,-0.5)--(4,-1)--(6,1)--(6.1,0.9);

\draw[line width=\tt,red,dashed] (0.5,-0.5)--(2.5,-2.5)--(4,-1);

\draw[line width=\st] (-0.7,1)--(6.2,1);
\draw[line width=\st] (-0.7,-1)--(6.2,-1);
\draw[line width=\st] (-0.7,-0.5)--(6.2,-0.5);
\draw[line width=\st] (-0.7,-2.5)--(6.2,-2.5);
\end{tikzpicture}}
$\to\cdots$
\end{center}

\smallskip

\begin{center}
$\to$
\scalebox{0.8}{
\begin{tikzpicture}[baseline=-33pt]
\fill[gray!30] (-0.7,-1.5)--(6.2,-1.5)--(6.2,-1.7)--(-0.7,-1.7)--(-0.7,-1.5);

\draw[line width=\st] (-0.7,-1.5)--(6.2,-1.5);
\draw[line width=\st] (-0.7,-1.7)--(6.2,-1.7);
\draw[line width=\tt,red] (-0.6,-1.6)--(0.5,-0.5)--(1,-1)--(1.2,-1.2)--(1.4,-1)--(1.6,-1.2)--(1.8,-1)--(2,-1.2)--(2.1,-1.1);
\draw[line width=\tt,red,dashed] (0.5,-0.5)--(2,1)--(4,-1);
\draw[fill,red] (2.25,-1.1) circle(0.02);
\draw[fill,red] (2.40,-1.1) circle(0.02);
\draw[fill,red] (2.55,-1.1) circle(0.02);
\draw[line width=\tt,red] (2.7,-1.1)--(2.8,-1)--(3,-1.2)--(3.2,-1)--(3.4,-1.2)--(3.6,-1)--(3.8,-1.2)--(4,-1)--(6,1)--(6.1,0.9);

\draw[line width=\tt,red,dashed] (0.5,-0.5)--(2.5,-2.5)--(4,-1);

\draw[line width=\st] (-0.7,1)node[anchor=east]{$y$}--(6.2,1);
\draw[line width=\st] (-0.7,-1)node[anchor=east]{$2y-x-n$}--(6.2,-1);
\draw[line width=\st] (-0.7,-0.5)node[anchor=east]{$x$}--(6.2,-0.5);
\draw[line width=\st] (-0.7,-2.5)node[anchor=east]{$y-n$}--(6.2,-2.5);
\end{tikzpicture}}
$\to\cdots\to$
\scalebox{0.8}{
\begin{tikzpicture}[baseline=-33pt]
\fill[gray!30] (-0.7,-1.5)--(6.2,-1.5)--(6.2,-1.7)--(-0.7,-1.7)--(-0.7,-1.5);

\draw[line width=\st] (-0.7,-1.5)--(6.2,-1.5);
\draw[line width=\st] (-0.7,-1.7)--(6.2,-1.7);
\draw[line width=\tt,red] (-0.6,-1.6)--(0.5,-0.5)--(1,-1)--(1.7,-1.7)--(1.9,-1.5)--(2.1,-1.7)--(2.3,-1.5)--(2.4,-1.6);
\draw[line width=\tt,red,dashed] (0.5,-0.5)--(2,1)--(4,-1);
\draw[fill,red] (2.5,-1.6) circle(0.02);
\draw[fill,red] (2.6,-1.6) circle(0.02);
\draw[fill,red] (2.7,-1.6) circle(0.02);
\draw[line width=\tt,red] (2.8,-1.6)--(2.9,-1.7)--(3.1,-1.5)--(3.3,-1.7)--(4,-1)--(6,1)--(6.1,0.9);

\draw[line width=\tt,red,dashed] (0.5,-0.5)--(2.5,-2.5)--(4,-1);

\draw[line width=\st] (-0.7,1)--(6.2,1);
\draw[line width=\st] (-0.7,-1)--(6.2,-1);
\draw[line width=\st] (-0.7,-0.5)--(6.2,-0.5);
\draw[line width=\st] (-0.7,-2.5)--(6.2,-2.5);
\end{tikzpicture}}
\end{center}

\medskip
\noindent
This process takes $y-1$ steps and is described by the signatures $\u{\sigma}^{(y)},\ldots,\u{\sigma}^{(1)}$
given by
$$
\u{\sigma}^{(i)}=
\left\{
\begin{array}{ll}
\u{\sigma}&\text{if }i=y;\\[6pt]
(\ldots\underbrace{0,\ldots,0}_{i,\ldots,2y-i}\ldots\underbrace{0}_{x+n}\ldots0)&\text{if }x\le i<y;\\[8pt]
(\ldots\underbrace{0}_{x}\;\;\ldots\!\!\underbrace{0,\ldots,0}_{2x+1-i,\ldots,2y-i}\!\!\!\!\;\;\ldots\;\;\underbrace{0}_{x+n}\ldots0)&\text{if }2y-x-n<i<x;\\[8pt]
(\ldots\underbrace{0}_{x}\;\;\;\;\;\ldots\!\!\!\!\!\!\underbrace{0,\ldots,0}_{2x+1-i,\ldots,2x-2y+2n-1+i}\!\!\!\!\!\!\!\!\!\!\!\!\!\!\!\;\;\;\;\;\;\;\ldots\;\;\;\;\;\;\;\;0)&\text{if }1\le i\le2y-x-n;\\[8pt]
\end{array}
\right.
$$

The diagram $\widehat D$ can be reduced to the same diagram as follows:

\medskip

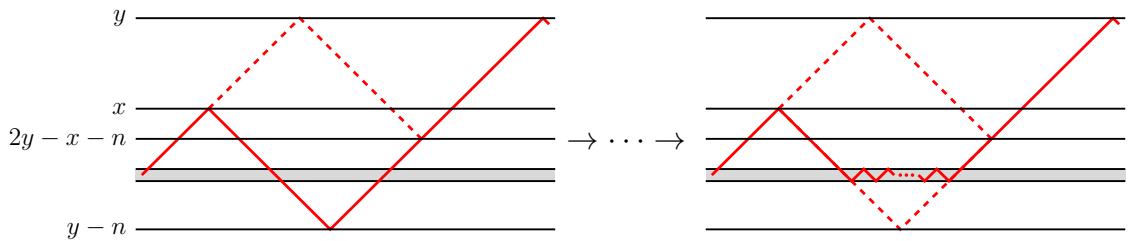
\def\tt{1.3pt}
\def\st{0.9pt}

\begin{center}
\scalebox{0.8}{
\begin{tikzpicture}[baseline=-33pt]
\fill[gray!30] (-0.7,-1.5)--(6.2,-1.5)--(6.2,-1.7)--(-0.7,-1.7)--(-0.7,-1.5);

\draw[line width=\st] (-0.7,-1.5)--(6.2,-1.5);
\draw[line width=\st] (-0.7,-1.7)--(6.2,-1.7);

\draw[line width=\tt,red] (-0.6,-1.6)--(0.5,-0.5)--(2.5,-2.5)--(4,-1)--(6,1)--(6.1,0.9);

\draw[line width=\tt,red,dashed] (0.5,-0.5)--(2,1)--(4,-1);

\draw[line width=\st] (-0.7,1)node[anchor=east]{$y$}--(6.2,1);
\draw[line width=\st] (-0.7,-1)node[anchor=east]{$2y-x-n$}--(6.2,-1);
\draw[line width=\st] (-0.7,-0.5)node[anchor=east]{$x$}--(6.2,-0.5);
\draw[line width=\st] (-0.7,-2.5)node[anchor=east]{$y-n$}--(6.2,-2.5);
\end{tikzpicture}}
$\to\cdots\to$
\scalebox{0.8}{
\begin{tikzpicture}[baseline=-33pt]
\fill[gray!30] (-0.7,-1.5)--(6.2,-1.5)--(6.2,-1.7)--(-0.7,-1.7)--(-0.7,-1.5);

\draw[line width=\st] (-0.7,-1.5)--(6.2,-1.5);
\draw[line width=\st] (-0.7,-1.7)--(6.2,-1.7);
\draw[line width=\tt,red] (-0.6,-1.6)--(0.5,-0.5)--(1,-1)--(1.7,-1.7)--(1.9,-1.5)--(2.1,-1.7)--(2.3,-1.5)--(2.4,-1.6);
\draw[line width=\tt,red,dashed] (0.5,-0.5)--(2,1)--(4,-1);
\draw[fill,red] (2.5,-1.6) circle(0.02);
\draw[fill,red] (2.6,-1.6) circle(0.02);
\draw[fill,red] (2.7,-1.6) circle(0.02);
\draw[line width=\tt,red] (2.8,-1.6)--(2.9,-1.7)--(3.1,-1.5)--(3.3,-1.7)--(4,-1)--(6,1)--(6.1,0.9);

\draw[line width=\tt,red,dashed] (0.5,-0.5)--(2.5,-2.5)--(4,-1);

\draw[line width=\st] (-0.7,1)--(6.2,1);
\draw[line width=\st] (-0.7,-1)--(6.2,-1);
\draw[line width=\st] (-0.7,-0.5)--(6.2,-0.5);
\draw[line width=\st] (-0.7,-2.5)--(6.2,-2.5);
\end{tikzpicture}}
\end{center}

\medskip

\noindent
This process takes $n-y$ steps and is described by the signatures $\u{\sigma}^{(y-n)},\u{\sigma}^{(y-n+1)}\ldots,\u{\sigma}^{(0)}$,
which are given by the same formulas as in case 1.

{\it Case 3:} $2y-x-n\le 0$. The diagram $D$ can be reduced as follows:

\medskip

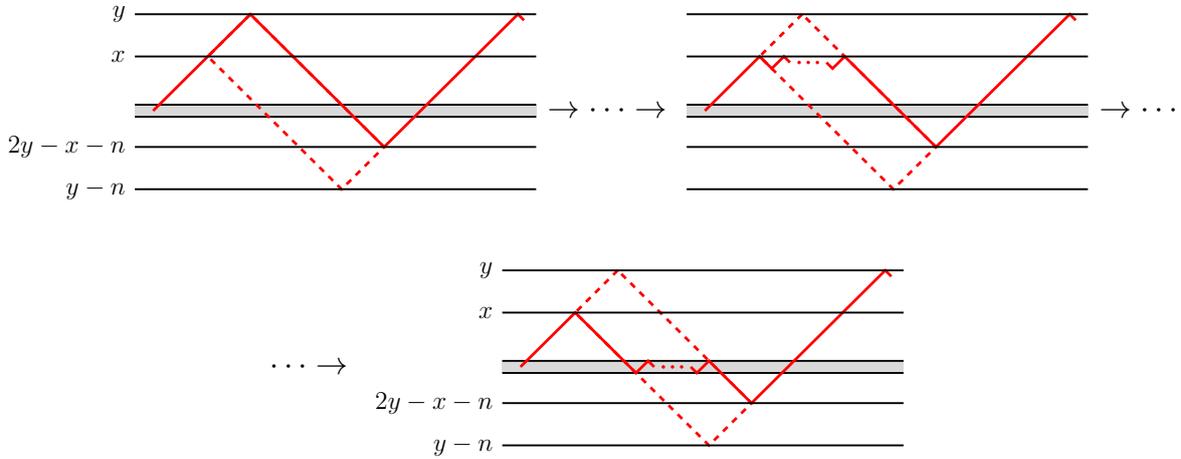
\begin{center}
\scalebox{0.8}{
\begin{tikzpicture}[baseline=-43pt]
\fill[gray!30] (-0.7,-1.3)--(5.9,-1.3)--(5.9,-1.5)--(-0.7,-1.5)--(-0.7,-1.3);

\draw[line width=\st] (-0.7,-1.3)--(5.9,-1.3);
\draw[line width=\st] (-0.7,-1.5)--(5.9,-1.5);

\draw[line width=\tt,red] (-0.4,-1.4)--(0.5,-0.5)--(1.2,0.2)--(3.4,-2)--(5.6,0.2)--(5.7,0.1);

\draw[line width=\tt,red,dashed](1.2,0.2)--(0.5,-0.5)--(2.7,-2.7)--(3.4,-2)--(1.2,0.2);

\draw[line width=\st] (-0.7,0.2)node[anchor=east]{$y$}--(5.9,0.2);
\draw[line width=\st] (-0.7,-2)node[anchor=east]{$2y-x-n$}--(5.9,-2);
\draw[line width=\st] (-0.7,-0.5)node[anchor=east]{$x$}--(5.9,-0.5);
\draw[line width=\st] (-0.7,-2.7)node[anchor=east]{$y-n$}--(5.9,-2.7);
\end{tikzpicture}}
$\to\cdots\to$
\scalebox{0.8}{
\begin{tikzpicture}[baseline=-43pt]
\fill[gray!30] (-0.7,-1.3)--(5.9,-1.3)--(5.9,-1.5)--(-0.7,-1.5)--(-0.7,-1.3);

\draw[line width=\st] (-0.7,-1.3)--(5.9,-1.3);
\draw[line width=\st] (-0.7,-1.5)--(5.9,-1.5);

\draw[line width=\tt,red](-0.4,-1.4)--(0.5,-0.5)--(0.7,-0.7)--(0.9,-0.5)--(1,-0.6);

\draw[fill,red] (1.15,-0.6) circle(0.02);
\draw[fill,red] (1.3,-0.6) circle(0.02);
\draw[fill,red] (1.45,-0.6) circle(0.02);

\draw[line width=\tt,red] (1.6,-0.6)--(1.7,-0.7)--(1.9,-0.5)--(3.4,-2)--(5.6,0.2)--(5.7,0.1);

\draw[line width=\tt,red,dashed](1.2,0.2)--(0.5,-0.5)--(2.7,-2.7)--(3.4,-2)--(1.2,0.2);

\draw[line width=\st] (-0.7,0.2)--(5.9,0.2);
\draw[line width=\st] (-0.7,-2)--(5.9,-2);
\draw[line width=\st] (-0.7,-0.5)--(5.9,-0.5);
\draw[line width=\st] (-0.7,-2.7)--(5.9,-2.7);
\end{tikzpicture}}
$\to\cdots$
\end{center}

\smallskip

\begin{center}
$\cdots\to$
\scalebox{0.8}{
\begin{tikzpicture}[baseline=-43pt]
\fill[gray!30] (-0.7,-1.3)--(5.9,-1.3)--(5.9,-1.5)--(-0.7,-1.5)--(-0.7,-1.3);

\draw[line width=\st] (-0.7,-1.3)--(5.9,-1.3);
\draw[line width=\st] (-0.7,-1.5)--(5.9,-1.5);

\draw[line width=\tt,red](-0.4,-1.4)--(0.5,-0.5)--(1.5,-1.5)--(1.7,-1.3)--(1.8,-1.4);

\draw[fill,red] (1.95,-1.4) circle(0.02);
\draw[fill,red] (2.1,-1.4) circle(0.02);
\draw[fill,red] (2.25,-1.4) circle(0.02);

\draw[line width=\tt,red] (2.4,-1.4)--(2.5,-1.5)--(2.7,-1.3)--(3.4,-2)--(5.6,0.2)--(5.7,0.1);

\draw[line width=\tt,red,dashed](1.2,0.2)--(0.5,-0.5)--(2.7,-2.7)--(3.4,-2)--(1.2,0.2);

\draw[line width=\st] (-0.7,0.2)node[anchor=east]{$y$}--(5.9,0.2);
\draw[line width=\st] (-0.7,-2)node[anchor=east]{$2y-x-n$}--(5.9,-2);
\draw[line width=\st] (-0.7,-0.5)node[anchor=east]{$x$}--(5.9,-0.5);
\draw[line width=\st] (-0.7,-2.7)node[anchor=east]{$y-n$}--(5.9,-2.7);
\end{tikzpicture}}

\end{center}

\noindent
This process takes $y-1$ steps and is described by the signatures $\u{\sigma}^{(y)},\ldots,\u{\sigma}^{(1)}$
given by
$$
\u{\sigma}^{(i)}=
\left\{
\begin{array}{ll}
\u{\sigma}&\text{if }i=y;\\[6pt]
(\ldots\underbrace{0,\ldots,0}_{i,\ldots,2y-i}\ldots\underbrace{0}_{x+n}\ldots0)&\text{if }x\le i<y;\\[8pt]
(\ldots\underbrace{0}_{x}\;\;\ldots\!\!\underbrace{0,\ldots,0}_{2x+1-i,\ldots,2y-i}\!\!\!\!\;\;\ldots\;\;\underbrace{0}_{x+n}\ldots0)&\text{if }1\le i<x;\\[8pt]
\end{array}
\right.
$$

The diagram $\widehat D$ can be reduced to the same diagram as follows:
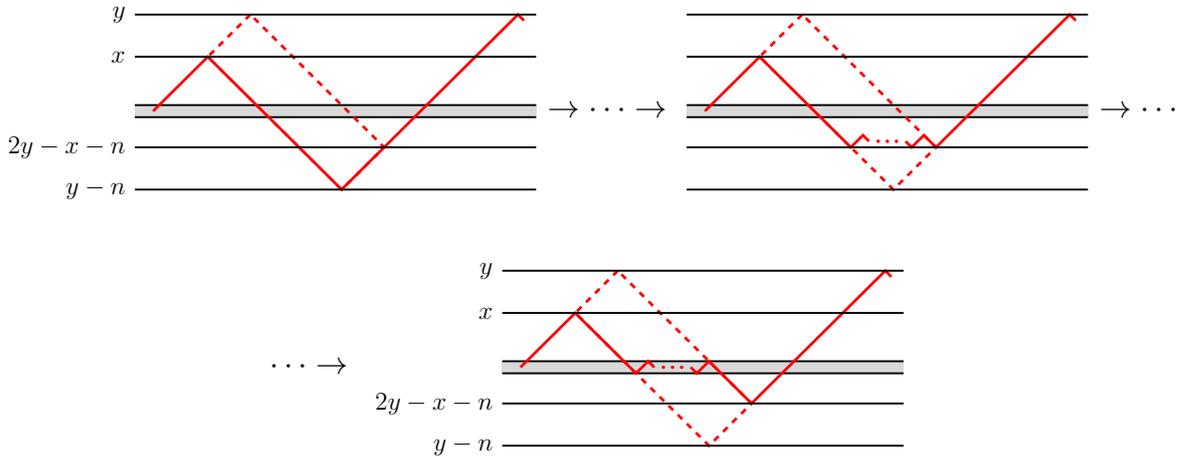
\begin{center}
\scalebox{0.8}{
\begin{tikzpicture}[baseline=-43pt]
\fill[gray!30] (-0.7,-1.3)--(5.9,-1.3)--(5.9,-1.5)--(-0.7,-1.5)--(-0.7,-1.3);

\draw[line width=\st] (-0.7,-1.3)--(5.9,-1.3);
\draw[line width=\st] (-0.7,-1.5)--(5.9,-1.5);

\draw[line width=\tt,red] (-0.4,-1.4)--(0.5,-0.5)--(2.7,-2.7)--(3.4,-2)--(5.6,0.2)--(5.7,0.1);

\draw[line width=\tt,red,dashed](1.2,0.2)--(0.5,-0.5)--(2.7,-2.7)--(3.4,-2)--(1.2,0.2);

\draw[line width=\st] (-0.7,0.2)node[anchor=east]{$y$}--(5.9,0.2);
\draw[line width=\st] (-0.7,-2)node[anchor=east]{$2y-x-n$}--(5.9,-2);
\draw[line width=\st] (-0.7,-0.5)node[anchor=east]{$x$}--(5.9,-0.5);
\draw[line width=\st] (-0.7,-2.7)node[anchor=east]{$y-n$}--(5.9,-2.7);
\end{tikzpicture}}
$\to\cdots\to$
\scalebox{0.8}{
\begin{tikzpicture}[baseline=-43pt]
\fill[gray!30] (-0.7,-1.3)--(5.9,-1.3)--(5.9,-1.5)--(-0.7,-1.5)--(-0.7,-1.3);

\draw[line width=\st] (-0.7,-1.3)--(5.9,-1.3);
\draw[line width=\st] (-0.7,-1.5)--(5.9,-1.5);

\draw[line width=\tt,red] (-0.4,-1.4)--(0.5,-0.5)--(2,-2)--(2.2,-1.8)--(2.3,-1.9);

\draw[fill,red] (2.45,-1.9) circle(0.02);
\draw[fill,red] (2.6,-1.9) circle(0.02);
\draw[fill,red] (2.75,-1.9) circle(0.02);

\draw[line width=\tt,red] (2.9,-1.9)--(3,-2)--(3.2,-1.8)--(3.4,-2)--(5.6,0.2)--(5.7,0.1);

\draw[line width=\tt,red,dashed](1.2,0.2)--(0.5,-0.5)--(2.7,-2.7)--(3.4,-2)--(1.2,0.2);

\draw[line width=\st] (-0.7,0.2)--(5.9,0.2);
\draw[line width=\st] (-0.7,-2)--(5.9,-2);
\draw[line width=\st] (-0.7,-0.5)--(5.9,-0.5);
\draw[line width=\st] (-0.7,-2.7)--(5.9,-2.7);
\end{tikzpicture}}
$\to\cdots$
\end{center}

\smallskip

\begin{center}
$\cdots\to$
\scalebox{0.8}{
\begin{tikzpicture}[baseline=-43pt]
\fill[gray!30] (-0.7,-1.3)--(5.9,-1.3)--(5.9,-1.5)--(-0.7,-1.5)--(-0.7,-1.3);

\draw[line width=\st] (-0.7,-1.3)--(5.9,-1.3);
\draw[line width=\st] (-0.7,-1.5)--(5.9,-1.5);

\draw[line width=\tt,red](-0.4,-1.4)--(0.5,-0.5)--(1.5,-1.5)--(1.7,-1.3)--(1.8,-1.4);

\draw[fill,red] (1.95,-1.4) circle(0.02);
\draw[fill,red] (2.1,-1.4) circle(0.02);
\draw[fill,red] (2.25,-1.4) circle(0.02);

\draw[line width=\tt,red] (2.4,-1.4)--(2.5,-1.5)--(2.7,-1.3)--(3.4,-2)--(5.6,0.2)--(5.7,0.1);

\draw[line width=\tt,red,dashed](1.2,0.2)--(0.5,-0.5)--(2.7,-2.7)--(3.4,-2)--(1.2,0.2);

\draw[line width=\st] (-0.7,0.2)node[anchor=east]{$y$}--(5.9,0.2);
\draw[line width=\st] (-0.7,-2)node[anchor=east]{$2y-x-n$}--(5.9,-2);
\draw[line width=\st] (-0.7,-0.5)node[anchor=east]{$x$}--(5.9,-0.5);
\draw[line width=\st] (-0.7,-2.7)node[anchor=east]{$y-n$}--(5.9,-2.7);
\end{tikzpicture}}

\end{center}

\medskip

\noindent
This process takes $n-y$ steps and is described by the signatures $\u{\sigma}^{(y-n)},\ldots,\u{\sigma}^{(0)}$
given by
$$
\u{\sigma}^{(i)}=
\left\{
\begin{array}{ll}
\u{\hat\sigma}&\text{if }i=y-n;\\[6pt]
(\ldots\underbrace{0}_{x}\;\;\;\ldots\!\!\!\underbrace{0,\ldots,0}_{2x-i,\ldots,2x-2y+2n+i}\!\!\!\!\!\!\!\!\!\;\;\;\;\ldots\;\;\;0)&\text{if }y-n<i\le 2y-x-n;\\[6pt]
(\ldots\underbrace{0}_{x}\;\;\;\ldots\!\underbrace{0,\ldots,0}_{2x-i,\ldots,2y-1-i}\!\!\!\!\;\;\;\ldots\;\;\;\underbrace{0}_{x+n}\ldots0)&\text{if }2y-x-n<i\le0.
\end{array}
\right.
$$
Note that the sequence $\u{\sigma}^{(2y-x-n+1)},\ldots,\u{\sigma}^{(0)}$ may be empty if $2y-x-n=0$.

Let $w=\u{t}^{\u{\delta}}$.
We get the following cycle in $\Sub(\u{t},w)$:
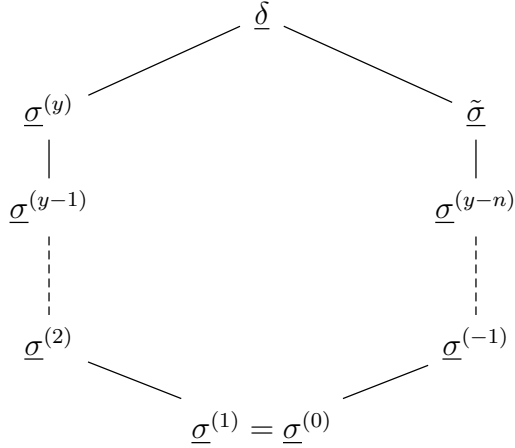
$$
\begin{tikzcd}[row sep=30pt]
&\u{\delta}\arrow[no head]{rd}\arrow[no head]{ld}&\\[-15pt]
\u{\sigma}^{(y)}\arrow[no head]{d}&&\u{\tilde\sigma}\arrow[no head]{d}\\[-15pt]
\u{\sigma}^{(y-1)}\arrow[dashed,no head]{d}&&\u{\sigma}^{(y-n)}\arrow[dashed,no head]{d}\\
\u{\sigma}^{(2)}\arrow[no head]{rd}&&\u{\sigma}^{(-1)}\arrow[no head]{ld}\\[-20pt]
&\u{\sigma}^{(1)}=\u{\sigma}^{(0)}&
\end{tikzcd}
$$
of length $n+2$ with maximal vertex $\u{\delta}$.
Both statements can be proved, applying~(\ref{eq:cycb6:6}) and Corollary~\ref{corollary:cycl:3},
exacly as for the cycles of the first type with the exception of
a possible equality $\u{\sigma}^{(i)}=\u{\tilde\sigma}$, where $1\le i\le y$.
However this equality can never hold, as the last element of $\u{\sigma}^{(i)}$
equals $0$ while the last element of $\u{\tilde\sigma}$ equals $1$.
We denote this cycle by $\Cyc^2_\c(x,y)$.

\subsection{Reducing a special vertex}\label{Reducing_a_special_vertex} Let $\u{s}$ be a finite sequence in $A$ and $B$ and
$\u{\gamma}\subset\u{s}$. Let $(i,k)$ and $(j,l)$
be special $\mu$- and $\lm$-pairs respectively for $\u{\gamma}$. Thus $\mu,\lm\in\Phi_\D^+$.
We assume that $i<j<k<l$.
Note that $\lm\ne\mu$, as otherwise $\u{\gamma}^{\to j}=-\mu$, which violates condition~\ref{definition:cycb6:3:p:iii} of
Definition~\ref{definition:cycb6:3}.

First consider the case $n<\infty$. Let $f:[0,1]\to S^1_{n/\pi}$ 
be an alcove walk
generating (see Definition~\ref{definition:cycb6:1}) the gallery
$$
\u{\Gamma}=(\mathscr C_1,\H_{\u{\gamma}^{\to1}},\mathscr C_2,\H_{\u{\gamma}^{\to2}},\ldots,\mathscr C_m,\H_{\u{\gamma}^{\to m}},\mathscr C_{m+1}).
$$
We define $O=f(0)$. It belongs to the fundamental arch chamber $\mathscr F$.
Without loss of generality, we may assume that $O$ is its center.
Let $\H_\mu\cap S_{n/\pi}^1=\{M_1,M_2\}$ and $\H_\lm\cap S_{n/\pi}^1=\{\Lm_1,\Lm_2\}$.
So we have five distinct points $O,M_1,M_2,\Lm_1,\Lm_2$.
To eliminate uncertainty, we choose $M_1$ and $\Lm_1$ so that these points are not separated from
the fundamental chamber by the hyperplanes $\H_\lm$ and $\H_\mu$ respectively:

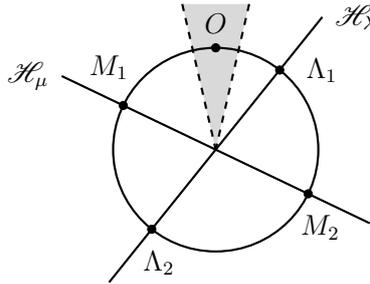
\begin{center}
\scalebox{0.9}{
\begin{tikzpicture}[baseline=-63pt]
\fill[gray!30] (0,0)--(-.4895460546, 2.144841407)--(.4895460546, 2.1448414075);
\draw[line width=\st] (0,0) circle(1.5);
\draw[line width=\st,dashed] (-.4895460546, 2.144841407) -- (0,0);
\draw[line width=\st,dashed] (.4895460546, 2.144841407) -- (0,0);
\draw[line width=\st] (-2.252422170, 1.084709348)node[left=1pt]{$\H_\mu$}--(2.252422170, -1.084709348);
\draw[fill] (-1.351453302, .6508256088) circle(0.06)node[left=6pt,above=7pt]{$M_1$};
\draw[fill] (1.351453302, -.6508256088) circle(0.06)node[right=5pt,below=5pt]{$M_2$};
\draw[fill] (0, 1.5) circle(0.06)node[above=2pt]{$O$};

\draw[fill] (-.9352347012, -1.172747225) circle(0.06)node[right=3pt,below=5pt]{$\Lm_2$};

\draw[fill] (.9352347012, 1.172747225) circle(0.06)node[right=7pt]{$\Lm_1$};
\draw[line width=\st] (-1.558724502, -1.954578708)--(1.558724502, 1.954578708)node[right=1pt]{$\H_\lm$};


\end{tikzpicture}}
\end{center}

\vspace{-8pt}

\noindent
Let $f^{-1}(P_n)=\{t_1<\cdots<t_N\}$, where $P_n=\Sigma_1(\D,\{A,B\})$ as in Section~\ref{Galleries_as_paths}.
As we have $f(t_i)\in\H_\mu$, we get $f(t_i)=M_1$ or $f(t_i)=M_2$.
Suppose that the latter equality holds. Then $f|_{[0,t_i]}$ intersects $\H_\lm$,
so there exists an index $i'<i$ such that $f(t_{i'})\in\H_\lm$.
Assuming that $i'$ is the maximal index with this property, we get $\H_\lm$ separates $\mathscr C_1$ from $f((t_{i'},t_i))$.
Thus $\H_\lm$, which is the $j$th wall of $\u{\Gamma}$, separates $\mathscr C_1$ from $\mathscr C_{i'+1}$.
It contradicts Lemma~\ref{lemma:cycl:2} applied to the special $\lm$-pair $(j,k)$.
Thus, we have proved that $f(t_i)=M_1$.

By Lemma~\ref{lemma:cycl:2}, we see that $\H_\mu$ separates $\mathscr C_1$ and
$$
f((t_i,t_{i+1})\cup(t_{i+1},t_{i+2})\cup\cdots\cup(t_{k-1},t_k))\subset \mathscr C_{i+1}\cup \mathscr C_{i+2}\cup\cdots\cup \mathscr C_k.
$$
As $t_j\in(t_i,t_k)$ and $f$ is continuous, we have only one possibility $f(t_j)=\Lambda_2$.
Similarly, by Lemma~\ref{lemma:cycl:2}, we see that $\H_\lm$ separates $\mathscr C_1$ and
$$
f((t_j,t_{j+1})\cup(t_{j+1},t_{j+2})\cup\cdots\cup(t_{l-1},t_l))\subset \mathscr C_{j+1}\cup \mathscr C_{j+2}\cup\cdots\cup \mathscr C_l.
$$
As $t_k\in(t_j,t_l)$ and $f$ is continuous, we get $f(t_k)=M_2$.
However, both cases $f(t_l)=\Lambda_1$ and $f(t_l)=\Lambda_2$ are possible.
We define $u=1$ or $u=2$, depending on which case holds. We denote by $O'$ the point at the centre
of the arc chamber attached to $\Lm_u$ that is not separated from $O$ by $\H_\lm$.

Next we count our parameters and points as follows:
$$
y_0=0,\quad y_1=t_i,\quad y_2=t_j,\quad y_3=t_k,\quad y_4=t_l,
$$
$$
Y_0=O,\quad Y_1=M_1,\quad Y_2=\Lm_2,\quad Y_3=M_2, \quad Y_4=\Lm_u,\quad Y_5=O',
$$
$$
Y'_1=\Lm_2,\quad Y'_2=M_2,\quad Y'_3=\Lm_1,\quad Y'_4=O.
$$
Let us introduce the numbers
\begin{equation}\label{eq:x}
x_0=1/2,
\quad\; x_i=x_{i-1}+|\arc{Y_{i-1}Y_i}|\;\text{ for }\;i=1,\ldots,5.
\end{equation}
and the restrictions $f_m=f|_{[y_{m-1},y_m]}$ for $m=1,\ldots,4$. We consider the paths to minor arcs
$$
g_m:[x_{m-1},x_m]\ito\arc{Y_{m-1}Y_m}
$$
having natural parametrizations such that $g_m(x_{m-1})=Y_{m-1}$ and $g_m(x_m)=Y_m$ for $m=1,\ldots,5$.
Applying Corollary~\ref{corollary:cycl:2} to $f_m$ and $g_m$ with $X=Y_{m-1}$, $Y=Y_m$, $Z=Y'_m$,
we obtain increasing maps $q_m:g_m^{-1}(P_n)\to f_m^{-1}(P_n)$
satisfying properties~\ref{corollary:cycl:2:p:i}--\ref{corollary:cycl:2:p:iv}
for $m=1,\ldots,4$. Note that the points $Y'_1,\ldots,Y'_4$ were chosen
to guarantee condition~\ref{corollary:cycl:2:p:4} of this corollary,
which follows from Lemma~\ref{lemma:cycl:2}.
We glue our maps as follows:
$$
g=g_1\cup g_2\cup g_3\cup g_4\cup g_5,\quad q=q_1\cup q_2\cup q_3\cup q_4.
$$
Below are the pictures representing $g$ for $u=1$ and $u=2$:

\smallskip

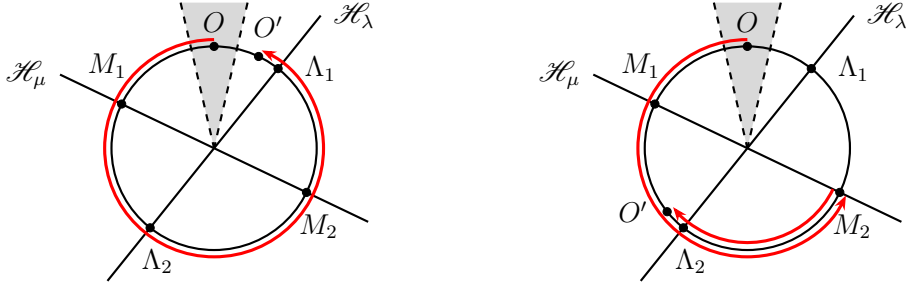
\begin{center}
\scalebox{0.9}{
\begin{tikzpicture}[baseline=-63pt]
\fill[gray!30] (0,0)--(-.4895460546, 2.144841407)--(.4895460546, 2.1448414075);
\draw[line width=\st] (0,0) circle(1.5);
\draw[line width=\st,dashed] (-.4895460546, 2.144841407) -- (0,0);
\draw[line width=\st,dashed] (.4895460546, 2.144841407) -- (0,0);
\draw[line width=\st] (-2.252422170, 1.084709348)node[left=1pt]{$\H_\mu$}--(2.252422170, -1.084709348);
\draw[fill] (-1.351453302, .6508256088) circle(0.06)node[left=6pt,above=7pt]{$M_1$};
\draw[fill] (1.351453302, -.6508256088) circle(0.06)node[right=5pt,below=5pt]{$M_2$};
\draw[fill] (0, 1.5) circle(0.06)node[above=2pt]{$O$};

\draw[fill] (-.9352347012, -1.172747225) circle(0.06)node[right=3pt,below=5pt]{$\Lm_2$};

\draw[fill] (.9352347012, 1.172747225) circle(0.06)node[right=7pt]{$\Lm_1$};
\draw[line width=\st] (-1.558724502, -1.954578708)--(1.558724502, 1.954578708)node[right=1pt]{$\H_\lm$};

\draw[fill] (.6508256091, 1.351453302)node[right=4pt, above=4pt]{$O'$} circle(0.06);

\draw[-{Stealth[length=\sl]},line width=\tt,red] (0,1.6) arc (90:424.2857143:1.6);

\end{tikzpicture}}
\qquad\qquad
\scalebox{0.9}{
\begin{tikzpicture}[baseline=-63pt]
\fill[gray!30] (0,0)--(-.4895460546, 2.144841407)--(.4895460546, 2.1448414075);
\draw[line width=\st] (0,0) circle(1.5);
\draw[line width=\st,dashed] (-.4895460546, 2.144841407) -- (0,0);
\draw[line width=\st,dashed] (.4895460546, 2.144841407) -- (0,0);
\draw[line width=\st] (-2.252422170, 1.084709348)node[left=1pt]{$\H_\mu$}--(2.252422170, -1.084709348);
\draw[fill] (-1.351453302, .6508256088) circle(0.06)node[left=6pt,above=7pt]{$M_1$};
\draw[fill] (1.351453302, -.6508256088) circle(0.06)node[right=5pt,below=5pt]{$M_2$};
\draw[fill] (0, 1.5) circle(0.06)node[above=2pt]{$O$};

\draw[fill] (-.9352347012, -1.172747225) circle(0.06)node[right=3pt,below=5pt]{$\Lm_2$};

\draw[fill] (.9352347012, 1.172747225) circle(0.06)node[right=7pt]{$\Lm_1$};
\draw[line width=\st] (-1.558724502, -1.954578708)--(1.558724502, 1.954578708)node[right=1pt]{$\H_\lm$};

\draw[fill] (-1.172747224, -.9352347024) circle(0.06)node[left=4pt]{$O'$};

\draw[-{Stealth[length=\sl]},line width=\tt,red] (0,1.6) arc (90:334.2857143:1.6);
\draw[-{Stealth[length=\sl]},line width=\tt,red] (1.252346727, -.6030983969) arc (334.2857143:218.5714286:1.39);
\end{tikzpicture}}
\end{center}

\vspace{-10pt}

\noindent
Thus $g$ is a path having natural parametrization from $[x_0,x_5]$ to $S_{n/\pi}^1$, which is an alcove walk
such that $g(x_m)=f(y_m)$ for any $m=0,\ldots,4$
and $q$ is an increasing map from
$g^{-1}(P_n)=\{1,2,\ldots,x_4\}$ to $f^{-1}(P_n)$ such that
$q(x_m)=y_m$ for any $m=1,\ldots,4$. We define the function $p:\{1,2,\ldots,x_4\}\to\{1,\ldots,N\}$ by
\begin{equation}\label{eq:p}
q(z)=t_{p(z)}.
\end{equation}
This function is obviously increasing and $p(x_1)=i$, $p(x_2)=j$, $p(x_3)=k$, $p(x_4)=l$.

\begin{lemma}\label{lemma:cycb6:1}
Let $\u{\Delta}$ be the gallery generated by $g$. Then $(\u{\delta},\u{\gamma})$ is a $p$-pair
of positive cosign.
\end{lemma}
\begin{proof}
Let $z=1,\ldots,x_4$. By property~\ref{corollary:cycl:2:p:i} of Corollary~\ref{corollary:cycl:2}, we get $g(z)=f(t_{p(z)})$.
We denote this point by $X$. By Definition~\ref{definition:cycb6:1},
$X\in\H_{\u{\delta}^{\to z}}\cap\H_{\u{\gamma}^{\to p(z)}}$.
Thus $\H_{\u{\delta}^{\to z}}=\H_{\u{\gamma}^{\to p(z)}}$, whence
\begin{equation}\label{eq:cycb6:7}
\u{\delta}^{\to z}=\pm\u{\gamma}^{\to p(z)}.
\end{equation}
This proves that $(\u{\delta},\u{\gamma})$ is a $p$-pair.
It remains to determine the sign in~(\ref{eq:cycb6:7}). Let $m=1,\ldots,4$ be the index such that $x_{m-1}<z\le x_m$.
Then we get $z\in g_m^{-1}(P_n)\setminus\{x_{m-1}\}$.
Thus by Corollary~\ref{corollary:cycl:2}\ref{corollary:cycl:2:p:ii},
we get $Y_{m-1}f(t')XY'_m$ as
$t'\to t_{p(z)}^-$.
Applying Lemma~\ref{lemma:circ:5} for $L=\H_{\u{\gamma}^{\to p(z)}}$, $\Phi=\arc{Y_{m-1}Y_m}$, $A=Y_{m-1}$, $D=Y'_m$, $Z=f(t')$ and $X=X$,
we get $f(t')\sim_L Y_{m-1}$.
If $Y_{m-1}\sim_L O$, then $f(t')\sim_L O$, whence $\mathscr C_{p(z)}\sim_L O$ and $\u{\gamma}^{\to p(z)}<0$ by Lemma~\ref{lemma:to}.
Similarly, if $Y_{m-1}\notsim_L O$, then $\mathscr C_{p(z)}\notsim_L O$ and $\u{\gamma}^{\to p(z)}>0$.

By Lemma~\ref{lemma:circ:2}, we get $Y_{m-1} g(z') X Y'_m$ for any $z'\in\R$
such that $x_{m-1}<z'<z\le x_m$. Applying Lemma~\ref{lemma:circ:5} with the same set of parameters
as just before except $Z=g(z')$, we get $g(z')\sim_L Y_{m-1}$. Arguing as in the previous paragraph,
we get $\u{\delta}^{\to z}<0$ if $Y_{m-1}\sim_L O$ and $\u{\delta}^{\to z}>0$ if $Y_{m-1}\notsim_L O$.
It follows from~(\ref{eq:cycb6:7}) that
$
\u{\delta}^{\to z}=\u{\gamma}^{\to p(z)}
$
in both cases. Thus the cosign of the $p$-pair$(\u{\delta},\u{\gamma})$ is positive.
\end{proof}

\begin{theorem}\label{theorem:3} Let $\ord(AB)<\infty$, $\u{s}$ be a finite sequence in $A$ and $B$, $\u{\gamma}\subset\u{s}$ and
$(i,k)$ and $(j,l)$ be special $\mu$- and $\lm$-pairs respectively for $\u{\gamma}$, where
$i<j<k<l$. Let the index $u\in\{1,2\}$, the path $g$, the numbers $x_0,\ldots,x_5$ and map $p$ be chosen as above and
$\u{\Delta}$ be the gallery generated by $g$.
Let us choose $\c\in\{\a,\b\}$ so that $g(1)\in\H_\c$ and define
$\u{t}=(\c,\bar\c,\ldots)_{x_4}$.

Then there exists a map $\phi:\SubExpr(\u{t})\to\SubExpr(\u{s})$ such that $\phi(\u{\delta})=\u{\gamma}$ and
the pair $(p,\phi)$ is a morphism $\u{t}\to\u{s}$ of positive cosing in $\Expr_\D$.

If $u=1$,
then $\u{\gamma}$ is the maximal vertex of any cycle
$\phi(\Cyc^1_\c(x,x_4))$, where $x_1\le x<x_2$, and the only edges of the sum of these cycles
with endpoint $\u{\gamma}$ are $\{\u{\gamma},\f_{i,k}\u{\gamma}\}$ and $\{\u{\gamma},\f_{j,l}\u{\gamma}\}$.\linebreak
If $u=2$, then $\u{\gamma}$ is the maximal vertex of $\phi(\Cyc^2_\c(x_1,x_2))$
and the only edges of this cycle with endpoint $\u{\gamma}$ are $\{\u{\gamma},\f_{i,k}\u{\gamma}\}$ and $\{\u{\gamma},\f_{j,l}\u{\gamma}\}$.
\end{theorem}
\begin{proof}
Existence of $\phi$ follows from Lemmas~\ref{lemma:cycb6:1} and~\ref{lemma:cat:1}.
Suppose that 
$u=1$.
Then $\u{\delta}=\u{1}^{x_4}$. For each $x$ such that $x_1\le x<x_2$,
the maximal vertex of the cycle $\Cyc^1_\c(x,x_4)$ is $\u{\delta}$ and the only edges with endpoint $\u{\delta}$ are
$\{\u{\delta},\f_{x,x+n}\u{\delta}\}$ and $\{\u{\delta},\f_{x+1,x+1+n}\u{\delta}\}$.
To understand why these facts are true, note that the paths $g$ and thus the galleries
$\u{\Delta}$ constructed in this section and Section~\ref{cycles_first} are the same (where $y=x_4$).
Taking the sum of these cycles $\sum_{x=x_1}^{x_2-1}\Cyc^1_\c(x,x_4)$, we get an even subgraph whose
only edges with endpoint $\u{\delta}$ are
$\{\u{\delta},\f_{x_1,x_1+n}\u{\delta}\}$ and $\{\u{\delta},\f_{x_2,x_2+n}\u{\delta}\}$.
It remains to apply $\phi$, which preserves the order by Lemma~\ref{lemma:21}.

In the remaining case $u=2$, the arguments are similar with the exception that we deal with
only one cycle $\Cyc^2_\c(x_1,x_2)$.
\end{proof}

\subsection{Infinite case}\label{Infinite_case} Here we consider the case $n=\infty$. Let $f:[0,1]\to\Sigma_0(\D,\{A,B\})$
be an alcove walk generating the gallery $\u{\Gamma}$.
Let
$$
\H_\mu\cap\Sigma_0(\D,\{A,B\})=\{M\},\quad\H_\lm\cap\Sigma_0(\D,\{A,B\})=\{\Lm\}$$
$$
f^{-1}(\Sigma_1(\D,\{A,B\}))=\{t_1<\cdots<t_N\}.
$$
As $f(t_i),f(t_k)\in\H_\mu$, we get $f(t_i)=f(t_k)=M$. Similarly, $f(t_j)=f(t_l)=\Lm$.
By Lemma~\ref{lemma:cycl:2}, the hyperplane $\H_\mu$ separates the fundamental chamber $\mathscr F$ and
$$
f((t_i,t_{i+1})\cup(t_{i+1},t_{i+2})\cup\cdots(t_{k-1},t_k))\subset \mathscr C_{i+1}\cup \mathscr C_{i+2}\cup\cdots\cup \mathscr C_k.
$$
As $t_j\in(t_i,t_k)$ and $f$ is continuous, this hyperplane separates $\mathscr F$ and $f(t_j)=\Lm$.
Therefore, $\H_\lm$ does not separate $\mathscr F$ and $M$.

Again by Lemma~\ref{lemma:cycl:2}, the hyperplane $\H_\lm$ separates $\mathscr F$ and
$$
f((t_j,t_{j+1})\cup(t_{j+1},t_{j+2})\cup\cdots\cup(t_{l-1},t_l))\subset \mathscr C_{j+1}\cup \mathscr C_{j+2}\cup\cdots\cup \mathscr C_l.
$$
As $t_k\in(t_j,t_l)$ and $f$ is continuous, this hyperplane separates $\mathscr F$ and $f(t_k)=M$.
It is a contradiction and this case is impossible.



%
%
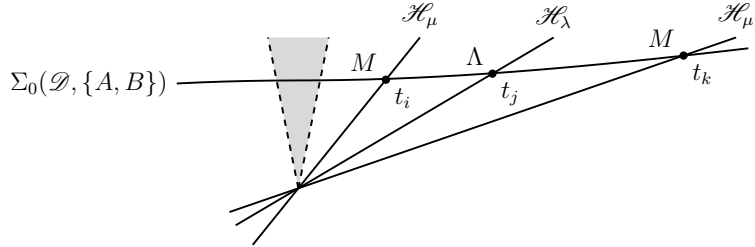
\def\xii{1.05}

\begin{center}
\scalebox{0.8}{
\begin{tikzpicture}[baseline=-33pt]
\fill[gray!30] (0,-2)--(-0.5,0.5)--(0.5,0.5)--(0,-2);


\draw[domain=7.4:-2, variable=\z, line width=\st] plot ({\z}, {((1+\xii)*((\xii-1)*\z^2+2))^(1/2)/(1+\xii)-1.2})node[left]{$\Sigma_0(\D,\{A,B\})$};

\draw[line width=\st] (-.7496340574, -2.937042572)--(2,0.5)node[above]{$\H_\mu$};

\draw[line width=\st,dashed](0,-2)--(-0.5,0.5);
\draw[line width=\st,dashed](0,-2)--(0.5,0.5);

\draw[fill] (1.44, -.2) circle(0.06)node[above left]{$M$}node[below right]{$t_i$};
\draw[line width=\st] (-1.031151461, -2.613780632)--(4.2,0.5)node[above]{$\H_\lm$};
\draw[fill] (3.1920, -.10) circle(0.06)node[above left]{$\Lm$}node[below right]{$t_j$};
\draw[line width=\st] (-1.133608396, -2.393614026)--(7.2,0.5)node[above]{$\H_\mu$};
\draw[fill] (6.336, .2) circle(0.06)node[above left]{$M$}node[below right]{$t_k$};
\end{tikzpicture}}
\end{center}

\section{Proof of Theorem~\ref{theorem:main}} In this section, we as usual assume that $(W,S)$ is the Coxeter system
associated to a based root system $(\Phi,\Pi)$ in a 
real space $V$ such that $\Pi$ is a basis of $V$.
Replacing the initial Coxeter group with a standard parabolic subgroup, we assume that $S$
consists only of the entries of the expression $\u{s}$ from the formulation of Theorem~\ref{theorem:main}.

\subsection{Triangles and squares}\label{Triangles_and_squares} Let $\u{\gamma}$ be a subexpression of an expression $\u{s}$ in $S$
with target~$w$. For indices $i,j,k$ such that $1\le i<j<k\le|\u{s}|$, we will consider the following three types of
subgraphs of $\Sub(\u{s},w)$ with maximal vertex $\u{\gamma}$:

{\renewcommand{\labelenumi}{{\rm\theenumi}}
\renewcommand{\theenumi}{{\rm($\Tr^\arabic{enumi})$}}
\begin{enumerate}
\item\label{Tr:1} If $\u{\gamma}^{\to i}=\pm\u{\gamma}^{\to j}$ and $\u{\gamma}^{\to j}=\u{\gamma}^{\to k}>0$, then
$\Tr^1_{i,j,k}(\u{\gamma})$ is the subgraph
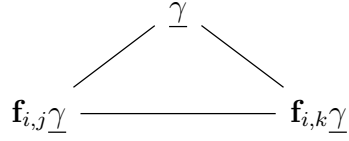
$$
\begin{tikzcd}[row sep=18pt]
&\u{\gamma}\arrow[no head]{ld}\arrow[no head]{dr}\\
\f_{i,j}\u{\gamma}\arrow[no head]{rr}&&\f_{i,k}\u{\gamma}\\
\end{tikzcd}
$$

\vspace{-20pt}

\item\label{Tr:2} If $\pm\u{\gamma}^{\to i}=\pm\u{\gamma}^{\to j}=\u{\gamma}^{\to k}>0$, then
$\Tr^2_{i,j,k}(\u{\gamma})$ is the subgraph
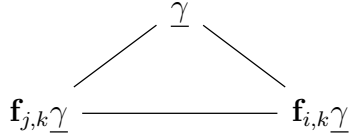
$$
\begin{tikzcd}[row sep=18pt]
&\u{\gamma}\arrow[no head]{rd}\arrow[no head]{dl}&\\
\f_{j,k}\u{\gamma}\arrow[no head]{rr}&&\f_{i,k}\u{\gamma}
\end{tikzcd}
$$

\item\label{Tr:3} If $\u{\gamma}^{\to i}=\pm\u{\gamma}^{\to j}$ and $\u{\gamma}^{\to j}=\u{\gamma}^{\to k}>0$, then
$\Tr^3_{i,j,k}(\u{\gamma})$ is the subgraph
$$
\begin{tikzcd}[row sep=18pt]
&\u{\gamma}\arrow[no head]{ld}\arrow[no head]{rd}&\\
\f_{i,j}\u{\gamma}\arrow[no head]{rr}&&\f_{j,k}\u{\gamma}
\end{tikzcd}
$$
\end{enumerate}


For indices $i,j,k,l$ such that
$1\le i<j<k<l\le|\u{s}|$,
we will consider the following two types of subgraphs of $\Sub(\u{s},w)$ with maximal vertex $\u{\gamma}$:

{\renewcommand{\labelenumi}{{\rm\theenumi}}
\renewcommand{\theenumi}{{\rm($\Sq^\arabic{enumi})$}}
\begin{enumerate}
\item\label{Sq:1} If $\pm\u{\gamma}^{\to i}=\u{\gamma}^{\to j}>0$ and $\pm\u{\gamma}^{\to k}=\u{\gamma}^{\to l}>0$, then
$\Sq^1_{i,j,k,l}(\u{\gamma})$ is the subgraph
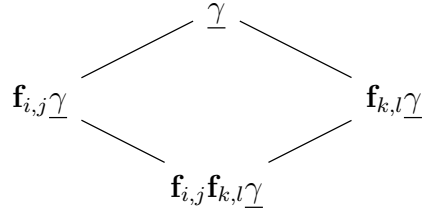
$$
\begin{tikzcd}[row sep=12pt]
&\u{\gamma}\arrow[no head]{ld}\arrow[no head]{dr}\\
\f_{i,j}\u{\gamma}\arrow[no head]{rd}&&\f_{k,l}\u{\gamma}\arrow[no head]{ld}\\
&\f_{i,j}\f_{k,l}\u{\gamma}
\end{tikzcd}
$$
\item\label{Sq:2} If $\pm\u{\gamma}^{\to i}=\u{\gamma}^{\to l}>0$ and $\pm\u{\gamma}^{\to j}=\u{\gamma}^{\to k}>0$, then
$\Sq^2_{i,j,k,l}(\u{\gamma})$ is the subgraph
$$
\begin{tikzcd}[row sep=12pt]
&\u{\gamma}\arrow[no head]{ld}\arrow[no head]{dr}\\
\f_{j,k}\u{\gamma}\arrow[no head]{rd}&&\f_{i,l}\u{\gamma}\arrow[no head]{ld}\\
&\f_{i,l}\f_{j,k}\u{\gamma}
\end{tikzcd}
$$
\end{enumerate}
}

\subsection{Moving an edge}
Before formulating the next result, let us remember the definitions of the color of an edge (Section~\ref{Expressions_and_subexpressions})
and of special pairs (Definition~\ref{definition:cycb6:3}) and how edges are moved (Definition~\ref{definition:cycb6:2}).

\begin{lemma}\label{lemma:moving}
Any edge $\{\u{\gamma},\u{\delta}\}$ of $\Sub(\u{s},w)$ of color $\alpha$ such that $\u{\gamma}>\u{\delta}$
can be $\u{\gamma}$-moved to an edge $\{\u{\gamma},\f_{i,j}\u{\gamma}\}$ of the same color
for some special $\alpha$-pair $(i,j)$ for $\u{\gamma}$
by a sequence of subgraphs of types $\Tr^2$, $\Tr^3$, $\Sq^1$ with maximal vertex $\u{\gamma}$. 
\end{lemma}
\begin{proof}
Let $\u{\delta}=\f_{p,q}\u{\gamma}$. In view of Proposition~\ref{lemma:cycl:3}, we can apply induction
with respect to the lexicographic order on the pair $(q,q-p)$, called the {\it parameter} only in this proof.
As $\u{\gamma}>\u{\delta}$ and these subexpressions differ maximally at $q$,
we get $\u{\gamma}^{\to q}=\alpha>0$.
Thus $\u{\gamma}^{\to p}=\pm\alpha$ by Proposition~\ref{proposition:applicable}.
We only need to consider the case where $(p,q)$ is not a special $\alpha$-pair for $\u{\gamma}$.
This can happen in the following cases (see Definition~\ref{definition:cycb6:3}).

{\it Case 1: $\u{\gamma}^{\to p}=\alpha$.} By Corollary~\ref{corolary:0}, there exists an index $r<p$
such that $\u{\gamma}^{\to r}=-\alpha$. We get
$$
\{\u{\gamma},\u{\delta}\}+\Tr^3_{r,p,q}(\u{\gamma})=\{\u{\gamma},\f_{r,p}\u{\gamma}\}+\{\f_{r,p}\u{\gamma},\u{\delta}\}.
$$
Thus $\{\u{\gamma},\u{\delta}\}$ has been $\u{\gamma}$-moved to $\{\u{\gamma},\f_{r,p}\u{\gamma}\}$ with parameter $(p,p-r)<(q,q-p)$.

{\it Case 2: there exists an index $k<p$ with $\u{\gamma}^{\to k}=\alpha$.}
By Corollary~\ref{corolary:0}, there exists an index $r<k$ such that $\u{\gamma}^{\to r}=-\alpha$.
We get
$$
\{\u{\gamma},\u{\delta}\}+\Sq^1_{r,k,p,q}(\u{\gamma})=\{\u{\gamma},\f_{r,k}\u{\gamma}\}+\{\u{\delta},\f_{r,k}\u{\delta}\}+\{\f_{r,k}\u{\gamma},\f_{r,k}\u{\delta}\}.
$$
Thus $\{\u{\gamma},\u{\delta}\}$ has been $\u{\gamma}$-moved to $\{\u{\gamma},\f_{r,k}\u{\gamma}\}$
with parameter $(k,r-k)<(q,q-p)$.

{\it Case 3: there exists an index $k$ such that $p<k<q$ and $\u{\delta}^{\to k}=-\alpha$.}
We get
$$
\{\u{\gamma},\u{\delta}\}+\Tr^2_{p,k,q}(\u{\gamma})=\{\u{\gamma},\f_{k,q}\u{\gamma}\}+\{\u{\delta},\f_{k,q}\u{\gamma}\}.
$$
Thus $\{\u{\gamma},\u{\delta}\}$ has been $\u{\gamma}$-moved to $\{\u{\gamma},\f_{k,q}\u{\gamma}\}$
with parameter $(q,q-k)<(q,q-p)$.
\end{proof}

\subsection{Generating cycles}\label{Generating_cycles} Now we are ready to prove Theorem~\ref{theorem:main}.
We denote by $\FDR(W,S)$ the set of all finite dihedral reflection subgroups of $W$.
For each of this subgroups $\D$, 
we consider the following disjoint subsets:
$$
K^1_\D=\{(x,y)\in\Z^2\suchthat 1\le x<n,x+n+1\le y\le 2n\},\;\;
K^2_\D=\{(x,y)\in\Z^2\suchthat 1\le x<y\le n\},
$$
where $n=|\D|/2$, and denote $K_\D=K^1_\D\cup K^2_\D$.

Let us fix an expression $\u{s}$ in $S$.
Let $\D\in\FDR(W,S)$ and $(x,y)\in K_\D$.
We set $u=1$ and $z=y$ if $(x,y)\in K^1_\D$ and $u=2$ and $z={2x-y+2n}$ if $(x,y)\in K^2_\D$.
We denote by $\mathbf{Cyc}_\D(\u{s},x,y)$
the set of cycles $\phi(\Cyc^u_\c(x,y))$ for all elements $\c$ of the canonical simple system for $\D$
and all morphisms $(p,\phi):(\c,\bar\c,\ldots)_z\to\u{s}$ of positive cosign
in the category $\Expr_\D(W,S)$.
Finally, we define
$$
\mathbf{Cyc}(\u{s})=\bigcup_{\D\in\FDR(W,S)}\;\bigcup_{(x,y)\in K_\D}\;\mathbf{Cyc}_\D(\u{s},x,y).
$$

\begin{remark}\label{rem:3}
\rm
It follows from the constructions of cycles in Sections~\ref{cycles_first} and~\ref{cycles_second} that
all cycles in $\mathbf{Cyc}_\D(\u{s},x,y)$
have lengths $|\D|/2+2$.
\end{remark}

\noindent
Finally, let $\mathbf{Cyc}(\u{s},w)$ be the subset of $\mathbf{Cyc}(\u{s})$
consisting of cycles contained in $\bm{\mathfrak S}(\u{s},w)$.

\begin{definition}
Let $\u{\gamma}$ be a subexpression of $\u{s}$ with target $w$.
A subgraph $F$ of $\bm{\mathfrak S}(\u{s},w)$
is said to be supported on subexpressions $\le\u{\gamma}$ (respectively $<\u{\gamma}$)
if $\u{\delta}\le\u{\gamma}$ (respectively $\u{\delta}<\u{\gamma}$) for any endpoint $\u{\delta}$
of any edge belonging to $F$.
\end{definition}

If $\u{\epsilon}$ is the maximal vertex of $\bm{\mathfrak S}(\u{s},w)$, then obviously any subgraph of
$\bm{\mathfrak S}(\u{s},w)$ is supported on subexpressions $\le\u{\epsilon}$.
Now we are ready to prove our main result.

\begin{theorem}\label{theorem:4}
Any even subgraph of $\bm{\mathfrak S}(\u{s},w)$ 
supported on subexpressions $\le\u{\gamma}$
is a sum of some cycles of $\mathbf{Cyc}(\u{s},w)$ and cycles of types $\Tr^1$, $\Tr^2$, $\Tr^3$, $\Sq^1$, $\Sq^2$
all supported on subexpressions $\le\u{\gamma}$.
\end{theorem}
\begin{proof}
Let $F$ be the subgraph under consideration.
We will apply induction on $\u{\gamma}$.
%
By Veblen's theorem~\cite{Veblen}, we may assume that $F$
is a cycle. Due to the inductive hypothesis, we may also assume that $\u{\gamma}$ is the maximal vertex of $F$,
that is, the maximal vertex of $\bm{\mathfrak S}(\u{s},w)$
incident with an edge belonging to $F$.

By Lemma~\ref{lemma:moving}, we can assume that the only edges belonging to $F$ with endpoint $\u{\gamma}$
are $\{\u{\gamma},\f_{i,k}\u{\gamma}\}$ and $\{\u{\gamma},\f_{j,l}\u{\gamma}\}$,
where $(i,k)$ and $(j,l)$ are special $\mu$- and $\lm$-pairs for $\u{\gamma}$.
Note that $i<k$, $j<l$ and $(i,k)\ne(j,l)$. We will consider several cases and
indicate how to apply induction in each of them.

{\it Case 1:} $\{i,k\}\cap\{j,l\}\ne\emptyset$. In this case $\mu=\lm$.
Permuting the pairs $(i,k)$ and $(j,l)$ if necessary, we get the following cases.

{\it Case 1.a:} $k=j$. In that case, $F+\Tr^3_{i,j,l}(\u{\gamma})$ is 
supported on subexpressions $<\u{\gamma}$.

{\it Case 1.b:} $i=j$ and $k<l$. In that case, $F+\Tr^1_{i,k,l}(\u{\gamma})$ is 
supported on subexpressions $<\u{\gamma}$.

{\it Case 1.c:} $i<j$ and $k=l$. In that case, $F+\Tr^2_{i,j,l}(\u{\gamma})$ is 
supported on subexpressions $<\u{\gamma}$.

{\it Case 2:} $\{i,k\}\cap\{j,l\}=\emptyset$.
Permuting the pairs $(i,k)$ and $(j,l)$, if necessary, we get the following cases.

{\it Case 2.a:} $k<j$. In that case, $F+\Sq^1_{i,k,j,l}(\u{\gamma})$ is 
supported on subexpressions $<\u{\gamma}$.

{\it Case 2.b:} $i<j$ and $l<k$. In that case, $F+\Sq^2_{i,j,l,k}(\u{\gamma})$ is 
supported on subexpressions~${<}\u{\gamma}$.

{\it Case 2.c:} $i<j<k<l$. It was explained at the beginning of Section~\ref{Reducing_a_special_vertex} that $\lm\ne\mu$.
Thus we get a two-dimensional real space $\mathscr V=\R\lm\oplus\R\mu$.
Let $\D=\<\r_\lm,\r_\mu\>$. This a reflection subgroup of order $2n$, where $n=\ord(\r_\lm\r_\mu)\ge2$.

Let $\{A,B\}$ be the canonical simple system for $\D$. Applying the procedure of projection described in Section~\ref{Projection},
we get a gallery $\u{\Delta}=\pr_\D(\u{\Gamma})$ in $\mathscr V^*$ and
an increasing bijection $p:\{1,\ldots,|\u{\Delta}|\}\to\u{\Gamma}(\D)$.
We assume that $\u{\delta}\subset\u{t}$ for the corresponding expression $t$ in $A$ and $B$
and denote
$$
i'=p^{-1}(i), \quad j'=p^{-1}(j),\quad k'=p^{-1}(k), \quad l'=p^{-1}(l).
$$
Note that these indices are well-defined. Indeed, the $i$th wall of $\u{\Gamma}$
is $H_{\u{\gamma}^{\to i}}=H_{\lm}=H_{\t_{\lm}}$ and $\r_{\lm}\in\D$.
Thus $i\in\u{\Gamma}(\D)=\im p$. Similarly $j,k,l\in\im p$.

By Theorem~\ref{lemma:15},
there exists a morphism $(p,\phi):\u{t}\to\u{s}$ of positive cosign in $\Expr_\D(W,S)$ for
some map $\phi:\SubEx(\u{t})\to\SubEx(\u{s})$ such that $\phi(\u{\delta})=\u{\gamma}$.
We claim that $(i',k')$ and $(j',l')$ are special
$\mu$- and $\lm$-pairs respectively for $\u{\delta}$.
Indeed, as $(\u{\delta},\u{\gamma})$ is a $p$-pair of positive cosign, we get
$$
\u{\delta}^{\to i'}=\u{\gamma}^{\to p(i')}=\u{\gamma}^{\to i}=-\mu,
$$
Similarly $\u{\delta}^{\to j'}=-\lm$,
$\u{\delta}^{\to k'}=\mu$, $\u{\delta}^{\to l'}=\lm$.
Thus we have checked part~\ref{definition:cycb6:3:p:i} of Definition~\ref{definition:cycb6:3} for both pairs.
Let us check part~\ref{definition:cycb6:3:p:ii} for $(i',k')$. Suppose that there exists $h<i'$
such that $\u{\delta}^{\to h}=\mu$. Calculating as just above, we get
$$
\u{\gamma}^{\to p(h)}=\u{\delta}^{\to h}=\mu.
$$
As $p$ is increasing, $p(h)<p(i')=i$. This violates property~\ref{definition:cycb6:3:p:ii}
for the special pair $(i,k)$. The property~\ref{definition:cycb6:3:p:iii} can be proved similarly.
The proof for $(j',l')$ is also similar.

By Theorem~\ref{theorem:3}, there exists $\c\in\{\a,\b\}$, an index $u\in\{1,2\}$
and a morphism $(q,\psi):\u{r}\to\u{t}$ of positive cosign in the category $\Seq_\D$
such that the only edges with endpoint $\u{\delta}$ of a certain sum of cycles $\psi(\Cyc^u_\c(x,y))$
all supported on subexpressions $\le\u{\delta}$ are $\{\u{\delta},\f_{i',k'}\u{\delta}\}$
and $\{\u{\delta},\f_{j',l'}\u{\delta}\}$. 
Considering the composition
$$
\begin{tikzcd}
\u{r}\arrow{r}{(\psi,q)}&\u{t}\arrow{r}{(\phi,p)}&\u{s},
\end{tikzcd}
$$
which is a morphism in the category $\Seq_\D$ of positive cosign by Proposition~\ref{proposition:8},
we obtain a sum of cycles belonging to $\mathbf{Cyc}(\u{s},w)$ all supported on subexpressions $\le\u{\gamma}$
whose only edges with endpoint $\u{\gamma}$ are $\{\u{\gamma},\f_{i,k}\u{\gamma}\}$
and $\{\u{\gamma},\f_{j,l}\u{\gamma}\}$. Adding this sum to $F$, we obtain an even
subgraph supported on subexpressions $<\u{\gamma}$.
\end{proof}

\begin{proof}[Proof\! of\! Theorem\!~\ref{theorem:main}]\!\!
The result follows from Theorem~\ref{theorem:4}, Remark~\ref{rem:3} and Corollary~\ref{corollary:6}.
\end{proof}

\def\sep{\\[-6.8pt]}


\begin{thebibliography}{00}
\bibitem[BB]{BB} A.\;Bjorner, F.\;Brenti, Combinatorics of Coxeter groups, Graduate Texts in Mathematics, {\bf 231} (2005).\sep
\bibitem[Bo]{Bourbaki} N.\;Bourbaki, Elements of mathematics, Lie groups and Lie algebras: chapters 4-6, {\it Springer}, 2002.\sep
\bibitem[BD]{BD} C.\;Bonnaf\'e, M.\;J.\;Dyer, Semidirect product decomposition of Coxeter groups, {\it Comm. Algebra}, 38(4):1549--1574, 2010.\sep
\bibitem[BW]{BW} T.\;Braden, G.\;Williamson, Modular intersection cohomology complexes on flag varieties, Math. Z., {\bf 272} (2012), n. 3-4,  697--727.\sep
\bibitem[CC]{CC} C.\;Contou-Carrère, Buildings and Schubert Schemes, 1st Edition, Taylor\&Francis, pp. 462, (2016).\sep
\bibitem[De1]{Deodhar1} V.\;V.\;Deodhar, A note on subgroups generated by reflections in Coxeter groups, {\it Arch. Math.} {\bf 53} (1989), 543--546. \sep
\bibitem[De2]{Deodhar2} V.\;V.\;Deodhar, A combinatorial setting for questions in Kazhdan-Lusztig theory, {\it Geom Dedicata} {\bf 36}, 95--119 (1990).\sep
\bibitem[Di]{Diestel} R.\,Diestel, Graph Theory, {\it Graduate Texts in Mathematics}, vol. 173, Fith edition, Springer (2017).\sep
\bibitem[Dy1]{Dyer1} M.\;Dyer, Reflection subgroups of Coxeter systems, {\it J. Algebra}, 135(1):57--73, 1990.\sep
\bibitem[Dy2]{Dyer2} M.\;Dyer, Imaginary cone and reflection subgroups of Coxeter groups, {\it Dissertationes Mathematicae}, {\bf 545} (2019), 1--117.\sep
\bibitem[EW]{EW} B.\,Elias, G.\,Williamson, Soergel calculus, {\it Represent. Theory}, {\bf 20}, 295--374 (2016).\sep
\bibitem[G]{G} S.\,Gaussent, The fibre of the Bott-Samelson resolution, {\it Indag. Math.}, 12(4), 453--468 (2001).\sep
\bibitem[GSch]{GSch} M.\,Graeber, P.\,Schwer, Shadows in Coxeter groups, {\it Ann. Comb.}, {\bf 24}(1), 119--147 (2020).\sep
\bibitem[GY]{GY} J.L.\,Gross, J.\,Yellen, Graph Theory and Its Applications (2nd ed.), {\it Chapman and Hall/CRC} (2005).\sep
\bibitem[H]{Humphreys} J.\;E.\;Humphreys, Reflection groups and Coxeter groups, {\it Cambridge Studies in Advanced Mathematics}, v. 29, Cambridge University Press, Cambridge, 1990.\sep
\bibitem[K]{Kumar} S. Kumar, The nil-Hecke ring and singularities of Schubert varieties, {\it Invent. Math.}, {\bf 123} (1996), 471--506.\sep
\bibitem[M]{M} H.\,Matsumoto. G\'en\'erateurs et relations des groupes de Weyl g\'en\'eralis\'es, {\it C. R. Acad. Sci. Paris}, 258:34193422, 1964.\sep
\bibitem[MST]{MST} E.\,Mili\'cevi\'c, P.\,Schwer, A.\,Thomas, Dimensions of Affine Deligne-Lusztig Varieties: A New Approach via Labeled Folded Alcove Walks and Root Operators, {\it Mem. Amer. Math. Soc.}, vol. 261 (1260), (2019), v+101.\sep
\bibitem[Ra]{Ram} A.\;Ram, Alcove walks, Hecke algebras, spherical functions, crystals and column strict tableaux, {\it Pure Appl. Math. Q.}, 2(4, Special Issue: In honor of Robert D. MacPherson. Part 2):963–1013, 2006.\sep
\bibitem[Ro]{Ronan} M.\,Ronan, Lectures on buildings, {\it University of Chicago Press}, Chicago, IL, 2009, Updated and revised.\sep
\bibitem[Ru]{Russell} B.\,Russell, Principles of Mathematics (1903).\sep
\bibitem[Shch1]{cat} V.\,Shchigolev, Categories of Bott-Samelson varieties, {\it Algebr. Represent. Theor.}, {\bf 23} (2), 349--391 (2020).\sep
\bibitem[Shch2]{subroot} V.\,Shchigolev, Galleries for Root Subsystems, {\it Algebr. Represent. Theor.} {\bf 27} (2024), 1537--1561.\sep
\bibitem[St]{Steinberg} R.\,Steinberg, Lectures on Chevalley Groups, New Haven, 1968.\sep
\bibitem[Ve]{Veblen} O.\;Veblen, An application of modular equations in analysis situs, Ann. Math. {\bf 14} (1912), 86--94.\sep
\bibitem[Vi]{Vinberg} E.B.\,Vinberg, Discrete linear groups generated by reflections, {\it Math. USSR-Izv.}, v.5, n.5 (1971), 1083--1119.\sep
\bibitem[Wa]{Wallis} W.D.\,Wallis, A Beginner's Guide to Graph Theory. Second Edition. Birkhäuser, 2007.\sep
\end{thebibliography}
\end{document}